\documentclass[reqno,11pt]{amsart}

\usepackage{amscd}
\usepackage{amsmath}
\usepackage{amssymb}
\usepackage[american]{babel}
\usepackage{bbm}
\usepackage{bookmark}
\usepackage{cmap} 
\usepackage{dsfont}
\usepackage{enumerate}
\usepackage{epigraph}
\usepackage[mathscr]{euscript}
\usepackage[myheadings]{fullpage}
\usepackage{graphicx}
\usepackage[geometry]{ifsym}
\usepackage{mathabx}
\usepackage{accents}
\usepackage{mathrsfs}
\usepackage{mathtools}
\usepackage{refcount}
\usepackage{setspace}
\usepackage{stmaryrd}
\usepackage{thinsp}
\usepackage{verbatim}
\usepackage[all,2cell]{xy} \UseTwocells
\usepackage{xr}
\usepackage[multiple]{footmisc}
\usepackage{hyperref}

\usepackage{tikz}
\usetikzlibrary{decorations.pathmorphing}
\usetikzlibrary{cd}

\numberwithin{equation}{subsection}

\newtheorem{thm}{Theorem}[subsubsection]
\newtheorem*{thm*}{Theorem}

\newtheorem{cor}[thm]{Corollary}
\newtheorem*{cor*}{Corollary}
\newtheorem{lem}[thm]{Lemma}

\newtheorem{prop}[thm]{Proposition}
\newtheorem{prop-const}[thm]{Proposition-Construction}

\newtheorem*{conjecture*}{Conjecture}

\newtheorem*{princ*}{Principle}

\newtheorem{introthm}{Theorem}

\theoremstyle{remark}
\newtheorem{rem}[thm]{Remark}
\newtheorem{example}[thm]{Example}

\newtheorem{defin}[thm]{Definition}

\newtheorem{notation}[thm]{Notation}

\newtheorem*{question*}{Question}

\newtheorem{convention}[thm]{Convention}

\newcounter{steps}[thm]
\newtheorem{step}[steps]{Step}

\newcommand{\presub}[1]{\prescript{}{#1}}

\newcommand{\xar}[1]{\xrightarrow{#1}}
\newcommand{\rar}[1]{\xar{#1}}
\newcommand{\isom}{\rar{\simeq}}

\newcommand{\into}{\hookrightarrow}

\newcommand{\eps}{\epsilon}

\newcommand{\Loc}{\on{Loc}}

\newcommand{\bDelta}{\mathbf{\Delta}}

\newcommand{\bA}{{\mathbb A}}
\newcommand{\bB}{{\mathbb B}}
\newcommand{\bC}{{\mathbb C}}
\newcommand{\bD}{{\mathbb D}}

\newcommand{\bG}{{\mathbb G}}
\newcommand{\bH}{{\mathbb H}}

\newcommand{\bL}{{\mathbb L}}
\newcommand{\bM}{{\mathbb M}}

\newcommand{\bO}{{\mathbb O}}

\newcommand{\bQ}{{\mathbb Q}}
\newcommand{\bR}{{\mathbb R}}

\newcommand{\bV}{{\mathbb V}}

\newcommand{\bZ}{{\mathbb Z}}

\newcommand{\sC}{{\EuScript C}}
\newcommand{\sD}{{\EuScript D}}
\newcommand{\sE}{{\EuScript E}}
\newcommand{\sF}{{\EuScript F}}
\newcommand{\sG}{{\EuScript G}}
\newcommand{\sH}{{\EuScript H}}
\newcommand{\sI}{{\EuScript I}}

\newcommand{\sK}{{\EuScript K}}
\newcommand{\sL}{{\EuScript L}}
\newcommand{\sM}{{\EuScript M}}
\newcommand{\sN}{{\EuScript N}}
\newcommand{\sO}{{\EuScript O}}
\newcommand{\sP}{{\EuScript P}}

\newcommand{\sS}{{\EuScript S}}

\newcommand{\sV}{{\EuScript V}}

\newcommand{\sY}{{\EuScript Y}}
\newcommand{\sZ}{{\EuScript Z}}

\newcommand{\fg}{{\mathfrak g}}

\newcommand{\fn}{{\mathfrak n}}

\newcommand{\fp}{{\mathfrak p}}

\newcommand{\on}{\operatorname}
\newcommand{\ol}[1]{\overline{#1}{}}
\newcommand{\ul}{\underline}

\newcommand{\mathendash}{\text{\textendash}}

\newcommand{\Ker}{\on{Ker}}

\newcommand{\End}{\on{End}}
\newcommand{\Hom}{\on{Hom}}

\newcommand{\Aut}{\on{Aut}}
\newcommand{\Spec}{\on{Spec}}

\newcommand{\id}{\on{id}}

\newcommand{\ad}{\on{ad}}

\newcommand{\Ad}{\on{Ad}}
\newcommand{\ev}{\on{ev}}

\newcommand{\ind}{\on{ind}}

\newcommand{\Rep}{\mathsf{Rep}}
\newcommand{\gr}{\on{gr}}

\newcommand{\act}{\on{act}}
\newcommand{\actson}{\curvearrowright}

\renewcommand{\dot}{\bullet}

\newcommand{\Fun}{\on{Fun}}

\newcommand{\vph}{\varphi}
\newcommand{\vareps}{\varepsilon}

\newcommand{\Vect}{\mathsf{Vect}}
\newcommand{\Fil}{\mathsf{Fil}\,}
\newcommand{\fil}{F}

\newcommand{\Gr}{\on{Gr}}
\newcommand{\Spr}{\on{Spr}}
\newcommand{\Whit}{\mathsf{Whit}}

\newcommand{\Bun}{\on{Bun}}

\newcommand{\LocSys}{\on{LocSys}}

\newcommand{\Op}{\on{Op}}
\newcommand{\IC}{\on{IC}}

\renewcommand{\mod}{\mathendash\mathsf{mod}}

\newcommand{\DGCat}{\mathsf{DGCat}}

\renewcommand{\lim}{\on{lim}}
\newcommand{\Ind}{\mathsf{Ind}}
\newcommand{\Pro}{\mathsf{Pro}}

\newcommand{\TwoHom}{\mathsf{Hom}}
\newcommand{\TwoEnd}{\mathsf{End}}

\newcommand{\heart}{\heartsuit}

\newcommand{\Oblv}{\on{Oblv}}
\newcommand{\Av}{\on{Av}}

\newcommand{\Perf}{\mathsf{Perf}}
\newcommand{\Coh}{\mathsf{Coh}}
\newcommand{\IndCoh}{\mathsf{IndCoh}}
\newcommand{\QCoh}{\mathsf{QCoh}}

\newcommand{\Lie}{\on{Lie}}

\newcommand{\Maps}{\sM aps}

\renewcommand{\o}[1]{\mathring{#1}}

\renewcommand{\subset}{\subseteq}
\renewcommand{\supset}{\supseteq}

\newcommand{\coeff}{\on{coeff}}
\newcommand{\Poinc}{\sP oinc}
\newcommand{\enh}{\on{enh}}

\newcommand{\Ran}{\on{Ran}}

\newcommand{\Shv}{\mathsf{Shv}}
\newcommand{\temp}{\on{temp}}
\newcommand{\cusp}{\on{cusp}}
\newcommand{\at}{\on{anti}\!\mathendash\!\on{temp}}
\newcommand{\xat}{x\mathendash\on{anti}\!\mathendash\!\on{temp}}
\newcommand{\xtemp}{x\mathendash\on{temp}}
\newcommand{\Nilp}{\sN ilp}
\newcommand{\Higgs}{\on{Higgs}}
\newcommand{\Kos}{\on{Kos}}
\newcommand{\constr}{\on{constr}}
\newcommand{\Irr}{\on{Irr}}

\newcommand{\sfe}{\mathsf{e}}
\newcommand{\Mir}{\on{Mir}}
\newcommand{\crit}{\on{crit}}
\newcommand{\KL}{\mathsf{KL}}

\newcommand{\irreg}{\on{irreg}}
\newcommand{\Girreg}{G\mathendash\irreg}

\newcommand{\dR}{\on{dR}}

\newcommand{\hol}{\on{hol}}
\renewcommand{\SS}{\on{SS}}
\newcommand{\CC}{\on{CC}}
\newcommand{\level}{\mathendash\on{lvl}}
\newcommand{\xlevel}{\on{lvl},x}

\makeatletter
\newcommand{\biggg}{\bBigg@{4}}
\newcommand{\Biggg}{\bBigg@{5}}
\makeatother

\date{\today}

\begin{document}

\frenchspacing

\setlength{\epigraphwidth}{0.4\textwidth}
\renewcommand{\epigraphsize}{\footnotesize}

\title{Non-vanishing of geometric Whittaker coefficients for reductive groups}

\author{Joakim F\ae rgeman}

\address{The University of Texas at Austin, 
Department of Mathematics, 
PMA 8.100, 2515 Speedway Stop C1200, 
Austin, TX 78712}

\email{joakim.faergeman@utexas.edu}

\author{Sam Raskin}

\address{The University of Texas at Austin, 
Department of Mathematics, 
PMA 8.100, 2515 Speedway Stop C1200, 
Austin, TX 78712}

\email{sraskin@math.utexas.edu}

\maketitle

\begin{abstract}

We prove that cuspidal automorphic $D$-modules have non-vanishing Whittaker
coefficients, generalizing known results
in the geometric Langlands program from $GL_n$ to general reductive groups.
The key tool is a microlocal interpretation of Whittaker coefficients.

We establish various exactness properties in the geometric Langlands
context that may be of independent interest. Specifically, we show
Hecke functors are $t$-exact on the category of tempered $D$-modules,
strengthening a classical result of Gaitsgory (with different hypotheses) 
for $GL_n$. We also show that Whittaker coefficient functors are $t$-exact
for sheaves with nilpotent singular support. An additional consequence
of our results is that the 
tempered, restricted geometric Langlands conjecture must
be $t$-exact. 

We apply our results to show that for suitably irreducible local systems, 
Whittaker-normalized Hecke eigensheaves
are perverse sheaves that are irreducible on each connected component
of $\Bun_G$.

 \end{abstract}

\setcounter{tocdepth}{2}
\tableofcontents

\section{Introduction}\label{s:intro}

\subsection{A mystery}

Although this paper is concerned with automorphic sheaves and not
with automorphic forms, 
our motivation comes from phenomena more easily witnessed in the latter context.
Therefore, we begin our story on the upper half plane.

\subsubsection{Background on modular forms}\label{sss:modular}

We start with some elementary recollections about modular forms
(say: holomorphic, of level $1$, without nebentypus, and of arbitrary even weight).

Recall that such a modular form is a holomorphic function $f$
on the upper half plane $\bH$, satisfying a family of functional equations.
Among these is the relation $f(\tau+1) = f(\tau)$ for $\tau \in \bH$. 
Moreover, as a function on the analytic punctured disc
function on: 
\[
\bH/\{\tau \sim \tau+1\} \overset{\tau \mapsto \exp(2\pi i \cdot \tau)}{\simeq} \{0<|q|<1\}
\]

\noindent there is a requirement that $f(q)$ extend to a holomorphic 
function over the puncture at $q = 0$.
It follows that $f(q)$ can be expanded as a power series:
\[
f(q) = \underset{n\geq 0}{\sum} a_n q^n.
\]

\noindent We remind that the coefficients $a_n$ in the $q$-expansion are the fundamental 
numerical invariants in the theory of modular forms.

\begin{rem}\label{r:cusp-fm}

Recall that $f$ is a \emph{cusp form} if $a_0 = 0$. 
It is manifest that for a non-zero cusp form $f$, there exists 
an $n \geq 1$ such that $a_n \neq 0$. More generally, non-constant
modular forms have some $a_n \neq 0$ for $n \geq 1$.

\end{rem}

\subsubsection{Ad\`elic interpretation}

It is not our purpose to review the construction of functions on 
ad\`elic groups from modular forms. However, we briefly state 
the outcomes.

We let $G = PGL_2$ and let $N = \bG_a$ denote the radical of its standard Borel
and $T = \bG_m$ be its standard Cartan.

\begin{itemize}

\item For $f$ as above, there is an associated function $\widetilde{f}$
on $G(\bA_{\bQ})$. 

\item The function $\widetilde{f}$ is invariant under the left action of 
$G(\bQ)$ and the right action of $G(\bA_{\bQ}^{\on{int}})$, where
$\bA_{\bQ}^{\on{int}} \coloneqq \widehat{\bZ} \subset \bA_{\bQ}$ is the
subring of integral ad\`eles.

\item The coefficient $a_0$ of $f$ is the \emph{constant term}:\footnote{More
conceptually, the \emph{constant term} of $f$ should be thought of as a function on 
$T(\bQ)\backslash T(\bA_{\bQ})/T(\bA_{\bQ}^{\on{int}}) \simeq \bR^{>0}$. It
happens to be constant in the holomorphic case; but this good fortune
does not occur for non-holomorphic modular forms, where one needs to consider
the constant term as a function on $\bR^{>0}$.}
\[
a_0 = \int_{N(\bQ)\backslash N(\bA_{\bQ})} \widetilde{f}(\gamma) d\gamma.
\]

\item The coefficient $a_n$ of $f$ is \emph{essentially} 
the \emph{Whittaker coefficient}:
\begin{equation}\label{eq:an}
a_n \sim \int_{N(\bQ)\backslash N(\bA_{\bQ})} \widetilde{f}(\gamma\alpha_n) \cdot 
\psi(-\gamma) d\gamma.
\end{equation}

\noindent 
Here $\alpha_n \in T(\bA_{\bQ}^{\on{fin}}) = \bA_{\bQ}^{\on{fin},\times}$
is $n$ (considered as a finite id\`ele, living in $PGL_2(\bA_{\bQ})$), 
and $\psi$ is the standard character of $\bA_{\bQ} = N(\bA)$ vanishing on $\bQ$
and $\widehat{\bZ}$. 
The symbol $\sim$ indicates that we have omitted a normalizing factor of
Archimedean nature;
see \cite{deligne-gl2} Proposition 2.5.3.2 or \cite{gelbart-automorphic}
Lemma 3.6 for precise assertions.

\end{itemize}

We refer to \cite{deligne-gl2} and \cite{gelbart-automorphic}  
for detailed derivations of the above dictionary.

\subsubsection{Generalization to reductive groups}

The notions from the previous section make sense for general 
reductive groups $G$ over global fields $F$.
There are automorphic forms, and they have constant terms indexed
by proper parabolic subgroups of $G$. A cusp form is one
with vanishing constant terms.

Similarly, there are Whittaker coefficients. When our automorphic forms
are unramified at finite places, these coefficients
are\footnote{This is not quite accurate; we can get away with it
only because of the simplicity of holomorphic modular forms. 
One needs to also allow the insertion
of elements of the group at infinite places in general; this is 
serious for Maass forms,
or automorphic forms for other groups. The easiest correction is to consider
Whittaker \emph{functions} on $G(\bA)$ rather than mere Whittaker \emph{coefficients}
(which are values of the Whittaker function at particular ad\`elic points; 
the above formula gives the values when we put the identity at the infinite place of $\bQ$). 

To simplify the exposition (particularly 
since we will eventually be concerned with everywhere unramified automorphic forms/
sheaves over function fields, we do not further emphasize this (important) point.
}
indexed by divisors on\footnote{The 
scare quotes indicate that for a global field of positive characteristic,
we take the corresponding smooth projective curve.} ``$\Spec$" of the ring
of integers of the global field, where these divisors take values in the
set $\check{\Lambda}^+$ of dominant coweights for $G$.\footnote{For instance,
for $G = PGL_2$, $\check{\Lambda}^+ = \bZ^{\geq 0}$. Note that 
a $\bZ^{\geq 0}$-valued (i.e., effective) divisor on $\Spec(\bZ)$ 
is equivalent data to a number $n \geq 1$: $D = \sum k_p [p]$ corresponds to 
$n = \prod p^{k_p}$.}

Then there is a natural question, attempting to generalize 
the naive observation from \S \ref{sss:modular}:

\begin{question*}

For a non-zero cusp form $f$ on $G(\bA_F)$, is some Whittaker
coefficient of $f$ non-zero?

\end{question*}

The easy argument from \S \ref{sss:modular} can readily be
adapted to the ad\`elic setting to give a positive answer for $(P)GL_2$. 
This argument adapts more generally to $GL_n$: 
this is related to the strong multiplicity one
theorem, and uses the special \emph{mirabolic} subgroup of $GL_n$.

However, the answer is \emph{no} for general $G$. It fails already for $SL_2$ 
for silly\footnote{Namely, the torus does not act transitively on the set of
characters.} reasons, and it fails for $GSp_4$ for serious reasons
(see \cite{howe-ps}). There is a conjecture
due to Shahidi that a 
tempered $L$-packet of automorphic representations has a unique
representative with non-zero Whittaker coefficients 
(see the discussion in the introduction to \cite{shahidi-arthur}). 

Roughly speaking, one can think of this failure as the source
of the many complications in the theory of automorphic forms
for reductive groups beyond $GL_n$, and for much of our ignorance
in the subject.

\subsubsection{A vague statement of the problem}

Broadly, the problem we consider is: \emph{where do Whittaker coefficients of 
automorphic forms come from for $G \neq GL_n$}?
In this paper, we
completely settle the corresponding problem in the setting of
\emph{global geometric Langlands} in characteristic $0$.

We describe the geometric context in more detail below.
But here, we state one long-standing motto: \emph{there
are no $L$-packets in geometric Langlands}. There are 
various ways of arriving at this conclusion, but one is simply 
that it is forced by the geometric Langlands conjectures, 
as reviewed below. 

So we may formulate the mystery stated above: why are there
no $L$-packets geometrically? This question has been the subject
of speculation in the geometric Langlands community for some time
now, with many possible answers having been suggested. 
The purpose of this paper is to provide a first definite answer this question. 

\subsection{The geometric setting}

We now survey the role of Whittaker coefficients in geometric Langlands
and state some of our main results.

\subsubsection{Notation}

We work over a field $k$ of characteristic zero. We fix $X$ a smooth,
geometrically connected projective curve over $k$.\footnote{In the body of
the paper, we assume at times that $X$ admits a $k$-point $x \in X(k)$.
This is a lazy crutch; our main theorems are verifiable after finite
degree field extensions. Therefore, in the body of the paper, 
we sometimes allow ourselves to ignore the fact that
this is a genuine additional hypothesis on $X$ if $k$ is not algebraically closed.}

We let $G$ be a split reductive group over $k$ with Langlands dual group 
$\check{G}$. We let $B$ be a Borel in $G$ with unipotent radical $N$.
We let e.g. $\Bun_G$ denote the moduli stack of $G$-bundles
on $X$, and $\LocSys_{\check{G}}$ the moduli stack of $\check{G}$-bundles on $X$
with connections, i.e., \emph{de Rham} $\check{G}$-local systems on $X$.

\subsubsection{The geometric Langlands conjecture (after Beilinson-Drinfeld,
Arinkin-Gaitsgory, and Gaitsgory)}

Recall the statement of the geometric Langlands conjecture of
Beilinson-Drinfeld (given in this form by Arinkin-Gaitsgory \cite{arinkin-gaitsgory}):
\[
\bL_G:D(\Bun_G) \simeq \IndCoh_{\Nilp_{\on{spec}}}(\LocSys_{\check{G}}) 
\simeq \Ind(\Coh_{\Nilp_{\on{spec}}}(\LocSys_{\check{G}})).
\]

\noindent Here we refer to \cite{arinkin-gaitsgory} for discussion of
the right hand side; we simply say that 
$\Coh_{\Nilp_{\on{spec}}}(\LocSys_{\check{G}}) \subset \Coh(\LocSys_{\check{G}})$
is a certain subcategory of the DG (derived) category of bounded complexes of coherent
sheaves; we are denoting by $\Nilp_{\on{spec}} \subset T^*[-1] \LocSys_{\check{G}}$
the \emph{spectral global nilpotent cone} considered in \cite{arinkin-gaitsgory}.

\subsubsection{}

There are many compatibilities the above equivalence is supposed to satisfy;
see \cite{dennis-laumonconf} for an overdetermined list.
Here is a key one, called the \emph{Whittaker normalization}.

There is a functor:
\[
\coeff:D(\Bun_G) \to \Vect
\]

\noindent of \emph{first\footnote{We are ambivalent about the indexing
here, but use this terminology in the introduction. See Remark \ref{r:first?}.}
Whittaker coefficient}; we remind that
in the categorical framework, vector spaces are considered analogues of numbers
in the classical setting.
This functor is a precise geometric
analogue of the integral \eqref{eq:an} for $n = 1$. Roughly speaking,
one pulls back to $\Bun_N$, tensors with an Artin-Schreier/exponential
sheaf (relative to a non-degenerate character), 
and then pushes forward to a point; we refer to \S \ref{s:coeff-backround}
for details (including normalizations on cohomological shifts).

Then the diagram:
\[
\begin{tikzcd}
D(\Bun_G) 
\arrow[rr,"\overset{\bL_G}{\simeq}"] 
\arrow[dr,"\coeff",swap]
&&
\IndCoh_{\Nilp_{\on{spec}}}(\LocSys_{\check{G}})
\arrow[dl,"{\Gamma(\LocSys_{\check{G}},-)}"]
\\
&
\Vect
\end{tikzcd}
\]

\noindent is supposed to commute.

\begin{rem}

For an irreducible local system $\sigma \in \LocSys_{\check{G}}(k)$, 
let $\delta_{\sigma}$ denote the skyscraper sheaf at this point.
There is an object $\Aut_{\sigma} \in D(\Bun_G)$ corresponding to it,
which is the corresponding automorphic eigensheaf. The Whittaker
normalization here should yield an isomorphism $\coeff(\Aut_{\sigma}) \simeq k
\in \Vect$, pinning down all ambiguity of the choice of eigensheaf.
This may be compared to the classical setting of modular forms,
where one normalizes a cuspidal eigenform by requiring $a_1 = 1$.

\end{rem}

\subsubsection{}

In the geometric setting, there are additional Whittaker coefficients,
analogous to the $a_n$ for other $n$'s. The reader may turn to 
\S \ref{s:coeff-backround} for the construction of functors
$\coeff_D:D(\Bun_G) \to \Vect$ indexed by $\check{\Lambda}^+$-valued
divisors $D$ on $X$; for a smarter (and more conceptual) construction, see
\cite{dennis-laumonconf} \S 5.8.

Below, we describe an alternative construction that is more easily stated. 

Gaitsgory has shown \cite{generalized-vanishing} that there
is a canonical action of $\QCoh(\LocSys_{\check{G}})$ on 
$D(\Bun_G)$ refining the Hecke action; this is the 
\emph{spectral decomposition} of the automorphic category
$D(\Bun_G)$.

It follows that there is a unique $\QCoh(\LocSys_{\check{G}})$-linear functor
functor:
\[
\coeff^{\enh}:D(\Bun_G) \to \QCoh(\LocSys_{\check{G}})
\]

\noindent fitting into a commutative diagram:
\[
\begin{tikzcd}
D(\Bun_G)
\arrow[rr,"\coeff^{\enh}"]
\arrow[drr,"\coeff"]
&&
\QCoh(\LocSys_{\check{G}})
\arrow[d,"\Gamma"]
\\
&&
\Vect.
\end{tikzcd}
\]

\noindent We provide details in \S \ref{ss:coeff-enh}. As in \emph{loc. cit}.,
$\coeff^{\enh}$ ``knows" all Whittaker coefficients of automorphic 
sheaves simultaneously (while also encoding reciprocity laws between them).

We see that the geometric Langlands equivalence must fit into a commutative
diagram:
\[
\begin{tikzcd}
D(\Bun_G) 
\arrow[rr,"\overset{\bL_G}{\simeq}"] 
\arrow[dr,"\coeff^{\enh}",swap]
&&
\IndCoh_{\Nilp_{\on{spec}}}(\LocSys_{\check{G}})
\arrow[dl,"\Psi"]
\\
&
\QCoh(\LocSys_{\check{G}})
\end{tikzcd}
\]

\noindent where the functor $\Psi$ is \emph{almost} an equivalence
(see \cite{indcoh}, \cite{arinkin-gaitsgory}).

\subsubsection{A strategy for proving the geometric Langlands equivalence}\label{sss:strategy}

The functor $\Psi$ displayed above is fully faithful on compact
objects (it induces the embedding 
$\Coh_{\Nilp_{\on{spec}}}(\LocSys_{\check{G}}) \subset \Coh(\LocSys_{\check{G}})
\subset \QCoh(\LocSys_{\check{G}})$).

Therefore, proving the geometric Langlands equivalence 
amounts to showing:

\begin{itemize}

\item $\coeff^{\enh}$ is fully faithful on the subcategory
$D(\Bun_G)^c \subset D(\Bun_G)$ of
compact objects.

\item $\coeff^{\enh}$ maps compact objects of $D(\Bun_G)$ onto
$\Coh_{\Nilp_{\on{spec}}}(\LocSys_{\check{G}})$.

\end{itemize}

In \cite{dennis-laumonconf}, Gaitsgory outlines a strategy for proving
these claims for $GL_n$ (cf. below), with a strategy that should
adapt for general reductive $G$ assuming some knowledge of Whittaker
coefficients here. 
The ideas are complicated, involving degenerate Whittaker
coefficients, Eisenstein series, and Kac-Moody representations.
However, the basic idea in the strategy is that stated above.

\begin{rem}

The above strategy may be compared to Soergel's bimodule theory 
via the analogies:

\begin{center}
\begin{tabular}{ |c|c| } 
\hline
 Soergel theory & Geometric Langlands \\
 \hline 
The functor $\bV$ & The functor $\coeff$ \\
Endomorphismensatz & The existence of $\coeff^{\enh}$ \\
Struktursatz & Fully faithfulness of $\coeff^{\enh}$ on $D(\Bun_G)^c$ \\
 \hline
\end{tabular}
\end{center}

We remark that the bottom right entry of this table remains conjectural
(beyond $GL_n$, see below). 

\end{rem}

\subsubsection{The $GL_n$ case}

It is known\footnote{In geometric Langlands, this style of argument
has a long lineage that we do not survey here. We refer to \cite{fgv-gl} as
one key example.} 
(cf. \cite{dennis-laumonconf} ``Quasi-Theorem" 8.2.10.
\cite{dario-extended}) that $\coeff^{\enh}$ is fully faithful 
on the category $D_{\cusp}(\Bun_{GL_n})$ 
of cuspidal $D$-modules \emph{in the $GL_n$ case}.
The argument imitates the proof of the multiplicity one theorem for $GL_n$
for automorphic forms, going through the
mirabolic subgroup.

\subsubsection{The general case}

However, as for automorphic forms, we have been unable to prove
anything about Whittaker coefficients for general reductive groups $G$.

This failure has stood as a point of some concern. For instance, number theorists
often express consternation that the geometric situation is conjectured to be
so different from the arithmetic situation, where cuspidal automorphic
representations commonly are non-generic. One  
imagines that if geometric Langlands fails, it fails because
the nice predictions regarding Whittaker coefficients 
are incompatible with some pathological example for automorphic sheaves.

Our first main theorem states that this does not occur:
 
\begin{introthm}\label{it:cusp-cons}

The functor:
\[
D_{\cusp}(\Bun_G) \subset D(\Bun_G) \xar{\coeff^{\enh}} \QCoh(\LocSys_{\check{G}})
\]

\noindent is conservative. That is, if $\sF \in D_{\cusp}(\Bun_G)$
with $\coeff^{\enh}(\sF) = 0$, then 
$\sF = 0$.

\end{introthm}

\begin{rem}

Applying the definition of cuspidal $D$-modules (and the
existence of the left adjoint Eisenstein functors $\on{Eis}_!$),
the above is equivalent to the assertion that $D(\Bun_G)$
is generated under colimits by Eisenstein series $D$-modules for proper
parabolic subgroups and Poincar\'e series $D$-modules;
we refer to \cite{dennis-laumonconf} for the definitions.

In other words, Theorem \ref{it:cusp-cons} asserts that the
geometric representation theorist's favorite methods for producing
automorphic $D$-modules are in fact exhaustive.

\end{rem}

\begin{rem}

To match our knowledge in the $GL_n$ case, we would want to know that
the functor of Theorem \ref{it:cusp-cons} is fully faithful;
this paper does not settle that question.

\end{rem}

\subsubsection{Tempered $D$-modules}\label{ss:temp-conventions}

In fact, we prove a stronger result that Theorem \ref{it:cusp-cons};
it is more technical to state, but optimal.

Fix a point $x \in X(k)$. Using derived Satake, Arinkin-Gaitsgory
defined a subcategory:
\[
D(\Bun_G)^{\at} \subset D(\Bun_G)
\]

\noindent of \emph{anti-tempered} objects; the terminology is taken
from \cite{dario-ramanujan} \S 2. The embedding here admits
a left adjoint. There is a certain quotient category $D(\Bun_G)^{\temp}$.

A priori, the above definitions depend on the point $x \in X(k)$; in
\cite{independence}, we showed the category is actually independent of
this choice in a strong sense; this justifies omitting $x$ from the notation.

According to Arinkin-Gaitsgory, under geometric Langlands, the quotient
$D(\Bun_G)^{\temp}$ should identify with the quotient
$\QCoh(\LocSys_{\check{G}})$ of $\IndCoh_{\Nilp_{\on{spec}}}(\LocSys_{\check{G}})$.

It is straightforward to see that $\coeff^{\enh}$ factors through
$D(\Bun_G)^{\temp}$. Therefore, by the logic of \S \ref{sss:strategy}, 
one expects the induced functor:
\[
D(\Bun_G)^{\temp} \to \QCoh(\LocSys_{\check{G}})
\] 

\noindent to be an equivalence; this is the \emph{tempered geometric
Langlands} conjecture.

We prove:

\begin{introthm}\label{it:temp-cons}

The above functor
\[
\coeff^{\enh}:
D(\Bun_G)^{\temp} \to \QCoh(\LocSys_{\check{G}})
\]

\noindent is conservative.

\end{introthm}

This result appears as Theorem \ref{t:temp-cons} in the body of the paper. 

By \cite{dario-ramanujan}, the composition: 
\[
D_{\cusp}(\Bun_G) \subset D(\Bun_G) \to  
D(\Bun_G)^{\temp}
\]

\noindent is fully faithful (more precisely, Beraldo shows
$D_{\cusp}(\Bun_G)$ is left orthogonal to $D(\Bun_G)^{\at}$).

Therefore, Theorem \ref{it:temp-cons} implies Theorem \ref{it:cusp-cons}.
We focus our attention on the latter result in the remainder of
this introduction.

\subsection{Informal overview}

Our work contains other new, intermediate results of independent interest; 
we detail them later in the introduction. 
First, we informally describe their context, and the
overall setting for our proof of Theorem \ref{it:temp-cons}. 

\subsubsection{}

First, the recent work \cite{agkrrv1} provides a new set of tools for studying
$D(\Bun_G)$. In effect, \S 20 of \emph{loc. cit}. reduces understanding
$D(\Bun_G)$ to studying a certain subcategory $\Shv_{\Nilp}(\Bun_G)$
where all the objects are (colimits of) \emph{holonomic $D$-modules
with regular singularities}. These are the objects the Riemann-Hilbert
correspondence is concerned with. One may use additional sheaf-theoretic
tools to study them, following e.g. \cite{kashiwara-kawai}, \cite{ginzburg-inventiones},
and \cite{ks}; broadly speaking, these additional tools are parts of
\emph{microlocal geometry}.
 
\subsubsection{Irregular singular support}

Let us return to the case where $G = PGL_2$. Recall from 
Remark \ref{r:cusp-fm} that modular forms with vanishing Whittaker
coefficients are constant. In the geometric setting,
one similarly can show that $D(\Bun_{PGL_2})^{\at}$ consists
of objects with constant cohomologies, equivalently (by simple-connectivity
in this case), with lisse cohomologies.

One can say this differently: $D(\Bun_{PGL_2})^{\at}$ is exactly the
category of $D$-modules with singular support in the zero section.

Our starting point is the idea that this should generalize: for general $G$,
$D(\Bun_G)^{\at}$ should be the category with \emph{irregular} singular
support. 

At the very least, we obtain a similar result in the nilpotent setting:

\begin{introthm}\label{it:at-irreg}

The category $\Shv_{\Nilp}(\Bun_G)^{\at}$ coincides
with $\Shv_{\Nilp_{\irreg}}(\Bun_G)$, the subcategory of objects
with \emph{irregular nilpotent singular support}.

\end{introthm}

We will refine this result in our later discussion. But, roughly speaking,
the overall strategy is to connect both (anti-)temperedness and Whittaker
coefficients to microlocal properties of sheaves, thereby proving 
Theorem \ref{it:temp-cons}.

\subsection{Results for nilpotent sheaves}

We obtain some striking results for Whittaker coefficients
of sheaves with nilpotent singular support, that we 
describe presently.  

\subsubsection{Some motivating geometry}

Recall that there is a characteristic polynomial map $\chi:\fg/G \to \fg//G$;
here the left hand side is the stack quotient and the right hand
side is the GIT quotient. The nilpotent cone $\sN$ is characterized
by the formula $\sN/G = \chi^{-1}(0)$.
The Kostant slice defines a section of $\chi$. Clearly the Kostant slice
intersects $\sN/G$ in a single point.

The above story has a global analogue. The role of $\fg/G$ is played
by $\Higgs_G = T^*\Bun_G$, the space of Higgs bundles for $G$.
The role of $\fg//G$ is played by the Hitchin base, while $\chi$ is
replaced by the Hitchin fibration.
The role of $\sN/G$ is played by the \emph{global nilpotent cone} 
$\Nilp \subset T^*\Bun_G$, which is by definition the zero fiber
of the Hitchin fibration. The Kostant slice admits a global analogue
(called $\fg\mathendash 0$-opers in \cite{hitchin} \S 3.1.14).

Therefore, the global Kostant slice intersects
$\Nilp$ at a distinguished point, which we label
$f^{\on{glob}}$ in \S \ref{sss:kost-defin}.
One can show that this point lies in the smooth
locus of $\Nilp$; therefore, there is a unique irreducible 
component $\Nilp^{\Kos}$ of $\Nilp$ containing $f^{\on{glob}}$.

We remark that when the genus of the curve is $>1$, 
the Kostant slice and $\Nilp$ are both Lagrangians in $T^*\Bun_G$.

\subsubsection{}

Finally, we make one more remark, connecting the 
above to Whittaker coefficients. 

To state it precisely,
we recall that it is not quite $\Bun_N$ that appears in the definition 
of $\coeff$, but a twisted form of $\Bun_N$ that we denote $\Bun_N^{\Omega}$
in \S \ref{s:coeff-backround}.
This form of $\Bun_N^{\Omega}$ has a canonical map 
$\psi:\Bun_N^{\Omega} \to \bA^1$, which is the \emph{Whittaker character}.
This defines a Lagrangian $d\psi:\Bun_N^{\Omega} \to T^*\Bun_N^{\Omega}$.
We may compose this Lagrangian with the natural Lagrangian
correspondence between $T^*\Bun_N^{\Omega}$ and $T^*\Bun_G$;
tracing the definitions, the resulting Lagrangian in $T^*\Bun_G$
is the global Kostant slice. 

In this sense, we may view the global Kostant slice as a microlocal
shadow of the functor $\coeff$; physicists would say that the
global Kostant slice is the \emph{brane} corresponding to the functor
$\coeff$.  

\subsubsection{Statement of the main result}

We can now formulate:

\begin{introthm}\label{it:micro}

The cohomologically shifted functor of first Whittaker coefficient:
\[
\coeff[\dim \Bun_G]:\Shv_{\Nilp}(\Bun_G) \to \Vect
\]

\noindent is $t$-exact and commutes with Verdier duality.
Moreover, for constructible $\sF \in \Shv_{\Nilp}(\Bun_G)$,
the Euler characteristic of $\coeff[\dim \Bun_G](\sF)$ 
equals the order of the characteristic cycle at $\Nilp^{\Kos}$.

\end{introthm}

We hope the geometry described above adequately has 
motivated this result.
It is obtained by combining Theorem \ref{t:index}, 
Theorem \ref{t:coeff-exact}, and Theorem \ref{t:coeff-!}
in the body of the text.

\subsubsection{A more refined picture}

It can be made more precise as follows. Suppose $\sY^{\on{an}}$ 
is a complex manifold, $\Lambda \subset T^*\sY^{\on{an}}$ is a closed, conical
holomorphic Lagrangian; let $\Lambda^{\on{sm}} \subset \Lambda$ be the smooth locus.

Kashiwara-Schapira\footnote{This assertion is difficult to track
in the stated form. A close result is \cite{ks} Corollary 7.5.7,
as well as the subsequent remark.} 
\cite{ks} associate to a 
sheaf $\sF$ on $\sY^{\on{an}}$ with singular support in $\Lambda$
a certain local system $\mu_{\Lambda}(\sF)$ on $\Lambda^{\on{sm}}$,
which is a form of the \emph{microlocalization} of $\sF$.

The fibers of $\mu_{\Lambda}(\sF)$ at points of $\Lambda^{\on{sm}}$
are called \emph{microstalks}, and may be computed by suitably
transverse vanishing cycles of $\sF$. In particular, formation 
of microstalks is $t$-exact (up to shift) and commutes with Verdier duality.
In addition, the Euler characteristics 
of these fibers are the degrees of the characteristic cycle of $\sF$
at the given point.

Therefore, our motivation is that for sheaves with nilpotent singular
supports (which, we remind, are automatically regular singular, so have
Betti cousins), $\coeff$ is the microstalk at $f^{\on{glob}}$.

\begin{rem}

This idea is quite natural, and indeed, when we began discussing this
work with others, we learned 
that it had been considered some time ago by 
others: Drinfeld advertised the idea some 20 years ago, and 
Nadler advertised it some 10 years ago. We are not aware of any recorded
source for it. 

We understand that David Nadler and Jeremy Taylor
have a proof of this precise assertion, directly proving that 
$\coeff$ is computed by the microstalk at $f^{\on{glob}}$, using topological methods;
this is in contrast to our methods, which use special properties
of the automorphic setting.\footnote{Since the first draft of this
paper was circulated, a preprint of their work has appeared: see 
\cite{nadler-taylor}.}

\end{rem}

\begin{rem}

We were unable to prove directly that the Whittaker coefficients
of nilpotent sheaves \emph{actually} are microstalks. What we prove
is Theorem \ref{it:micro} as stated, which informally asserts that
the coefficient functor has all the same good properties as the microstalk
functor would.

Although the geometric picture described above is quite simple and feel
general, we use the specific tools of geometric representation theory; 
specifically, we use a (still forthcoming) result of Kevin Lin.
See \S \ref{s:coeff-exact} for details.

\end{rem}

\subsubsection{Comparison to arithmetic ideas}

We note that similar principles have appeared previously in 
harmonic analysis. See \cite{rodier-characters} \S IV Remarque 2 (see also
\cite{moeglin-waldspurger}). The assertion is that
(under favorable hypotheses), the multiplicity of the Fourier transform of a character
of a representation of a $p$-adic reductive group 
at the regular nilpotent orbit is the dimension
of its Whittaker model. 

It is quite enticing to better understand this analogy more precise.

\subsection{A result for Hecke functors}

We now state another intermediate result of independent interest.

\subsubsection{Background}

Recall that Hecke functors are not $t$-exact on $\Bun_G$.
For instance, a Hecke functor for a representation $V$ 
acting on the constant sheaf of $\Bun_G$ (which is in a single 
degree by smoothness of $\Bun_G$) tensors this constant
sheaf with a \emph{cohomologically sheared} version of $V$. 

However, Hecke functors are not too far from exact either.
For instance, one classically expects cuspidal perverse eigensheaves
for irreducible local systems (and see Theorem \ref{it:eigensheaf} below).
By definition, Hecke functors
transform such objects by tensoring with a (classical) vector space.

\subsubsection{Statement of the result}

The above discussion is suggestive of what the obstruction to exactness is
in general:

\begin{introthm}\label{it:hecke}

\begin{enumerate}

\item There is a unique $t$-structure on $D(\Bun_G)^{\temp}$
for which the projection $D(\Bun_G) \to D(\Bun_G)^{\temp}$
is $t$-exact.

\item\label{i:it-hecke-2} Let $V \in \Rep(\check{G})^{\heart}$ be a representation.
Then for $x \in X$, the induced Hecke functor:
\[
D(\Bun_G)^{\temp} \to D(\Bun_G)^{\temp}
\]

\noindent is $t$-exact.

\item More generally, for $V$ as above, the \emph{parametrized} Hecke
functor:
\[
H_V:D(\Bun_G)^{\temp} \to D(\Bun_G)^{\temp} \otimes D(X)
\]

\noindent is $t$-exact (up to shift by $1 = \dim X$).

\end{enumerate}

\end{introthm}

See Theorem \ref{t:hecke-exact-loc} and Theorem \ref{t:hecke-exact-curve}
in the body of the paper.

The proofs of the first two statements are quite direct,
but seem not to have been previously observed.
The third is a minor variant, except it relies on the
independence of point in the definition of temperedness
(in other words: in \eqref{i:it-hecke-2}, it is important
that the implicit point $x \in X(k)$ in the definition of
the tempered category be taken to be the same as where the Hecke
functors are taken).
The argument from \cite{independence} works for
$\Shv_{\Nilp}(\Bun_G)$ in the $\ell$-adic setting, but 
we are not sure how to adapt it to the full category
$\Shv(\Bun_G)$ in this setting. If so, our methods would
yield a proof of Theorem \ref{it:hecke} in the $\ell$-adic case as well.

\begin{rem}\label{r:vanishing}

Besides the $\ell$-adic issue raised above, our argument
provides an alternative to \cite{vanishing} \S 2.12.
We highlight that the construction in \emph{loc. cit}. 
applies only for $GL_n$, and is the major technical point
in that paper. 

More explicitly: the main result of \cite{vanishing} is the
formation of a certain quotient
of $D(\Bun_{GL_n})$ with certain favorable properties, including
that Hecke functors act exactly on it. The construction in \emph{loc. cit}.
does not make sense for general $G$. We have provided an alternative
(genuinely different) construction of a quotient category with
the same favorable properties. Moreover, our arguments are
substantially more direct than those in \cite{vanishing}. 

With that said, our argument in \cite{independence} uses
Gaitsgory's \emph{generalized vanishing conjecture} from
\cite{generalized-vanishing}. As explained in \cite{generalized-vanishing},
this result immediately implies the vanishing conjecture 
considered in \cite{vanishing}. For this reason, 
we cannot say that we have found a better understanding of the
(not generalized) vanishing conjecture, only of the intermediate results
used in \cite{vanishing}.

\end{rem}

\subsection{A remark on tempered Langlands}\label{ss:intro-temp-gl}

We now describe a surprising consequence of our work
for the geometric Langlands equivalence.

Roughly speaking, one is commonly taught that
geometric Langlands is an equivalence of \emph{derived} categories,
not abelian categories. We explain that this is in some sense
wrong; \emph{most} of
geometric Langlands can actually be understood as an equivalence
of \emph{abelian} categories.

\subsubsection{}

It is well-known that the geometric Langlands equivalence 
is emphatically not exact. 

First, in geometric class field theory, 
one finds $D(\Bun_{\bG_m}) \simeq \QCoh(\LocSys_{\bG_m})$.
The functor is a variant on Fourier-Mukai for abelian varieties;
experimentally, one finds the latter is far from exact.

Second, for non-abelian $G$, the constant sheaf on $\Bun_G$
should map to an object of $\IndCoh_{\Nilp_{\on{spec}}}(\LocSys_{\check{G}})$
concentrated in cohomological degree $-\infty$ (i.e., in degrees 
$\leq -n$ for all $n$).

However, one expects some exactness properties. For instance,
there are supposed to exist \emph{perverse} eigensheaves for irreducible
local systems (and see Theorem \ref{it:eigensheaf});
these correspond (up to shift) to skyscraper sheaves at smooth, 
irreducible points of $\LocSys_{\check{G}}$.

\subsubsection{}\label{sss:temp-gl-exact-intro}

We recall the setting of \emph{restricted} geometric Langlands
considered in \cite{agkrrv1}.

Here we expect an equivalence:
\[
\Shv_{\Nilp}(\Bun_G) \simeq 
\IndCoh_{\Nilp_{\on{spec}}}(\LocSys_{\check{G}}^{\on{restr}}).
\]

\noindent The space $\LocSys_{\check{G}}^{\on{restr}}$ is defined
as in \emph{loc. cit}.

The tempered analogue should instead find an equivalence:
\[
\Shv_{\Nilp}(\Bun_G)^{\temp} \overset{\bL_G^{\temp}}{\simeq} 
\QCoh(\LocSys_{\check{G}}^{\on{restr}})
\]

\noindent We claim that our results imply this (conjectural) equivalence must be $t$-
exact up to shift. Perhaps more concretely, this means that the composition:
\[
\Shv_{\Nilp}(\Bun_G) \simeq 
\IndCoh_{\Nilp_{\on{spec}}}(\LocSys_{\check{G}}^{\on{restr}})
\xar{\Psi}
\QCoh(\LocSys_{\check{G}}^{\on{restr}})
\]

\noindent should be $t$-exact.

Indeed, let $x \in X(k)$ be a point; this defines
a $\check{G}$-torsor on $\LocSys_{\check{G}}^{\on{restr}}$.
By \cite{agkrrv1} Theorem 1.4.5, the total space of this
torsor is a union of ind-affine formal schemes by an action of $\check{G}$.

It follows that its functor $\Gamma_!$ to $\Vect$ (considered in \cite{agkrrv1}
\S 7) is $t$-exact, and that an object 
$\sG \in \QCoh(\LocSys_{\check{G}}^{\on{restr}})$ lies
in degree $0$ if and only if $\Gamma_!(\sG \otimes \sE_{V,x}) \in \Vect$ is
in degree $0$ for all $V$; here $\sE_{V,x}$ is the vector
bundle on $\LocSys_{\check{G}}^{\on{restr}}$ defined by the pair $(V,x)$.

On the other hand, for $\sF \in D(\Bun_G)^{\temp,\heart}$,
we expect $\Gamma_!(\bL_G^{\temp}(\sF) \otimes \sE_{V,x})
\simeq \coeff(H_{V,x} \star \sF)$; the latter lies purely
in degree $\dim \Bun_G$ by Theorem \ref{it:micro} and
Theorem \ref{it:hecke}, so $\bL_G^{\temp}(\sF)$ must 
lie in degree $\dim \Bun_G$ as well. 

\begin{rem}

We similarly expect that the tempered \emph{Betti} Langlands
equivalence conjecture in \cite{bz-n-betti-gl} is $t$-exact;
this corresponds to fact that the Betti moduli stack of
local systems is the quotient of an affine scheme by 
an action of $\check{G}$. 

\end{rem}

\subsection{Application to Hecke eigensheaves}

We apply our results to show the following result, which can be
thought of as an unconditional realization of the philosophy of 
\S \ref{ss:intro-temp-gl}.

\begin{introthm}\label{it:eigensheaf}

Let $\sigma$ be an irreducible $\check{G}$-local system on $X$.

\begin{enumerate}

\item Any Whittaker-normalized\footnote{See Remark \ref{r:whit-norm} for our
precise convention.}
Hecke eigensheaf $\sF_{\sigma}$ with
eigenvalue $\sigma$ is perverse. Moreover, at least if $k$ is algebraically
closed, Whittaker-normalized Hecke eigensheaves exist.

\item If $\sigma$ is \emph{very irreducible} in the sense of \S \ref{sss:very-irred},
the restriction of $\sF_{\sigma}$ to each connected component of $\Bun_G$
is irreducible.

\end{enumerate}

\end{introthm}

The reader will find this result as Theorem \ref{t:eigensheaf}; we also
refer to Remark \ref{r:jh-genl} where it is noted that Whittaker-normalized
eigensheaves are semi-simple for any (possibly not very) irreducible local system.

We note that this answers an old question. Namely, from the point of view of the
categorical geometric Langlands conjectures, it is not clear why 
eigensheaves should be perverse or irreducible. We show that this
follows from exactness (and conservativeness) properties of the Whittaker functor.

We remark that the existence of Hecke eigensheaves stated above is
proved via opers and is disjoint from the methods of our paper; we have
relegated the argument to Appendix \ref{a:localization}. Our contributions
are more to the structure of (normalized) eigensheaves, and we include the
existence argument for the sake of completeness.

\subsection{Outline of the argument}

We now outline our argument for Theorem \ref{it:temp-cons}.
The details are provided in \S \ref{s:finale}.

\subsubsection{}

First, we reduce to the corresponding statement
for $\Shv_{\Nilp}(\Bun_G)$ using the technology of
\cite{agkrrv1}.

\subsubsection{}

Let $\o{\Nilp} \subset \Nilp$ denote the open 
of \emph{generically regular} nilpotent Higgs bundles.

We prove:

\begin{introthm}\label{it:at}

Any object $\sF\in \Shv_{\Nilp}(\Bun_G)$ that does not lie in 
$\Shv_{\Nilp}(\Bun_G)^{\at}$ has $\SS(\sF) \cap \o{\Nilp} \neq \emptyset$.

\end{introthm}

This result is our Theorem \ref{t:at}. The proof reduces to a parallel statement
for the flag variety of the finite dimensional group $G$. 
We translate this to a statement about Lie algebra representations
via Beilinson-Bernstein. Finally, we apply a theorem of Losev \cite{losev-holonomic}
regarding associated varieties of $\fg$-modules (and proved
using ideas reminiscent of microlocal differential operators).

We wish to be clear: this is not the proof of the theorem 
written in The Book (in the sense of Erd\H{o}s); our argument
is not geometric. It would be
far better to have a proof that relies only on standard
properties of singular support and adapts to the $\ell$-adic setting.

\subsubsection{}

Next, we have:

\begin{introthm}\label{it:trans}

Suppose $\sF\in \Shv_{\Nilp}(\Bun_G)$ satisfies 
$\SS(\sF) \cap \o{\Nilp} \neq \emptyset$. Then for a
point $x \in X(k)$, there exists a representation 
$V^{\check{\lambda}} \in \Rep(\check{G})^{\heart}$
such that $\Nilp^{\Kos} \subset \SS(H_{V^{\check{\lambda}},x} \star \sF)$.

\end{introthm}

This result appears in our work as Theorem \ref{t:hecke-to-kostant}. 
The proof uses basic geometry of the global nilpotent cone and standard
properties of singular support. 

\subsubsection{}

Finally, we deduce the claim as follows. 

Suppose $\sF\in \Shv_{\Nilp}(\Bun_G)$ does not lie
in $\Shv_{\Nilp}(\Bun_G)^{\at}$. We need to show that
$\coeff^{\enh}(\sF) \neq 0$. 

By Theorem \ref{it:at}, we have 
$\SS(\sF) \cap \o{\Nilp} \neq \emptyset$. 
By definition, it suffices to show that 
$\coeff(H_{V^{\check{\lambda}},x} \star \sF) \neq 0$
for some $V^{\check{\lambda}}$. Therefore, by Theorem \ref{it:trans},
we are reduced to showing that $\coeff(\sF) \neq 0$
when $\Nilp^{\Kos} \subset \SS(\sF)$.

But now the argument follows from Theorem \ref{it:micro}:
by $t$-exactness, we are reduced to the case where $\sF$ is
perverse (i.e., constructible and concentrated in cohomological degree $0$).
In that case,
the Euler characteristic of $\coeff(\sF)$ equals the degree of 
$\CC(\sF)$ at $\Nilp^{\Kos}$, which is non-zero by assumption.

\subsubsection{}

We remark that Theorem \ref{it:at-irreg} follows easily, and therefore
do not prove it in the body of the paper.
Here is the outline:
\begin{itemize}

\item $\Shv_{\Nilp}(\Bun_G)^{\at} \subset \Shv_{\Nilp_{\irreg}}(\Bun_G)$
by Theorem \ref{it:at}.

\item $\Shv_{\Nilp_{\irreg}}(\Bun_G)$ lies in 
$\Ker(\coeff)$ by Lemma \ref{l:cons-kos}.
Therefore, as the category $\Shv_{\Nilp_{\irreg}}(\Bun_G)$ 
is stable under Hecke functors (cf. \S \ref{sss:hecke-irreg}),
so $\Shv_{\Nilp_{\irreg}}(\Bun_G)$ lies in $\Ker(\coeff_D)$ for
each $D$. It follows formally that 
$\Shv_{\Nilp_{\irreg}}(\Bun_G) \subset \Ker(\coeff^{\enh})$. 

\item $\Ker(\coeff^{\enh})$ equals (and in particular, is contained
in) $D(\Bun_G)^{\at}$ by Theorem \ref{it:temp-cons}.

\end{itemize}

\subsubsection{Regarding the $\ell$-adic setting}

We expect analogues of each of the theorems above to hold for
the setting of $\ell$-adic sheaves considered in \cite{agkrrv1}.
In particular, we believe our overall strategy is the right one.

However, for each\footnote{The first two parts of Theorem \ref{it:hecke} are 
an exception.}
of the above theorems, we at some point in the argument use
specifics of $D$-modules, particularly regular holonomic $D$-modules/Betti
perverse sheaves. This is most egregious for Theorem \ref{it:at},
but is true at some point for every one of these results.

\subsection{Acknowledgements}

We are happy to thank Dima Arinkin, Sasha Beilinson, David Ben-Zvi, 
Dario Beraldo, Gurbir Dhillon, Vladimir Drinfeld, 
Pavel Etingof, Tony Feng, Dennis Gaitsgory, Tom Gannon, David Kazhdan, Kevin Lin,
Ivan Losev, Victor Ginzburg,
Sergey Lysenko, Ivan Mirkovic, David Nadler, 
Nick Rozenblyum, Yiannis Sakellaridis, Will Sawin,
Jeremy Taylor, and Yasha Varshavsky
for generously sharing their ideas, for related collaborations, and
for their encouragement.

S.R. was supported by NSF grant DMS-2101984.

\section{Notation}\label{s:notation}

In this section, we set up some notation that will be used throughout the
paper.

\subsection{Categories}

We freely use the language of $\infty$-categories and higher algebra,
cf. \cite{higheralgebra}, \cite{sag}, \cite{gr-i}, \cite{gr-ii}.

We understand DG categories as $k$-linear stable $\infty$-categories.
We let $\DGCat_{cont}$ denote the category of presentable (in particular:
cocomplete) DG categories with morphisms being continuous DG functors.
We freely use Lurie's symmetric monoidal structure $\otimes$ on
$\DGCat_{cont}$, and the associated duality formalism.

\subsection{Categories of sheaves}

Here we set up our notation and conventions for sheaves. We refer to
\cite{toy-model} Appendix A and \cite{agkrrv1} Appendices D and E for details
and proofs of various assertions.

\subsubsection{$D$-modules}

For a prestack $\sY$ locally almost of finite type, we let 
$D(\sY)$ denote the DG category of $D$-modules on $\sY$, 
defined as in \cite{gr-i}. For a map $f:\sY \to \sZ$, 
we let $f^!:D(\sZ) \to D(\sY)$ denote the corresponding pullback functor.
If $f$ is ind-representable, we let
$f_{*,dR}:D(\sY) \to D(\sZ)$ denote the pushforward functor.

Where defined, we let $f_!$ (resp. $f^{*,dR}$) denote the left adjoints
to these functors.

\subsubsection{}\label{sss:loc-cpt}

For $\sY$ an algebraic stack and $\sF \in D(\sY)$, we say
that $\sF$ is \emph{locally compact} if for every
affine scheme $S$ and every smooth map $f:S \to \sY$,
$f^!(\sF) \in D(S)$ is compact.

We remind that compact objects of $D(\sY)$ are 
locally compact, but the converse does not hold. For example, the constant
sheaf on $\bB \bG_m$ is locally compact but not compact; the same
applies for the constant sheaf on a non-quasi-compact scheme.
More generally, any constructible object (defined as below) is locally compact.

\subsubsection{Ind-constructible sheaves}

For $S$ an affine scheme of finite type, 
we let $\Shv(S)^c \subset D(S)$ denote the
subcategory of compact objects that are holonomic with regular singularities.
We then let $\Shv(S) = \Ind(\Shv(S)^c)$; this is a full subcategory
of $D(S)$.

For $\sY$ an algebraic stack, we let 
$\Shv(\sY) \coloneqq \lim_{S \to \sY} \Shv(S)$ and
let $\Shv(\sY)^{\on{constr}} \coloneqq \lim_{S \to \sY} \Shv(S)^c$.
In both circumstances, the limits are taken over affine schemes
$S$ mapping to $\sY$ and the implied functors are upper-$!$ functors.
Standard arguments allow us to replace the limit by that over
the subcategory of $S$'s mapping \emph{smoothly} to $\sY$.
It follows that $\Shv(\sY)$ has a natural $t$-structure,
and $\Shv(\sY)^{\on{constr}}$ is closed under truncations.

We refer to objects of $\Shv(\sY)$ as \emph{ind-constructible sheaves} on $\sY$
and objects of $\Shv(\sY)^{\on{constr}}$ as \emph{constructible sheaves}
on $\sY$. 

As in \cite{agkrrv1} \S F.2.5, we have a well-defined Verdier duality
equivalence $\Shv(\sY)^{\on{constr},\on{op}} \simeq \Shv(\sY)^{\on{constr}}$
that we denote $\bD^{\on{Verdier}}$.

\begin{rem}

Our usage of the notation $\Shv$ 
here is slightly different than e.g. in \cite{agkrrv1}, where
it is meant to express a certain ambivalence about the specific choice
of sheaf theory. We work in the context of $D$-modules in characteristic $0$, 
so do not share the ambivalence of \emph{loc. cit}. With that said,
the notation is similar in spirit.

\end{rem}

\subsubsection{Singular support}

When $\sY$ is an algebraic stack and 
$\Lambda \subset T^*\sY$ a closed, conical 
subscheme, we let
$D_{\Lambda}(\sY) \subset D(\sY)$ denote the full subcategory of
sheaves with singular support in $\Lambda$. 

Similarly, we let $\Shv_{\Lambda}(\sY) \subset \Shv(\sY)$
denote the corresponding full subcategory. 

We again refer to \cite{toy-model} \S A.3 for definitions.

For $\sF \in D(\sY)$, we let $\SS(\sF) \subset T^*\sY$ denote
the singular support of $\sF$, which is ind-closed in $T^*\sY$;
cf. \cite{agkrrv1} \S H.1.2. On occasion, for $\sF$ constructible,
we let $\CC(\sF)$ denote the characteristic cycle of $\sF$.

\subsection{Lie theory}

\subsubsection{} Throughout the paper, $G$ denotes a split reductive group
over $k$. We choose opposing Borel subgroups $B,B^- \subset G$ with
$B \cap B^- = T$ a fixed Cartan. We let $N$ (resp. $N^-$) denote the
unipotent radical of $B$ (resp. $B^-$).

We let $\Lambda$ denote the lattice of weights of $G$ and let
$\check{\Lambda}$ denote the lattice of coweights.\footnote{Our convention
here is opposite to \cite{agkrrv1}.} 
For $\check{\lambda} \in \check{\Lambda}$ and $\mu \in \Lambda$, we
let $(\check{\lambda},\mu) \in \bZ$ denote the pairing of the two.
We let $\Lambda^+ \subset \Lambda$ denote the subset of
dominant weights, and similarly for $\check{\Lambda}^+$.

We let $\sI_G$ denote the set of nodes for the Dynkin diagram of $G$.
For $i \in \sI_G$, we let $\alpha_i$ denote the corresponding
simple root.

We let $2\rho \in \Lambda$ denote the sum of the positive roots,
and similarly for $2\check{\rho} \in \check{\Lambda}$.

We let $\check{G}$ denote the Langlands dual group of $G$, considered
as an algebraic group over $k$.

\subsubsection{} We let $\fg_{\irreg} \subset \fg$ denote the reduced closed 
subscheme consisting of irregular elements. 

We let $\sN \subset \fg$ denote the nilpotent cone. 
We let $\sN_{\irreg} \coloneqq \sN \cap \fg_{\irreg}$ denote the subscheme
of irregular nilpotent elements. We let $\o{\sN} \subset \sN$
denote the open complement to $\sN_{\irreg}$, which parametrizes of
regular nilpotent elements.

\subsection{Higgs bundles}

\subsubsection{}

We remind that a \emph{Higgs bundle} (on $X$, for the group $G$) is a pair
$(\sP_G,\vph)$ where $\sP_G$ is a $G$-bundle and 
$\vph \in \Gamma(X,\fg_{\sP_G} \otimes \Omega_X^1)$.

Recall that Higgs bundles form an algebraic stack $\Higgs_G$, which 
can be written as a mapping stack:
\[
\Higgs_G \coloneqq
\Maps(X,\fg/{G \times \bG_m}) \underset{\Maps(X,\bB \bG_m)}{\times}
\{\Omega_X^1\}.
\]

We remind that our choice\footnote{More canonically, 
$T^*\Bun_G$ identifies with the variant of $\Higgs_G$ with 
$\fg^{\vee}$ replacing $\fg$ everywhere. For example, this applies
as well to possibly non-reductive affine algebraic groups.} 
of $\kappa_0$ induces an isomorphism:
\[
T^*\Bun_G \simeq \Higgs_G,
\]

\noindent which we take for granted in the sequel.

\subsubsection{Globalization}\label{sss:globalization}

For $\Lambda \subset \fg$ closed, conical, and stable under the $G$-action, 
it is convenient to denote:
\[
\Higgs_{G,\Lambda} \coloneqq 
\Maps(X,\Lambda/{G \times \bG_m}) \underset{\Maps(X,\bB \bG_m)}{\times}
\{\Omega_X^1\}.
\]

\noindent Clearly $\Higgs_{G,\Lambda}$ forms a closed substack of
$\Higgs_G$.

In the special case $\Lambda = \sN$, we let:
\[
\Nilp \coloneqq \Higgs_{G,\sN}.
\]

\noindent To reiterate: 
\[
\Nilp \subset \Higgs_G = T^*\Bun_G \text{ \emph{is the global nilpotent cone}.} 
\]

\noindent (We highlight this in part in acknowledgement that the notation does 
not make it easy for the reader to remember which of $\sN$ and $\Nilp$ has
to do with $\fg$ and which has to do with $\Higgs_G$.)

\subsection{Global nilpotent cone}

We now establish some notation relating to $\Nilp$.

\subsubsection{Irregular nilpotent Higgs bundles}

We define:
\[
\Nilp_{\irreg} \coloneqq \Higgs_{G,\sN_{\irreg}} \subset \Nilp.
\]

\noindent This stack parametrizes \emph{irregular nilpotent} Higgs bundles.
Clearly $\Nilp_{\irreg} \subset \Nilp$ is a closed substack.

\begin{example}

For $G = GL_2$, we have $\Nilp_{\irreg} = \Bun_G$, which 
is embedded in $\Higgs_G$ as the zero section.

\end{example}

\subsubsection{Generically regular nilpotent Higgs bundles}

We let:
\[
\o{\Nilp} \subset \Nilp
\]

\noindent denote the open complement to $\Nilp_{\irreg}$.
This is the stack of \emph{generically regular} Higgs bundles.
(We use this terminology because a point $\vph \in \Nilp$ lies
in $\o{\Nilp}$ if and only if there
is a dense open $U \subset X$ over which $\vph$ is regular nilpotent.)

\begin{example}

For $G = GL_2$, $\o{\Nilp}$ parametrizes pairs $(\sE,\vph)$ where
$\sE$ is a rank $2$ vector bundle and $\vph:\sE \to \sE \otimes \Omega_X^1$
is a non-zero Higgs field with $\vph^2 = 0$.

\end{example}

\subsubsection{Mapping stack notation}

Let $\sY$ be a stack and let $\o{\sY} \subset \sY$ be an open substack. 

We let $\Maps_{\on{nondeg}}(X,\sY \supset \o{\sY})$ denote the
prestack with $S$-points given by maps:
\[
y:X_S \coloneqq X \times S \to \sY
\]

\noindent such that there exists an open $U \subset X_S$ 
such that:
\begin{itemize}

\item $U \subset X_S$ is schematically dense.\footnote{I.e., 
for $U \subset Z \subset X_S$ with $Z \subset X_S$ closed, we necessarily
have $Z = X_S$.}

\item $U \to S$ is a (necessarily flat) cover.

\item $y|_U$ factors through $\o{\sY}$.

\end{itemize}

(See e.g. \cite{cpsi} \S 2.9 and \cite{schieder} \S 2.2.1, where similar
constructions are discussed.)

\subsubsection{A Springer construction}\label{sss:springer-desc}

In the above notation, we now clearly have:
\[
\o{\Nilp} = \Maps_{\on{nondeg}}(X,\sN/G \supset \o{\sN}/G).
\]

Let $\o{\fn} \subset \fn$ denote the open subscheme 
of elements of $\fn$ that are regular as elements of $\fg$.
For the reader's convenience, we remind that if we choose 
negative simple root vectors $f_i \in \fn^-$ (for $i \in \sI_G$) and let
$f_i^{\vee} \coloneqq \kappa_0(f_i,-):\fn \to \bA^1$ denote the
corresponding projection onto the simple root space, then
$\o{\fn} = \cap_{i \in \sI_G} \{f_i^{\vee} \neq 0\}$.

We now form:
\[
\o{\widetilde{\Nilp}} \coloneqq \Maps_{\on{nondeg}}(X,\o{\fn}/B \subset \fn/B).
\]

\noindent Observe that there are canonical maps:
\[
\begin{tikzcd}
\o{\widetilde{\Nilp}} 
\arrow[d] 
\arrow[r] 
&
\Bun_B
\arrow[r]
&
\Bun_T
\\
\o{\Nilp}.
\end{tikzcd}
\]

For a coweight $\check{\lambda}$, we let:
\[
\o{\widetilde{\Nilp}}^{\check{\lambda}} \subset 
\o{\widetilde{\Nilp}}
\]

\noindent denote the inverse image of\footnote{Explicitly, $\Bun_T^{\check{\lambda}}$
is the component of $\Bun_T$ parametrizing
$T$-bundles $\sP_T$ where $\deg(\sP_T^{\mu}) = (\mu,\check{\lambda})$
for each weight $\mu$; here $\sP_T^{\mu}$ is the line bundle (aka $\bG_m$-bundle)
induced by the map $\mu:T \to \bG_m$.}
$\Bun_T^{\check{\lambda}}$ along the top map.

It is easy to see that the projection:
\[
\o{\widetilde{\Nilp}}^{\check{\lambda}} \to \o{\Nilp}
\]

\noindent is a locally closed embedding. We therefore
abuse notation in letting $\o{\Nilp}^{\check{\lambda}}$ denote
the corresponding locally closed subscheme of $\o{\Nilp}$.
Note that every field-valued point of $\o{\Nilp}$ lifts
to $\o{\widetilde{\Nilp}}$ because $\fn/B \to \sN/G$ is 
proper and induces an isomorphism $\o{\fn}/B \isom \o{\sN}/G$;
therefore, the various $\o{\Nilp}^{\check{\lambda}}$ define 
a stratification of $\o{\Nilp}$.

The following result describes the geometry in more detail:

\begin{prop}\label{p:nilp-reg}

Let $\check{\lambda} \in \check{\Lambda}$ be a coweight.

\begin{enumerate}

\item $\o{\Nilp}^{\check{\lambda}}$ is smooth of dimension $\dim \Bun_G$.

\item Suppose $(\alpha_i,\check{\lambda})+2g-2 < 0$ for some
$i \in \sI_G$. Then: 
\[
\o{\Nilp}^{\check{\lambda}} = \emptyset.
\]

\item Suppose $(\alpha_i,\check{\lambda})+2g-2 \geq 0$ for all $i \in \sI_G$.
Then $\o{\Nilp}^{\check{\lambda}}$ is non-empty and connected.

\end{enumerate}

\end{prop}

This result is proved\footnote{In fact, the result in \cite{hitchin}
is formulated in greater generality: it allows for non-regular nilpotent
elements as well. We remark for the sake of comparison 
that in the notation of \cite{hitchin}
\S 2.10.3, $Y_C = Y_C^{*}$ for the regular nilpotent conjugacy class.
Similarly, in this regular nilpotent case, the space 
$M_C$ maps to $\Bun_T$; its fiber over $\sP_T \in \Bun_T^{\check{\lambda}}$
is $\prod_{i \in \sI_G} \Gamma(X,\sP_T^{\alpha_i} \otimes \Omega_X^1) \setminus 0$;
in particular, the fibers are empty or connected, and the condition of 
some fiber being non-empty is exactly the numerical condition given in the proposition.
} 
in \cite{hitchin} \S 2.10.3.

It follows from the proposition that the irreducible components
of $\o{\Nilp}$ are exactly the closures of the strata
$\o{\Nilp}^{\check{\lambda}}$. Therefore, we have an injective map:
\begin{equation}\label{eq:irr-nilp}
c:\Irr(\o{\Nilp}) \into \check{\Lambda}
\end{equation}

\noindent where $\Irr(\o{\Nilp})$ is the set of irreducible components
of $\o{\Nilp}$. We let $\check{\Lambda}^{\on{rel}} \subset \check{\Lambda}$
denote the image of this map (the notation abbreviates \emph{relevant}); 
explicitly, $\check{\lambda} \in \check{\Lambda}^{\on{rel}}$ if
and only if $(\alpha_i,\check{\lambda}) \geq 2-2g$ for all $i \in \sI_G$.

\begin{example}\label{e:nilp-gl2}

Suppose $G = GL_2$. Then for a coweight $\check{\lambda} = (d_1,d_2) \in \bZ^2$, 
$\o{\Nilp}^{\check{\lambda}}$ parametrizes short exact sequences 
$0 \to \sL_1 \to \sE \to \sL_2 \to 0$
plus a non-zero map $\ol{\vph}:\sL_2 \to \sL_1 \otimes \Omega_X^1$ where
$\deg \sL_i = d_i$; the corresponding Higgs field is:
\[
\sE \to \sL_2 \to \sL_1 \otimes \Omega_X^1 \into \sE\otimes \Omega_X^1.
\]

\noindent We remark that $\sL_1 = \Ker(\vph)$ and 
$\sL_2  = \on{Image}(\vph) \otimes \Omega_X^{1,\otimes -1}$ can be recovered from the
generically regular nilpotent Higgs field $\vph$.

\end{example}

\subsubsection{Invariants}\label{sss:discrepancy} 

For later reference, we attach two numerical data to a field-valued
point $(\sP_G,\vph) \in \o{\Nilp}$. 

First, we let: 
\[
c_1(\sP_G,\vph) \in \pi_1^{\on{alg}}(G) 
\coloneqq \check{\Lambda}/\bZ\bDelta
\]

\noindent denote the first Chern class of $\sP_G$
(cf. \cite{hitchin} \S 2.1.1). We
explicitly say: there is no dependence on $\vph$, and this
invariant behaves well in moduli (it is locally constant on $\Bun_G$).

Second, we define the \emph{discrepancy divisor}
of $(\sP_G,\vph)$, which is a $\check{\Lambda}_{G^{\ad}}^+$-valued divisor 
$\on{disc}(\sP_G,\vph)$ on
$X$, (for $\check{\Lambda}_{G^{\ad}}$ being 
coweights of the adjoint group $G^{\ad}$ of $G$) as follows. 
First, as above, by generic regularity (and the valuative criterion
of properness), this point lifts uniquely to a map 
$X \to \fn/(B \times \bG_m)$, where the underlying
map $X \to \bB \bG_m$ is given by the canonical bundle. 
We have a projections $\fn/B \times \bG_m \to \bA^1/\bG_m$
corresponding to projections to simple coroot spaces.
By assumption, each induced map:
\[
X \to \bA^1/\bG_m
\]

\noindent sends the generic point of $X$ to the open point of $\bA^1/\bG_m$.
Equivalently, we have obtained ($\bZ$-valued) 
effective divisors $\on{disc}_i(\sP_G,\vph)$ on $X$ 
for each $i \in \sI_G$. As 
the map: 
\[
\check{\Lambda}_{G^{\ad}} \xar{\check{\lambda} \mapsto 
(\check{\lambda},\alpha_i)_{i \in \sI_G}} \oplus_{i \in \sI_G} \bZ
\]

\noindent is an isomorphism, we see that there is a unique
divisor $\on{disc}(\sP_G,\vph)$ as above
such that $(\on{disc}(\sP_G,\vph),\alpha_i) = \on{disc}_i(\sP_G,\vph)$.

Finally, we observe that for $(\sP_G,\vph) \in \o{\Nilp}^{\check{\lambda}}$,
we can explicitly compute the above invariants. Specifically, 
$c_1(\sP_G,\vph)$ is the image of $\check{\lambda}$ in $\pi_1^{\on{alg}}(G)$.
Second, we have:
\begin{equation}\label{eq:disc-rho}
\deg(\on{disc}(\sP_G,\vph)) = \overline{\check{\lambda}} + (2g-2)\check{\rho} \in 
\check{\Lambda}_{G^{\ad}}^+.
\end{equation}

\noindent Here $\overline{\check{\lambda}} \in \check{\Lambda}_{G^{\ad}}$ is the
image of $\check{\lambda}$ under the natural map $\check{\Lambda} \to 
\check{\lambda}_{G^{\ad}}$.

\begin{rem}\label{r:disc-c1}

As $\check{\Lambda} \to \pi_1^{\on{alg}}(G) \times \check{\Lambda}_{G^{\ad}}$
is injective,\footnote{It is an isomorphism on tensoring with $\bQ$, and
$\check{\Lambda}$ is of course torsion-free.} we can recover
the invariant $\check{\lambda}$ of $(\sP_G,\vph)$ from $c_1(\sP_G,\vph)$
and $\deg(\on{disc}(\sP_G,\vph))$.

\end{rem}

\begin{rem}

The discrepancy divisor 
is a more natural indexing tool that $\check{\lambda}$ itself.
For example, for $\check{\lambda} \in \check{\Lambda}$ to be relevant
is equivalent to saying $\overline{\check{\lambda}} + (2g-2)\check{\rho}$
is a dominant coweight for $G^{\ad}$. 

\end{rem}

\begin{rem}\label{r:disc-lifts}

Observe that $\deg(\on{disc}(\sP_G,\vph))$ always lies in the
image of $\check{\Lambda} \to \check{\Lambda}_{G^{\ad}}$;
indeed, this follows from \eqref{eq:disc-rho} as
$(2g-2)\check{\rho} = (g-1)\cdot 2\check{\rho}$ lifts. 

\end{rem}

\begin{example}

In the setting of Example \ref{e:nilp-gl2}, $c_1(\sE,\vph)$ is the degree
of $\sE$, while the discrepancy divisor is the divisor of zeroes of 
the map $\ol{\vph}$.

\end{example}

\subsubsection{Everywhere regular Higgs fields and the Kostant component}\label{sss:kost-defin}

We let $\o{\Nilp}^{\on{reg}}$ denote the mapping space:
\[
\Maps(X,\o{\sN}/G).
\]

\noindent Clearly $\o{\Nilp}^{\on{reg}} \subset \o{\Nilp}$ is open.

Let $Z_f \subset G$ denote the stabilizer subgroup of 
some regular nilpotent $f \in \o{\sN}$. Because $G$ acts
transitively on $\o{\sN}$, we have 
$\o{\Nilp}^{\on{reg}} = \Bun_{Z_f}$. Clearly the center
$Z_G$ of $G$ embeds into $Z_f$; recall that $Z_G$ maps isomorphically 
onto the reductive quotient of $Z_f$; therefore, $\o{\Nilp}^{\on{reg}}$
is smooth and its connected components are 
$\pi_0(\o{\Nilp}^{\on{reg}}) = \Hom(\bG_m,Z_G) \subset \check{\Lambda}$.

The Kostant slice defines a base-point $f^{\on{glob}} \in \o{\Nilp}^{\on{reg}}$.
We let $\Nilp^{\Kos} \subset \o{\Nilp}^{\on{reg}}$ denote the corresponding
connected component. Explicitly, we have:
\[
\Nilp^{\Kos} = \o{\Nilp}^{(2-2g)\check{\rho}}.
\]

The point $f^{\on{glob}}$ is given by the $G$-bundle $\sP_G^{\on{can}}$
(induced from $\Omega^{\otimes\frac{1}{2}}$ via $-2\check{\rho}:\bG_m \to T \to G$)
with its natural regular nilpotent Higgs bundle.

\begin{rem}

The distinguished component $\Nilp^{\Kos}$ plays an outsized role in this work.

\end{rem}

\begin{rem}

We remark that $(\sP_G,\vph) \in \o{\Nilp}$ lies in 
$\o{\Nilp}^{\on{reg}}$ if and only if $\on{disc}(\sP_G,\vph) = 0$.
It additionally lies in $\o{\Nilp}^{\Kos}$ if 
$c_1(\sP_G,\vph) = c_1(-\check{\rho}(\Omega))$

\end{rem}

\subsection{Tempered $D$-modules}

\subsubsection{Local formalism}

Fix $x \in X(k)$ a point. Let $\sH_x^{\on{sph}}$ denote the spherical
Hecke category based at this point.

For $\sD$ be a module category for $\sH_x^{\on{sph}}$. 
We refer to \cite{dario-ramanujan} \S 2 for a definition of
the categories $\sD^{\xat}$ and $\sD^{\xtemp}$; our notation 
is the same as \cite{dario-ramanujan} \S 2.4.1, except that we
include the dependence on the point $x$ in the notation.
We remind that $\sD^{\xat} \subset \sD$ is a certain full
subcategory, and the embedding admits a left adjoint.
Then $\sD^{\xtemp}$ is the quotient $\sD/\sD^{\xat}$ in $\DGCat_{cont}$;
the projection $\sD \to \sD^{\xtemp}$ admits a fully faithful left adjoint.

\subsubsection{Global setting}

The above discussion applies in particular for $\sD = D(\Bun_G)$.

The main result of \cite{independence} asserts that $D(\Bun_G)^{\xat}$
and $D(\Bun_G)^{\xtemp}$ are independent of the choice of $x$ (in a strong sense).
Therefore, we often write $D(\Bun_G)^{\at}$ and $D(\Bun_G)^{\temp}$ to
indicate this category. 

On the other hand, we sometimes include the point $x$ in the
notation when we are performing a particular manipulation at the point. 

We refer to objects of $D(\Bun_G)^{\at}$ as \emph{anti-tempered} $D$-modules
on $\Bun_G$, and objects of $D(\Bun_G)^{\temp}$ as \emph{tempered} $D$-modules
on $\Bun_G$. 

Although we can think of $D(\Bun_G)^{\temp}$ as a subcategory of
$D(\Bun_G)$ (via the left adjoint referenced above), we generally consider
it rather as a quotient category. Roughly speaking, the quotient
functor is better behaved.

\subsection{Normalizations regarding exponential sheaves and characters}

\subsubsection{}

We let $\exp \in D(\bA^1)$ denote the exponential $D$-module, normalized
to live in cohomological degree $-1$, i.e., the same degree as 
the dualizing sheaf $\omega_{\bA^1}$. This object is a multiplicative sheaf
with respect to upper-$!$ functors.

\subsubsection{}

At various points in the text, we consider characters $\psi:N \to \bG_a$, or
loop group analogues with $N$ replaced by $\o{I}$ (the radical of
Iwahori) or the loop group $N(K)$. These are always assumed to be
\emph{non-degenerate} in the appropriate sense. Precisely:

\begin{itemize}

\item For $N$, the map $\Lie(\psi):\fn \to k$ should send each Chevalley generator
$e_i \in \fn$ ($i \in \sI_G$) to a non-zero number.

\item For $\o{I}$, the map $\o{I} \to \bG_a$ should be the composition
of the previous non-degenerate character $N \to \bG_a$ with
the projection $\o{I} \to N$ (so this is non-degenerate in the
standard Whittaker sense, but not the affine Kac-Moody sense).

\item For $N(K)$, the map $N(K) \to \bG_a$ should have conductor $0$ and
be non-degenerate in the standard sense.

\end{itemize}

For $\sY$ with a $G$-action, we write $D(\sY)^{N,\psi} \subset D(\sY)$ to mean
the category of $D$-modules that are twisted $N$-equivariant
against $\psi^!(\exp)$; similarly for $\o{I}$ and $N(K)$.

\part{Singular support and temperedness}

\section{Irregular singular support in finite dimensions}\label{s:irreg-finite}

In this section, we study a version of irregular singular support 
for $G$-spaces. Our main result is Theorem \ref{t:g-irreg}, which 
we prove by reduction to \cite{losev-holonomic}.

This material is used in \S \ref{s:irreg-bun-g}, and may safely 
be skipped at the first pass.

\subsection{$G$-irregularity}

\subsubsection{}\label{sss:g-irreg}

Let $\sY$ be an algebraic stack locally of finite type and
equipped with a $G$-action.

We obtain a moment map $\mu:T^*\sY \to \fg^{\vee} \simeq \fg$.
We define: 
\[
D_{\Girreg}(\sY) \coloneqq D_{\mu^{-1}(\fg_{\irreg})}(\sY).
\]

\noindent In other words, a $D$-module $\sF$ lies in 
$D_{\Girreg}(\sY)$ if for every point $(y,\xi) \in \SS(\sF)$,
$\mu(y,\xi)$ is an irregular element of $\fg$.

We let $\Shv_{\Girreg}(\sY) \coloneqq \Shv(\sY) \cap D_{\Girreg}(\sY)$.

\begin{rem}

We use the notation $\Girreg$ here rather than simply $\irreg$
to reserve the latter for the \emph{global} context considered
in \S \ref{s:irreg-bun-g}.

\end{rem}

\subsubsection{}

We can now state:

\begin{thm}\label{t:g-irreg}

Let $\sY$ be an algebraic stack locally of finite type with a $G$-action. 
Then the Whittaker averaging
functor:

\[
\Shv_{\Girreg}(\sY)^{B^-} \to D(\sY)^{B^-} \xar{\Av_!^{\psi}} D(\sY)^{N,\psi} 
\]

\noindent is identically zero (on $\Shv_{\Girreg}(\sY)^{B^-}$).

\end{thm}

\subsection{A statement for $\fg$-modules}

Below, we fix a $G$-invariant isomorphism $\fg \simeq \fg^{\vee}$
for convenience. 

\subsubsection{Singular support}

Let $M \in \fg\mod^{\heart}$ be finitely generated.
Recall that we may choose a good filtration on $M$
(relative to the PBW filtration on $U(\fg)$);
the reduced support of its associated graded is 
a well-defined closed conical subscheme 
$\SS(M) \subset \fg^{\vee} \simeq \fg$,
which we call the \emph{singular support} of $M$. 
We remind that this construction is often instead called the
\emph{associated variety} of $M$.

\subsubsection{}

Let $Z(\fg) \subset U(\fg)$ denote the center, 
and let $\chi_0:Z(\fg) \to k$ be the homomorphism defined
by the trivial representation of $\fg$.\footnote{In what follows,
all our results about $\fg$ generalize to the case where
$\chi_0$ is replaced by any central character 
$\chi_{\lambda}:Z(\fg) \to k$. We use $\chi_0$ to simplify the
notation, and to focus on our case of interest.}
Let $U(\fg)_0 = U(\fg)/U(\fg) \cdot \Ker(\chi_0)$,
so $U(\fg)_0$ is quotient of $U(\fg)$ whose modules
of those $\fg$-modules with the same central character as the trivial 
representation. We sometimes use the notation $\fg\mod_0$ to denote
the DG category of $U(\fg)_0$-modules.

Recall that $U(\fg)_0$ carries a filtration induced from
the PBW filtration, and $\gr_{\dot} U(\fg)_0 = \Gamma(\sN,\sO_{\sN})$.
Therefore, for any $M \in \fg\mod_0^{\heart}$ finitely 
generated, we have $\SS(M) \subset \sN \subset \fg$.

\subsubsection{Whittaker localization}\label{sss:whit-loc}

Note that:
\[
\TwoHom_{G\mod}(\fg\mod,D(G)^{N,\psi}) \simeq D(G)^{(N,\psi),(G,w)}
= \fg\mod^{N,\psi} \simeq Z(\fg)\mod
\]

\noindent where in the last equality we have used
Skryabin's theorem. We let:
\[
\Loc^{\psi}:\fg\mod \to D(G)^{N,\psi}
\]

\noindent denote the $G$-equivariant functor corresponding to 
$Z(\fg) \in Z(\fg)\mod$ under the above identification.
It is not hard to see that $\Loc^{\psi}$ is $t$-exact up to 
shift, but we do not use this result below.

We also use a variant of this construction; 
let:
\[
\Loc_0^{\psi}:\fg\mod_0\to D(G)^{N,\psi}
\]

\noindent denote the composition:
\[
\fg\mod_0 \to \fg\mod \xar{\Loc^{\psi}} D(G)^{N,\psi}.
\]

\noindent Equivalently, this is the $G$-equivariant functor
corresponding to:
\[
k \in \Vect \simeq \fg\mod_0^{N,\psi} \simeq 
\TwoHom_{G\mod}(\fg\mod_0,D(G)^{N,\psi}).
\]

\subsubsection{}

We will show:

\begin{thm}\label{t:lie-g-irreg}

Suppose $M \in \fg\mod_0^{\heart}$ is a finite length
module with $\SS(M) \subset \sN_{\irreg}$.\footnote{Here we
remind that $\SS(M)$ is by definition reduced; so e.g.,
such an inclusion can be checked on field-valued points.}
Then $\Loc_0^{\psi}(M) = 0 \in D(G)^{N,\psi}$.

\end{thm}

\begin{rem}

Probably the theorem is true for finitely generated modules,
not simply finite length ones. The restriction to 
finite length modules corresponds to our (probably overkill)
reference to \cite{losev-holonomic}.  

\end{rem}

\subsubsection{Preliminary comments}

Our proof of Theorem \ref{t:lie-g-irreg} relies on some results
about primitive ideals. We collect some results from the literature
here.

First, we have the following theorem of Losev:
\begin{thm}[Losev, \cite{losev-holonomic}]\label{t:losev}

Let $M \in \fg\mod_0^{\heart}$ be a faithful $U(\fg)_0$-module;
i.e., suppose that the map $U(\fg)_0 \to \End_{\Vect^{\heart}}(M)$
is injective. Then $\SS(M) \cap \o{\sN} \neq \emptyset$. 

\end{thm}

Indeed, this is a special case of 
\cite{losev-holonomic} Theorem 1.1 (1). 

\begin{rem}

In fact, it is a \emph{very} special case. It should 
be possible to obtain a more elementary argument for 
Theorem \ref{t:losev} (in our special case), but we were unable to 
do so.

\end{rem}

\subsubsection{}

Next, let $\bM_0^{\psi} \in \fg\mod_0^{\heart}$ be the object
corresponding to:
\[
k \in \Vect \simeq \fg\mod_0^{N,\psi} \subset \fg\mod_0.
\]

\noindent Explicitly, we have:
\[
\bM_0^{\psi} = \ind_{\fn}^{\fg}(\psi) \underset{Z(\fg)}{\otimes} k.
\]

\noindent Here $\psi \in \fn\mod^{\heart}$ denotes the 1-dimensional
module defined by the character $\psi$, and we consider
$k$ as a $Z(\fg)$-module via $\chi_0$.

We recall the following basic fact:

\begin{prop}[Kostant, \cite{kostant-whittaker} Theorem D]\label{p:whit-faithful}

$\bM_0^{\psi}$ is a faithful $U(\fg)_0$-module.

\end{prop}

We deduce:

\begin{cor}\label{c:whit-faithful}

Any non-zero object of $\fg\mod_0^{(N,\psi),\heart}$ is a faithful
$U(\fg)_0$-module.

\end{cor}

Indeed, by Skryabin's equivalence, such an object has the form $V \otimes \bM_0^{\psi}$
for some non-zero vector space $V$, so Proposition \ref{p:whit-faithful} yields
the assertion.

\subsubsection{Proof of the theorem for Lie algebras}

With the preliminary results above completed, we are now in a 
position to prove the theorem.

\begin{proof}[Proof of Theorem \ref{t:lie-g-irreg}]

\step We are obviously reduced to the case where
$M$ is a simple module (with $\SS(M) \subset \sN_{\irreg}$). 

Let $I \subset U(\fg)_0$ be the
two-sided ideal annihilating $M$. 
By Losev's theorem (Theorem \ref{t:losev}), $I \neq 0$. 
Set $A = U(\fg)_0/I$; this is a classical associative algebra.

\step Note that $A$ receives a Lie algebra homomorphism $i:\fg \to A$
via $\fg \to U(\fg) \to U(\fg)_0 \to A$. Moreover, the
adjoint\footnote{I.e., the action defined by the formula $\xi \star a \coloneqq
[i(\xi),a]$ for $\xi \in \fg$ and $a \in A$.} 
action of $\fg$ on $A$ 
integrates to an action of $G$ on $A$; indeed, this follows from
the corresponding property of $U(\fg)$.

In other words, $A$ receives a canonical Harish-Chandra datum for
$G$.

It follows that there is a (strong) $G$-action on $A\mod$ such that
the forgetful functor $A\mod \to \fg\mod_0$ is $G$-equivariant.

\step 

Next, we claim that $A\mod^{N,\psi} = 0$.

It suffices to check this at the abelian categorical level
(the $t$-structure on $A\mod^{N,\psi}$ is obviously separated).

Any object of $A\mod^{\heart}$ maps to a non-faithful module
of $U(\fg)_0$ (as $I \neq 0$). Therefore, any 
object of $A\mod^{N,\psi,\heart}$ maps to 
an object of $\fg\mod_0^{N,\psi,\heart}$ that is
not faithful as a $U(\fg)_0$-module; by Corollary \ref{c:whit-faithful},
the object must be zero.

\step Finally, we note that (as in \S \ref{sss:whit-loc}),
we have:
\[
\TwoHom_{G\mod}(A\mod,D(G)^{N,\psi}) \simeq A\mod^{N,\psi} = 0.
\]

\noindent Therefore, the composition:
\[
A\mod \to \fg\mod_0 \xar{\Loc_0^{\psi}} D(G)^{N,\psi}
\]

\noindent is zero. 

As $M \in \fg\mod_0^{\heart}$ lifts to $A\mod^{\heart}$ by definition, 
the result follows.

\end{proof}

\subsection{Proof of Theorem \ref{t:g-irreg}}

We now prove the theorem. We proceed by steps.

Below, we fix $\sF \in \Shv_{\Girreg}(\sY)^{B^-}$.

\subsubsection{Reduction to the quasi-compact case}

Note that the stack $\sY/G$ is a union of its 
quasi-compact open substacks. Replacing $\sY$ by the preimage
of each such open, we are reduced to the case where
$\sY$ itself is quasi-compact.

\subsubsection{Reduction to the simple case}

First, recall from \cite{bbm} (see also \cite{b-function} Appendix A) 
that the functor:
\[
\Av_!^{\psi}:D(\sY)^{B^-} \to D(\sY)^{N,\psi}
\]

\noindent is $t$-exact up to shift. Therefore, as 
$\Shv_{\Girreg}(\sY)^{B^-} \subset D(\sY)^{B^-}$ is closed under
truncations, we may assume $\sF \in \Shv_{\Girreg}(\sY)^{B^-,\heart}$.

Moreover, as $\Shv_{\Girreg}(\sY)^{B^-,\heart} \subset \Shv(\sY)^{B^-,\heart}$
is stable under taking subobjects, we may assume $\sF$ is also perverse
(i.e., ``small") by quasi-compactness of $\sY$. In particular, $\sF$ has finite length. 
Finally, we are clearly reduced to the case where 
$\sF \in \Shv_{\Girreg}(\sY)^{B^-,\heart}$ is simple; 
we explicitly remark that it is equivalent to say that
it is simple in $D(\sY)^{B^-,\heart}$.

\subsubsection{Reduction to the case where $\sY = G \times \sY_0$ where
$G$ acts trivially on $\sY_0$}

Consider $G \times \sY$ as acted on by $G$ via the action on the
first factor alone. 
Next, we have the $G$-equivariant map:
\[
\begin{tikzcd}
G \times \sY \arrow[r,"\act"] & \sY.
\end{tikzcd} 
\]

\noindent As $\act$ is smooth, $\act^![-\dim G]$ 
is $t$-exact, conservative, preserves simples, and maps 
$\Shv_{\Girreg}(\sY)$ to $\Shv_{\Girreg}(G \times \sY)$.
By $G$-equivariance, $\act^!$ commutes with finite Whittaker averaging.

Therefore, we are clearly reduced to case where $\sY$ has the stated form.

\subsubsection{Reduction to the case $\sY = G$}

Next, observe that $\Av_!^{\psi}(\sF)$ is a compact, holonomic 
object (because $\sF$ is).
Therefore, it suffices to check that for every field extension
$k^{\prime}/k$ of finite degree and every pair 
$(g,y) \in (G \times \sY_0)(k^{\prime})$, the
$!$-restriction of $\Av_!^{\psi}(\sF)$ to $(g,y)$ vanishes.
Up to a finite extension fo the ground field, we may assume $k^{\prime} = k$;
we do so in what follows to simplify the notation. 

Let $y \in \sY_0(k)$ be fixed. Let $i_y:\Spec(k) \to \sY_0$ be
the corresponding map. The map $G \xar{\id \times i_y} G \times \sY_0$
is obviously $G$-equivariant, so we have:
\[
(\id\times i_y)^!\Av_!^{\psi}(\sF) = 
\Av_!^{\psi}(\id\times i_y)^!(\sF) \in D(G)^{N,\psi}. 
\] 

Now we claim:
\begin{lem}\label{l:irreg-functoriality}

The object $(\id\times i_y)^!(\sF) \in \Shv(G)^{B^-}$ lies
in $\Shv_{\Girreg}(G)^{B^-}$. 

\end{lem}

Clearly this result follows from the general result:

\begin{prop}[Kashiwara-Schapira]

Let $\sZ_1$ and $\sZ_2$ be smooth schemes, and let
$\Lambda \subset T^*\sZ_1$ be closed and conical. Let $f:\sZ_3 \to \sZ_2$
be a given map with $\sZ_3$ smooth (but $f$ possibly non-smooth).

Then the functor: 
\[
(\id \times f)^!:\Shv(\sZ_1 \times \sZ_2) \to 
\Shv(\sZ_1 \times \sZ_3)
\]

\noindent maps $\Shv_{\Lambda \times T^*\sZ_2}(\sZ_1 \times \sZ_2)$ to 
$\Shv_{\Lambda \times T^*\sZ_3}(\sZ_1 \times \sZ_3)$.

\end{prop}

\begin{proof}

This follows immediately from the estimates\footnote{We fill in some
details in the reference here, referring to \cite{ks} for notation.

In our setting, for $\sG \in \Shv_{\Lambda \times T^*\sZ_2}(\sZ_1 \times \sZ_2)$,
recall that \cite{ks} Corollary 6.4.4 (ii) bounds $\SS((\id \times f)^!(\sG))$
by something denoted $(\id \times f)^{\sharp}( \Lambda \times T^*\sZ_2)$.
It is straightforward to see $(\id \times f)^{\sharp}( \Lambda \times T^*\sZ_2)
\subset \Lambda \times T^*\sZ_2$; e.g., one can reduce to the
case where $f$ is a closed embedding and then the description 
from \cite{ks} Remark 6.2.8 (i) is convenient.}
on singular support of pullbacks from \cite{ks} Corollary 6.4.4.

\end{proof}

\begin{rem}

One may also appeal to \cite{ginzburg-inventiones} in place 
of \cite{ks}. In addition, we remark that the proof of 
\cite{bellamy-ginzburg} Proposition 2.1.2 essentially 
(up to a quasi-projectvity assumption) shows more generally
that $\Shv_{\Girreg}$ is preserved under sheaf-theoretic operations 
for $G$-equivariant maps. Indeed, \emph{loc. cit}. considers
a similar setting to ours, but with $0 \in \fg$ replacing
$\sN_{\irreg} \subset \fg$; more generally, the argument in \emph{loc. cit}.
applies for any closed, conical, $G$-equivariant subscheme
of $\fg$.

\end{rem}

\begin{rem}

We remind that the results from \cite{ks} and \cite{ginzburg-inventiones} 
fail for general holonomic $D$-modules; regularity is crucial assumption.

\end{rem}

We now continue with the reduction: by Lemma \ref{l:irreg-functoriality},
$(\id\times i_y)^!(\sF) \in \Shv_{\Girreg}(G)^{B^-}$; therefore,
if we know the result for $\sY_0 = \Spec(k)$, we obtain
$\Av_!^{\psi}(\id\times i_y)^!(\sF) = 0$, which yields the claim for
general $\sY_0$ by our earlier discussion.

\subsubsection{Proof for $\sY_0 = \Spec(k)$}

Roughly speaking, the idea is to reduce to Theorem \ref{t:lie-g-irreg}
via Beilinson-Bernstein.

We remind the setting: $\sF \in D(B^-\backslash G)^{\heart}$ 
is a simple\footnote{At this point, 
regularity, and even holonomicity, are no longer essential hypotheses.}
$D$-module on the flag variety with singular support in:
\[
\widetilde{\sN}_{\irreg} \coloneq \widetilde{\sN} \underset{\sN}{\times} \sN_{\irreg}
\subset \widetilde{\sN} \simeq T^*(B^-\backslash G).
\]

\noindent where $\widetilde{\sN}$ is the Springer resolution of the
nilpotent cone. We wish to show that:
\[
\Av_!^{\psi}(\sF) \in D(G)^{N,\psi}
\]

\noindent is zero (where $(N,\psi)$ invariants are taken for the left action).

We have a diagram:
\[
\begin{tikzcd}
D(B^-\backslash G)
\arrow[d,swap,"{\Gamma(B^-\backslash G,-)}"]
\arrow[dr,"\Av_!^{\psi}"]
\\
\fg\mod_0
\arrow[r,"\Loc^{\psi}"] 
& 
D(G)^{N,\psi}. 
\end{tikzcd}
\]

\noindent that commutes up to a cohomological shift. 
Indeed, by $G$-equivariance of the functors involved,
it suffices to check that the images of the $\delta$ $D$-module
at the base point of the flag variety are mapped to the
same object (up to shift), and this is straightforward. 

Now for our $\sF$, let $M \coloneqq \Gamma(B^-\backslash G,\sF) \in \fg\mod_0$.
By the above, it suffices to show that $\Loc^{\psi}(M) = 0$.

By Beilinson-Bernstein localization \cite{bb-loc}, 
$M \in \fg\mod_0^{\heart}$ is simple. Therefore, 
by Theorem \ref{t:lie-g-irreg}, it suffices to show that
$\SS(M) \subset \sN_{\irreg}$.
This is a 
standard compatibility: by the Corollary in \S 1.9 of \cite{borho-brylinski}, 
$\SS(M)$ is the image of $\SS(\sF)$ along the projection map:
\[
\widetilde{\sN} \to \sN.
\]

\section{Irregular singular support on $\Bun_G$}\label{s:irreg-bun-g}

\subsection{Formulation of the main result}\label{ss:d-irreg-defin}

Let $D_{\irreg}(\Bun_G) \coloneqq D_{\Higgs_{G,\fg_{\irreg}}}(\Bun_G)$;
see \S \ref{sss:globalization} for the notation.
We similarly have $D_{\Nilp_{\irreg}}(\Bun_G) = \Shv_{\Nilp_{\irreg}}(\Bun_G)$.

The purpose of this section is to prove:

\begin{thm}\label{t:at}

Any object of $\Shv_{\Nilp_{\irreg}}(\Bun_G)$ is anti-tempered.

\end{thm}

We will prove this result by reduction to 
Theorem \ref{t:g-irreg}.

Throughout, we remind that we have fixed a point $x \in X(k)$.
We let $K$ denote the field of Laurent series based at $x$ and
$O \subset K$ the subring of Taylor series. We sometimes use
$t$ for a coordinate at $x$. We let e.g.
$G(K)$ and $G(O)$ denote the loop and arc groups for $G$.
We freely use the formalism of (strong) loop group actions 
on objects of $\DGCat_{cont}$.

\begin{rem}

Using the full results of this paper, we were able to show
$D_{\irreg}(\Bun_G) \subset D(\Bun_G)^{\at}$. We conjecture
that this is an equivalence for any $G$.
This is a folklore statement for $G = PGL_2$; see \cite{dario-ramanujan}
Corollary 4.2.6 for a variant of the assertion in this case.

\end{rem}

\subsection{Anti-temperedness and Whittaker averaging}\label{ss:at-whit}

We have the following general characterization of anti-temperedness.

Suppose the loop group $G(K)$ acts on $\sC \in \DGCat_{cont}$. Recall that we have
the anti-tempered subcategory $\sC^{G(O),\xat} \subset \sC^{G(O)}$, cf. \S \ref{ss:temp-conventions}. 

By \cite{dario-ramanujan} Theorem 1.4.8, we have:

\begin{lem}\label{l:at-whit}

$\sC^{G(O),\xat}$ is the kernel of the Whittaker $!$-averaging functor:
\[
\sC^{G(O)} \xar{\Av_!^{\psi}} \Whit(\sC).
\]

\end{lem}

\begin{rem}

We make a brief philosophical point. Ultimately, we are interested
in Whittaker coefficients of tempered $D$-modules on $\Bun_G$. 
If we think heuristically of $\Bun_G$ as a double quotient
$G(k(X))\backslash G(\bA)/G(\bO)$, this involves integrating
on the \emph{left} with respect to $N(\bA)$ (twisted by a character, of course).
On the other hand, temperedness involves the derived Satake action, 
which occurs \emph{on the right}. Therefore, when we apply the
present lemma to $\sC = D(\Bun_G^{\on{level},x})$, the Whittaker averaging
in question should be thought of as occurring on the \emph{right}.
Moreover, the Whittaker integral on the right occurs at a single point,
while the one on the left involves all points simultaneously. 

Ultimately, the reader may think that singular support translates
between left and right Whittaker conditions.

\end{rem}

\subsection{Baby Whittaker}\label{ss:baby-at}

Next, we record the following well-known result.

Let $I \subset G(O)$ (resp. $I^-$) denote the Iwahori subgroup corresponding to 
$B \subset G$ (resp. $B^- \subset G$). Let $\o{I}$ denote the prounipotent radical
of $I$. Let $\o{I}_1 \coloneqq \Ad_{-\check{\rho}(t)}(\o{I})$ 
and let $I_1^- \coloneqq \Ad_{-\check{\rho}(t)}(I^-)$. We abuse
notation in letting $\psi$ denote the restriction of the canonical character
of $N(K)$ to $\o{I}_1$. Finally, we let $\sK_1$ denote $\Ad_{-\check{\rho}(t)}$
applied to the first congruence subgroup of $G(O)$; note that
$\sK_1 = \o{I}_1 \cap I_1^-$. Note that
$\o{I}_1/\sK_1 = N$ and $I_1^-/\sK_1 = B^-$.

Following \cite{whit}, for $\sC \in G(K)\mod$, we use the notation
$\Whit(\sC) \coloneqq \sC^{N(K),\psi}$ and 
$\Whit^{\leq 1}(\sC) \coloneqq \sC^{\o{I}_1,\psi}$. 
We let $\iota_{1,\infty}^!:\Whit(\sC) \to \Whit^{\leq 1}(\sC)$
denote the $*$-averaging functor.

Finally, we can state:

\begin{lem}\label{l:baby}

Let $\sC \in G(K)\mod$ be given. Then the functor
$\iota_{1,\infty}^!:\Whit(\sC) \to \Whit^{\leq 1}(\sC)$
admits a fully faithful left adjoint $\iota_{1,\infty,!}$; moreover, there is a 
commutative diagram:
\[
\begin{tikzcd}[column sep = small]
\sC^{G(O)} 
\arrow[rr,"\Oblv"] 
\arrow[drrrrr,swap,"\Av_!^{\psi}"]
&&
\sC^{I} \arrow[rr,"\Av_*"]
&&
\sC^{I_1^-}
\arrow[rr,equals]
&&
(\sC^{\sK_1})^{B^-}
\arrow[rr,"\Av_!^{\psi}"]
&&
(\sC^{\sK_1})^{N,\psi}
\arrow[rr,equals]
&&
\Whit^{\leq 1}(\sC)
\arrow[dlllll,"\iota_{1,\infty,!}"]
\\
&&&&&\Whit(\sC).
\end{tikzcd}
\]

\end{lem}

\begin{proof}

The existence of the fully faithful functor $\iota_{1,\infty,!}$ is
a special case of \cite{whit} Theorem 2.3.1. 
The existence of the commutative diagram is standard:
it follows from the fact that the functor
$\Av_*:\sC^{I^-} \to \sC^{I_1^-}$ is an equivalence (with inverse
the natural $!$-averaging functor). 

\end{proof}

Combining this result with Lemma \ref{l:at-whit}, we find that
$\sC^{G(O),\xat}$ is the kernel of the composition:
\begin{equation}\label{eq:at-ker}
\sC^{G(O)} \xar{\Av_*} \sC^{I_1^-} \xar{\Av_!^{\psi}} \sC^{\o{I}_1,\psi}.
\end{equation}

\subsection{Parahoric bundles}

Let $P \subset G(K)$ be a compact open subgroup. Let
$\Bun_G^{P\level} \coloneqq \Bun_G^{\on{level},x}/P$. For example,
$\Bun_G^{G(O)\level} = \Bun_G$. We remark that $\Bun_G^{P\level}$ is always
an Artin stack locally almost of finite type.

\begin{rem}

We care about exactly four cases: when $P = G(O)$, $I^-$, $I_1^-$,
or $\sK_1$. If $\check{\rho}$ is an (integral) coweight for $G$, e.g. the subgroups
$I^-$ and $I_1^-$ are conjugate, so we can identify
$\Bun_G^{I^-\level}$ and $\Bun_G^{I_1^-\level}$; in general, the latter
can be thought of as a twisted form of the former, with only a mild technicality
separating them.

\end{rem}

\subsubsection{Ramified Higgs bundles}

Let $P$ be a compact open subgroup as above.
Let $(\sP_G,\tau) \in \Bun_G^{P\level}$ be a point (the notation indicates
that $\sP_G$ is a $G$-bundle on $X \setminus x$ and $\tau$ is a reduction 
of the $G(K)$-bundle $\sP_G|_{\o{\sD}_x}$ to $P$).

The standard Serre duality argument shows that the cotangent space
$T_{(\sP_G,\tau)}^* \Bun_G^{P\level}$ consists of Higgs bundle
$\vph \in \Gamma(X\setminus x,\fg_{\sP_G} \otimes \Omega_X^1)$
satisfying the condition: 

\[
(\sP_G,\tau,\vph|_{\o{\sD}_x}) 
\in \Lie(P)^{\perp}/P \subset \fg((t))dt/P = \fg((t))^{\vee}/P.
\]

\noindent We refer to such data as a \emph{($P$-)ramified Higgs bundles}.

\subsubsection{}\label{sss:ram-irreg}

We say that such a ramified Higgs bundle is \emph{irregular} if the Higgs bundle 
$(\sP_G|_{X\setminus x},\vph)$ on $X\setminus x$ is so. Note that
irregular ramified Higgs bundles form a closed conical
substack of $T^*\Bun_G^{P\level}$.
Therefore, we obtain the categories $D_{\irreg}(\Bun_G^{P\level})$
and $\Shv_{\irreg}(\Bun_G^{P\level})$ as in \S \ref{ss:d-irreg-defin}.

\subsubsection{Compatibility with the Hecke action}\label{sss:hecke-irreg}

Now suppose $P_1,P_2 \subset G(K)$ are a pair of parahoric subgroups.
We can form the category $D(P_1\setminus G(K)/P_2)$ of
$(P_1,P_2)$-equivariant $D$-modules on $G(K)$. 
There is a natural convolution action:
\[
D(P_1 \backslash G(K)/P_2) \otimes D(\Bun_G^{P_2\level}) \to 
D(\Bun_G^{P_1\level}).
\]

\begin{prop}\label{p:parahoric-irreg}

The above convolution functor induces a functor:
\[
D(P_1 \backslash G(K)/P_2) \otimes D_{\irreg}(\Bun_G^{P_2\level}) \to 
D_{\irreg}(\Bun_G^{P_1\level}).
\]

\end{prop}

In other words, in the parahoric setting, convolution preserves
irregular singular support.

The proof of the proposition is identical to the proof of
the Nadler-Yun theorem in \cite{toy-model} Theorem B.5.2;
we particularly refer to \emph{loc. cit}. \S B.6.6. We remark
that the argument\footnote{To be clear: we mean when the point is
fixed, as in \S B.6.6 in \emph{loc. cit}. What follows there, regarding what
is denoted $\xi_X$ in \emph{loc. cit}. and concerns variation in the point $x$, 
crucially uses nilpotence. But this is irrelevant for our purpose.}
in \emph{loc. cit}. uses nothing about nilpotence. To be completely
explicit: the argument in \cite{toy-model} applies
for any closed $G$-invariant conical subscheme $\Lambda \subset \fg$, where in 
\emph{loc. cit}. $\Lambda = \sN$, and for us here, $\Lambda = \fg_{\irreg}$.

Finally, we remark that only $P_2$ needs to be parahoric; we need
ind-properness of $G(K)/P_2$ to control the singular support of the convolution.

\subsection{Mixing the ingredients}

We now prove Theorem \ref{t:at}.

\subsubsection{Step 1}

By \eqref{eq:at-ker}, our goal is to show that the composition:
\[
\Shv_{\Nilp_{\irreg}}(\Bun_G) \subset D(\Bun_G) \to D(\Bun_G^{I_1^-\level}) = 
D(\Bun_G^{\sK_1\level})^{B^-}
\to D(\Bun_G^{\sK_1\level})^{N,\psi}
\]

\noindent is zero. 

By Proposition \ref{p:parahoric-irreg}, the first functor maps
$\Shv_{\Nilp_{\irreg}}$ into $\Shv_{\irreg}(\Bun_G^{I_1^-\level})$.
(The right hand side may be replaced with $\Shv_{\Nilp_{\irreg}}$ just as well,
but this level of precision is not needed below.)

\subsubsection{Step 2}

Consider $\Bun_G^{\sK_1\level}$ as an algebraic stack acted on by
$G$. The moment map:
\[
\mu:T^*\Bun_G^{\sK_1\level}
\to \fg^{\vee} \simeq \fg
\]

\noindent sends a ramified Higgs bundle $(\sP,\tau,\vph)$
to $t\Ad_{\check{\rho}(t)}(\vph) \on{mod }t$, where we note that
$\vph \in t^{-1}\Ad_{-\check{\rho}(t)}\fg[[t]]dt/\sK_1$.
Because irregularity is a closed condition, we see that any 
irregular $\vph$ maps into $\fg_{\irreg}$.

Therefore, it follows that we have an embedding:
\[
\Shv_{\irreg}(\Bun_G^{\sK_1\level}) \subset \Shv_{\Girreg}(\Bun_G^{\sK_1\level}).
\]

\noindent Here for clarity, we say that the left hand 
side is defined in \S \ref{sss:ram-irreg}, and the
right hand side is defined using the $G$-action as in \S \ref{sss:g-irreg}.

\subsubsection{Step 3}

We now conclude the argument.

By the above, we have a composition:
\[
\Shv_{\Nilp_{\irreg}}(\Bun_G) \to \Shv_{\irreg}(\Bun_G^{\sK_1^-\level})^{B^-} 
\subset  \Shv_{\Girreg}(\Bun_G^{\sK_1^-\level})^{B^-} 
\to
D(\Bun_G^{\sK_1\level})^{N,\psi}
\]

\noindent and the latter functor is zero by Theorem \ref{t:g-irreg}.
This concludes the argument.

\part{Microlocal properties of Whittaker coefficients}

\section{Background on coefficient functors}\label{s:coeff-backround}

In this section, we review the classical geometric theory
of Whittaker coefficients from \cite{fgv} and establish notation.
This section houses no new results.

\subsection{Moduli spaces}

\subsubsection{}

We fix $\Omega_X^{\frac{1}{2}}$ a square root of the canonical sheaf on $X$.
We obtain $\check{\rho}(\Omega_X^1) \coloneqq (2\check{\rho})(\Omega_X^{\frac{1}{2}}) 
\in \Bun_T$.

We let $\Bun_N^{\Omega}$ denote the fiber product:
\[
\Bun_N^{\Omega} \coloneqq \Bun_B \underset{\Bun_T}{\times} \Spec(k)
\]

\noindent where $\Spec(k) \to \Bun_T$ is $\check{\rho}(\Omega_X^1)$.

\subsubsection{}

There is a standard character:
\[
\psi:\Bun_N^{\Omega} \to \bA^1.
\]

\noindent We refer to \cite{fgv} \S 4.1.3 for its definition.

\subsubsection{}

We also denote the canonical projection by:

\[
\fp:\Bun_N^{\Omega} \to \Bun_G.
\]

\subsubsection{}

More generally, let $D$ be a $\check{\Lambda}$-valued divisor on $X$.

We have a corresponding point $\sO_X(D) \in \Bun_T$: it is characterized
by the fact that for every weight $\lambda:T \to \bG_m$, 
$\lambda(\sO_X(D)) = \sO_X(\lambda(D))$, where we note that
$\lambda(D)$ is a usual ($\bZ$-valued) divisor on $X$.

For a $T$-bundle $\sP_T$, we let $\sP_T(D)$ denote the image of
$(\sP_T,\sO_X(D))$ under the multiplication 
$\Bun_T \times \Bun_T \to \Bun_T$.

We let $\Bun_N^{\Omega(D)}$ denote the fiber product:
\[
\Bun_N^{\Omega(D)} \coloneqq \Bun_B \underset{\Bun_T}{\times} \Spec(k)
\]

\noindent where this time we use the $T$-bundle 
$\check{\rho}(\Omega_X^1)(D)$.

\subsubsection{}

For $D$ a $\check{\Lambda}^+$-valued divisor, there is a canonical character:
\[
\psi_D:\Bun_N^{\Omega(-D)} \to \bA^1.
\]

\noindent We again refer to \cite{fgv} \S 4.1.3 for its definition.

\subsubsection{}

In this context, we let: 
\[
\fp_D:\Bun_N^{\Omega(-D)} \to \Bun_G
\]

\noindent denote the canonical projection. 

\subsubsection{Large divisors}\label{sss:large-div}

For later use, we record the following bounds.

\begin{defin}

We say $D$ is \emph{sufficiently large} if $\deg(D) \in \check{\Lambda}^+$
satisfies:
\begin{equation}\label{eq:deg-big}
(\deg(D),\alpha) > g-1
\end{equation}

\noindent for every positive\footnote{More economically,
this condition for $\alpha$ simple obviously implies it for 
$\alpha$ positive.} root $\alpha>0$.
(Here $g$ is the genus of $X$.)

\end{defin}

By reduction to the case of $\bG_a$, it is immediate that:

\begin{lem}\label{l:large-div}

For $D$ sufficiently large, $\Bun_N^{\Omega(-D)}$ is an affine scheme.

\end{lem}

\subsection{Coefficient functors}

\subsubsection{The primary Whittaker coefficient}

We define the \emph{primary} Whittaker coefficient functor is the functor:
\[
\coeff:D(\Bun_G) \to \Vect
\]

\noindent defined by:
\[
\sF \mapsto 
C_{\dR}\big(\Bun_N^{\Omega},\fp^!(\sF) \overset{!}{\otimes} \psi^!(\exp)\big)
[-\dim  \Bun_N^{\Omega}].
\]

\begin{rem}

We include the above shift in the definition to make various formulae
work out more nicely. 

\end{rem} 

\begin{rem}\label{r:first?}

The above functor is usually called the \emph{first} Whittaker coefficient.
The terminology is borrowed from modular forms, where the above
corresponds to the term $a_1$, cf. \S \ref{sss:modular}. This multiplicative 
normalization
can be confusing in the geometric context, where \emph{zeroth} Whittaker
coefficient would be the more natural convention (since we index
by divisors rather than their norms).

\end{rem}

\subsubsection{Other Whittaker coefficients}

Now suppose $D$ is a $\check{\Lambda}^+$-valued divisor on 
$X$. Define the functor:
\[
\coeff_D: D(\Bun_G) \to \Vect
\]

\noindent by the formula:
\[
\sF \mapsto 
C_{\dR}\big(\Bun_N^{\Omega(-D)},\fp_D^!(\sF) \overset{!}{\otimes} \psi_D^!(\exp)\big)
[-\dim \Bun_N^{\Omega(-D)}].
\]

\subsection{The Casselman-Shalika formula}\label{ss:cs}

We now recall the primary classical result in the subject.

Let $D$ be a $\check{\Lambda}^+$-valued divisor on $X$ as above. 
There is an associated object $\sV^D$ of $\Rep(\check{G})_{\Ran}$;
if $D = \sum \check{\lambda}_i \cdot x_i$ for some finite
set of distinct points $x_i$, then 
$\sV^D = \oplus V^{\check{\lambda}_{x_i}} \otimes  \delta_{x_i}$.

We recall the following result from \cite{fgv}, which is a
geometric analogue\footnote{See \cite{fgkv} for more discussion 
of the relationship between the results of \cite{fgv} and \cite{casselman-shalika}.} 
of the Casselman-Shalika formula from \cite{casselman-shalika}.

\begin{thm}[Frenkel-Gaitsgory-Vilonen]\label{t:cs}

There is a canonical isomorphism of functors:
\[
\coeff_D \simeq \coeff(\sV^D \star -): D(\Bun_G) \to \Vect.
\]

\end{thm}

\begin{proof}

For the sake of completeness, we include an argument deducing
this result from \cite{fgv}.

As above, suppose $D = \sum_{i=1}^n \check{\lambda}_i \cdot x_i$.
Let $\bold{x}$ denote the collection of points $x_i$.
As in \cite{fgv}, we have usual the 
ind-stack:\footnote{It is defined in \cite{fgv} \S 2.3, where it would be 
denoted $\presub{\bold{x},\infty}{\Bun}_N^{\check{\rho}(\Omega_X^1)}$.}
\[
\ol{\Bun}_N^{\Omega,\infty\cdot \bold{x}}.
\]

\noindent
There is a natural map $\fp_{\infty\cdot \bold{x}}: 
\ol{\Bun}_N^{\Omega,\infty\cdot \bold{x}} \to \Bun_G$.
There is a natural Hecke action of $\Rep(\check{G})^{\otimes n}$
on $D(\ol{\Bun}_N^{\Omega,\infty\cdot \bold{x}})$ compatible
with the Hecke action on $\Bun_G$ (corresponding to the
points $x_1,\ldots,x_n$). We also note that $\sV^D$ can
evidently be considered as an object of $\Rep(\check{G})^{\otimes n}$.

There are natural locally closed embeddings:
\[
\begin{gathered}
\jmath:\Bun_N^{\Omega} \to \ol{\Bun}_N^{\Omega,\infty\cdot \bold{x}} \\
\jmath_D:\Bun_N^{\Omega(-D)} \to \ol{\Bun}_N^{\Omega,\infty\cdot \bold{x}}
\end{gathered}
\]

\noindent compatible with the maps to $\Bun_G$.

Let $\sV^{D,\vee} = \sV^{-w_0(D)}$ 
be the dual to $\sV^D$ in $\Rep(\check{G})^{\otimes n}$.
By \cite{fgv} Theorem 4 and Theorem 3 (2), we have:\footnote{In fact,
there is a sign choice to be made once and for all in defining the Hecke
action on $D(\Bun_G)$; roughly speaking, the difference is whether
we think of the Hecke category as acting naturally on the left or right,
where in the latter case we use the inversion automorphism to change
a right module structure to a left module structure. 
We have implicitly pinned down this choice in the statement of the formula.
By contrast, the reader who looks in \cite{fgv} will find 
statements for each of the two possible choices.}
\[
\sV^{D,\vee} 
\star \jmath_{*,dR}(\psi^!(\exp))[-\dim \Bun_N^{\Omega}] \simeq 
\jmath_{D,*,dR}(\psi_D^!(\exp))[-\dim \Bun_N^{\Omega(-D)}].
\]

Therefore, for $\sF \in D(\Bun_G)$, we obtain:
\[
\begin{gathered}
\coeff_D(\sF) \coloneqq 
C_{\dR}\big(\Bun_N^{\Omega(-D)},\fp_D^!(\sF) \overset{!}{\otimes} \psi_D^!(\exp)\big)
[-\dim \Bun_N^{\Omega(-D)}] = \\
C_{\dR}\big(\ol{\Bun}_N^{\Omega,\infty\cdot \bold{x}},
\fp_{\infty\cdot \bold{x}}^!(\sF) \overset{!}{\otimes} 
\jmath_{D,*,dR}(\psi_D^!(\exp))\big) 
[-\dim \Bun_N^{\Omega(-D)}] = \\
C_{\dR}\Big(\ol{\Bun}_N^{\Omega,\infty\cdot \bold{x}},
\fp_{\infty\cdot \bold{x}}^!(\sF) \overset{!}{\otimes} 
\big(\sV^{D,\vee} \star \jmath_{*,dR}(\psi^!(\exp))\big)\Big)
[-\dim \Bun_N^{\Omega}]
= \\
C_{\dR}\Big(\ol{\Bun}_N^{\Omega,\infty\cdot \bold{x}},
\big(\sV^{D} \star \fp_{\infty\cdot \bold{x}}^!(\sF)\big) \overset{!}{\otimes} 
\jmath_{*,dR}(\psi^!(\exp))\Big)
[-\dim \Bun_N^{\Omega}] = \\
C_{\dR}\big(\ol{\Bun}_N^{\Omega,\infty\cdot \bold{x}},
\fp_{\infty\cdot \bold{x}}^!(\sV^D \star \sF) \overset{!}{\otimes} 
 \jmath_{*,dR}(\psi^!(\exp))\big)
[-\dim \Bun_N^{\Omega}] = \\
C_{\dR}\big(\Bun_N^{\Omega},
\fp^!(\sV^D \star \sF)\big) \overset{!}{\otimes} 
 \jmath_{*,dR}(\psi^!(\exp))\big)
[-\dim \Bun_N^{\Omega}] \eqqcolon 
\coeff(\sV^D \star \sF).
\end{gathered}
\]

\end{proof}

\begin{rem}

The above assertions admit natural generalizations where the divisors
are along to vary in moduli, even over Ran space; we omit the statement
here as we do not need it.

\end{rem}

\subsection{More notation}

In the remainder of this section, we briefly introduce more notation.

\subsubsection{The Poincar\'e sheaf}

We let $\Poinc_! \in D(\Bun_G)$ be the object corepresenting 
$\coeff$. Explicitly, we have:
\begin{equation}\label{eq:poinc-!}
\Poinc_! = \fp_!((-\psi)^*(\exp[-2]))[\dim \Bun_N^{\Omega}]
\end{equation}

\noindent In other words, we take the character sheaf (or its inverse) 
on $\Bun_N^{\Omega}$,
cohomologically normalized to be perverse,
and $!$-push it forward to $\Bun_G$; this makes sense by holonomicity.

\begin{convention}

We use the subscript $!$ to remind that a lower-$!$ functor appears
in the formation of $\Poinc_!$. We use similar notation in related
settings without further mention. We remark that $\coeff$ would
be denoted $\coeff_*$ in this regime; we omit the $*$ for brevity, given
how often the $\coeff$ functor is used in this work.

\end{convention}

\subsubsection{The $!$-coefficient functor}

Let $D_{\hol}(\Bun_G) \subset D(\Bun_G)$ 
denote the category of (ind-)holonomic $D$-modules on $\Bun_G$,
i.e., $D$-modules $\sF \in D(\Bun_G)$ such that for
any $\pi:S \to \Bun_G$ with $S$ affine, $\pi^!(\sF) \in D(S)$ is (ind-)holonomic.
We remark that $\Shv_{\Nilp}(\Bun_G) \subset D_{\hol}(\Bun_G)$.

Then we have a functor:
\[
\begin{gathered}
\coeff_!:D_{\hol}(\Bun_G) \to \Vect \\
\sF \mapsto 
C_{\dR,c}\big(\Bun_N^{\Omega},\fp^*(\sF) \overset{*}{\otimes} \psi^*(\exp[-2])\big)
[\dim  \Bun_N^{\Omega}].
\end{gathered}
\]

In other words, $\coeff_!$ is the Verdier conjugate to $\coeff$. That is,
we have:

\begin{lem}\label{l:coeff-!-verdier}

For $\sF \in D_{\hol}(\Bun_G)$ be locally compact
with Verdier dual\footnote{Unlike e.g. \cite{agkrrv2}, we consider
Verdier duality as mapping locally compact $D$-modules on $\Bun_G$ 
to locally compact $D$-modules; it can be computed smooth locally by 
usual Verdier duality on schemes. By contrast, \emph{loc. cit}. considers
a smarter construction, sending compact $D$-modules to compact objects
of $D(\Bun_G)^{\vee}$. The smarter construction from \emph{loc. cit}.
recovers ours after applying the functor 
$\on{Id}^{\on{naive}}:D(\Bun_G)^{\vee} \to D(\Bun_G)$ from 
\emph{loc. cit}.}
$\bD^{\on{Verdier}} \sF \in D_{\hol}(\Bun_G)$, 
we have:
\[
\coeff(\sF) = \coeff_!(\bD^{\on{Verdier}} \sF)^{\vee}.
\]

\end{lem}

\subsubsection{}

Similarly, we have functors:
\[
\coeff_{D,!}: D_{\hol}(\Bun_G) \to \Vect.
\]

The Verdier dual to Theorem \ref{t:cs} asserts:

\begin{cor}\label{c:cs-!}

There is a canonical isomorphism of functors:
\[
\coeff_{D,!} \simeq \coeff_!(\sV^D \star -): D_{\hol}(\Bun_G) \to \Vect.
\]

\end{cor}

This follows from the good properties of Hecke functors:
see \cite{agkrrv3} \S 1.2.3.
  
\section{The index formula}\label{s:index}

\subsection{Statement of the theorem}

\subsubsection{Kostant invariant}

Let $\sF \in \Shv_{\Nilp}(\Bun_G)^{\constr}$.

Let $\Irr(\Nilp)$ denote the set of irreducible components of
$\Nilp$. Recall that $\sF$ has a characteristic cycle:
\[
\CC(\sF) = \sum_{\alpha \in \Irr(\Nilp)} c_{\alpha,\sF}[\alpha]
\]

\noindent for $c_{\alpha,\sF} \in \bZ$; here $[\alpha]$ is the class
of the component $\alpha$ in the group of cycles.\footnote{Technically, our
group of cycles is completed here: it is the inverse limit over 
$U$ of free abelian groups
on cycles in $T^*U$ for $U \subset \Bun_G$ a quasi-compact open. In particular,
the infinite sum displayed above has a clear meaning.}

For $\alpha = \Nilp^{\Kos}$, we use the abbreviation: 
\[
c_{\Kos,\sF} \coloneqq c_{\Nilp^{\Kos},\sF} \in \bZ
\]

\noindent for the multiplicity of the characteristic cycle of $\sF$
at the Kostant component (see \S \ref{sss:kost-defin}).

\subsubsection{Index formula}

For $\sF$ constructible as above, we may form $\coeff(\sF) \in \Vect$. Because
$\sF$ is constructible, this object is compact, so it has
a well-defined Euler characteristic $\chi(\coeff(\sF)) \in \bZ$.

The purpose of this section is to prove the following result.

\begin{thm}\label{t:index}

There is a sign $\eps = \eps_{G,X} \in \{1,-1\}$ (depending only on $G$ and 
the genus of the curve $X$) such that for 
$\sF \in \Shv_{\Nilp}(\Bun_G)^{\on{constr}}$, we have the equality of integers:
\[
\chi(\coeff(\sF)) = \eps \cdot c_{\Kos,\sF}. 
\]

Specifically, the sign $\eps$ is:
\[
\vareps = (-1)^{\dim \Bun_G}.
\]

\end{thm}

\subsection{Proof via filtered $D$-modules}

We now present a second proof of Theorem \ref{t:index}.
The proof is quick and simple, using filtered $D$-modules on stacks.
However, it passes through $D$-modules with filtrations that
are not good, so is unlikely to have an easy analogue in 
other sheaf theoretic settings.

\subsubsection{Filtered $D$-modules on stacks}

We briefly develop the theory here, for lack of a good reference.
Let $\sY$ be a smooth algebraic stack.

Recall from \cite{gr-ii} Example 2.4.3 that there is a canonical
prestack\footnote{It is denotes $(\sY_{\dR})_{\on{scaled}}$ in 
\cite{gr-ii}.} 
$\sY_{\dR,\hbar}$ with a map $\pi:\sY_{\dR,\hbar} \to \bA_{\hbar}^1/\bG_m$
so that $\pi^{-1}(1) \simeq \sY_{\dR}$ and $\pi^{-1}(0) = \bB_{\sY}(T(\sY)_0^{\wedge})$
is the classifying stack for the tangent space of $\sY$ along its zero section.
Filtered $D$-modules on $\sY$ are by definition ind-coherent sheaves on 
$\sY_{\dR,\hbar}$; we denote the category by $\Fil D(\sY)$. 
For a filtered $D$-module 
$\fil_{\dot} \sF \in \Fil D(\sY)$, its underlying 
object $\sF$ is the fiber at the open point $1 \in \bA_{\hbar}^1/\bG_m$,
using the equivalence (or definition) $\IndCoh(\sY_{\dR}) \simeq D(\sY)$.
We may form the associated graded $\gr_{\dot} \sF \in \QCoh(T^*\sY)$
by taking the fiber at $\hbar = 0$ and applying Koszul duality\footnote{Smoothness
of $\sY$ is needed here.}
$\IndCoh(\bB_{\sY}(T(\sY)_0^{\wedge})) \simeq 
\QCoh(T^*\sY)$.\footnote{For a slower introduction to this circle of ideas, 
we refer to \cite{whit} Appendix A.} We remind that $T^*\sY$ is a derived
stack in general. For $\sY$ a smooth scheme, the comparison results
in \cite{gr-ii} show that the above notion corresponds to the usual notion.

For a morphism $f:\sY \to \sZ$, we form the usual correspondence:
\[
\begin{tikzcd}
& 
T^*\sZ \underset{\sZ}{\times} \sY 
\arrow[dl,"Df",swap] 
\arrow[dr,"\pi"]
\\
T^*\sY 
&&
T^*\sZ.
\end{tikzcd}
\]

\noindent Then for $\sF \in \Fil D(\sY)$ (resp. 
$\sG \in \Fil D(\sZ)$), $f_{*,dR}(\sF)$ (resp. $f^!(\sG)$)
inherits a canonical filtration, and there are natural identifications:
\begin{equation}\label{eq:fil-functoriality}
\begin{gathered}
\gr_{\dot} f_{*,dR}(\sF) \simeq \pi_*Df^*(\gr_{\dot} \sF). \\
\gr_{\dot} f^!(\sG) \simeq Df_*\pi^{!,\QCoh}(\gr_{\dot} \sG). \\
\end{gathered}
\end{equation}

\noindent (We remark that as $\sY$ and $\sZ$ are smooth, $\pi$ is 
a quasi-smooth morphism, so $\pi^{!,\QCoh}$ is defined.)

\subsubsection{Good filtrations}\label{sss:good-filt-constr}

We say a filtration $\fil_{\dot} \sF$ on $\sF \in D(\sY)$ is a
\emph{good filtration} if for any $p:U \to \sY$ with 
$p$ smooth and $U$ affine, the induced filtration on $p^!\sF$
is a good filtration (equivalently: the filtration is bounded from
below and $\gr_{\dot} p^! \sF \in \Perf(T^*U)$).
By \eqref{eq:fil-functoriality}, this is equivalent to the filtration being bounded from 
below with $\gr_{\dot} \sF \in \Coh(T^*\sY)$.

Clearly if $\sF$ admits a good filtration, it is locally compact
(cf. \S \ref{sss:loc-cpt}).
Conversely, we have:

\begin{lem}\label{l:good-fil}

Suppose $\sY$ is QCA and $\sF \in D(\sY)^{\heart}$ is locally compact. 
Then $\sF$ admits a good filtration.

\end{lem}

\begin{proof}

By \cite{dg-finiteness} Theorem 0.4.5, there exists
$\sG \in \Coh(\sY)^{\heart}$ and a map 
$\ind(\sG) \to \sF \in D(\sY)$ that is an epimorphism on 
$H^0$ (where $\ind:\IndCoh(\sY) \to D(\sY)$ is the $D$-module
induction functor). 
We immediately see that $\ind(\sG)$ admits a good filtration, 
which we denote by $\fil_{\dot} \ind(\sG)$.
By adjunction, we 
then obtain a map: 
\[
\alpha:\fil_{\dot} \ind(\sG) \to \sF^{\on{cnst}} \in \Fil D(\sY)
\]

\noindent where the right hand side denotes $\sF$ with the ``constant" filtration
(informally: $\fil_i \sF = \sF$ for all $i \in \bZ$). 

Now observe that $\Fil D(\sY)$ has a natural $t$-structure
such that the forgetful functor to $\IndCoh(\sY \times \bA_{\hbar}^1/\bG_m)$ is
$t$-exact.\footnote{This $t$-structure may be constructed
directly from the case of smooth schemes. Alternatively, 
one may observe that $\sY \times \bA_{\hbar}^1/\bG_m \to 
\sY_{\dR,\hbar}$ has a connective relative tangent complex; therefore,
by \cite{gr-ii} \S 9, the corresponding monad on 
$\IndCoh(\sY \times \bA_{\hbar}^1/\bG_m)$ is right $t$-exact, and the claim
follows on general grounds.}

We can then form the image of $H^0(\alpha)$ in $\Fil D(\sY)^{\heart}$.
By exactness of the functor $\Fil D(\sY) \to D(\sY)$ (forgetting the filtration),
$H^0(\alpha)$ is a filtration on $\sF$. It is immediate to see that
the induced filtration on $\sF$ is good (and in fact: $\gr_{\dot}\sF$ lies in 
degree 0, i.e., it is a filtration in the \emph{abelian} categorical
sense, not just the derived categorical sense).

\end{proof}

\subsubsection{Twisted de Rham cohomology}

Next, we digress again. 

Suppose $V$ is a finite-dimensional vector space, which 
we consider as a scheme with a $\bG_m$-action.
For $\lambda \in V^{\vee}$, we have a $D$-module
$\lambda^!(\exp)$ on $V$.

\begin{lem}\label{l:exp-g_m}

Suppose $\fil_{\dot} \sF \in \Fil D(V/\bG_m)$ and $\lambda \in V^{\vee}$. 
We abuse notation in also letting $\sF$ denote the $!$-pulled back
$D$-module on $V$. 

Then $C_{\dR}(V,\sF \overset{!}{\otimes} \lambda^!(\exp)) \in \Vect$
has a canonical filtration such that:
\[
\gr_{\dot} C_{\dR}(V,\sF \overset{!}{\otimes} \lambda^!(\exp)) \simeq
\Gamma(V,d\lambda^*\gr_{\dot} \sF).
\]

Moreover, if the filtration on $\sF$ is bounded from below, 
then the filtration on $C_{\dR}(V,\sF \overset{!}{\otimes} \lambda^!(\exp))$
is bounded from below as well.

\end{lem}

\begin{proof}

In this setting, it is easy to see that the Fourier transform $\check{\sF}$ of
$\sF$ has a canonical filtration with the same associated graded
as $\sF$, noting $T^*V = V \times V^{\vee} = T^*V^{\vee}$.
Now the result follows from passing to $!$-fibers at $\lambda \in V^{\vee}$
(and applying \eqref{eq:fil-functoriality}).

\end{proof}

\subsubsection{Proof of Theorem \ref{t:index}}

It clearly suffices to show the result when 
$\sF \in \Shv_{\Nilp}(\Bun_G)^{\constr,\heart}$.
In this case, take a good filtration on $\sF$ by applying Lemma \ref{l:good-fil}.

As our filtration is good, $\gr_{\dot}$ is set-theoretically suppported
on $\Nilp$. 

Note that $\Bun_N^{\Omega} \to \Bun_G$ factors through 
$\Bun_N^{\Omega}/T$. Moreover, note that the map
$\psi:\Bun_N^{\Omega} \to \bA^1$ is $\bG_m$-equivariant,
using the action of $\bG_m$ on the source via 
$\check{\rho}$ (remarking that the action of $T$ factors through
an action of $T^{\on{ad}} = T/Z_G$, the Cartan of $G^{\ad}$). 

By \eqref{eq:fil-functoriality} and Lemma \ref{l:exp-g_m}, 
we see that $\coeff(\sF)[\dim \Bun_N^{\Omega}]$ has a canonical filtration bounded
from below such that its associated graded is essentially computed
by composing the correspondences:
\[
\begin{tikzcd}
&
T^*\Bun_G \underset{\Bun_G}{\times} \Bun_N^{\Omega} 
\arrow[dl] 
\arrow[dr]
&& 
\Bun_N^{\Omega}
\arrow[dl,"d\psi",swap]
\arrow[dr]
\\
T^*\Bun_G 
&&
T^*\Bun_N^{\Omega}.
&&
\Spec(k)
\end{tikzcd}
\]

\noindent Here \emph{essentially} means the following:
we are supposed to apply upper-$!$ along the first leftward
arrow and upper-$*$ along the second leftward arrow; however, for the first
arrow, upper-$!$ and upper-$*$ differ by tensoring with a graded line
bundle, so up to this discrepancy, we can compose the correspondences
well by base-change.

The composed correspondence is:
\[
\begin{tikzcd}
& \Kos_G^{\on{glob}} 
\arrow[dl,"\sigma",swap]
\arrow[dr]
\\
T^*\Bun_G
&&
\Spec(k)
\end{tikzcd}
\]
where $\Kos_G^{\on{glob}}$ is the global Kostant section;
indeed, this is essentially the definition of the Kostant section. 
Note that $\Nilp \times_{\Bun_G} \Kos_G^{\on{glob}} = \Spec(k)$ (as derived stacks), mapping to
$\Nilp$ via $f^{\on{glob}} \in \Nilp^{\Kos}$. 

Now suppose $\sG \in \Coh(\Nilp)^{\heart}$.
As $\Nilp^{\Kos}$ is smooth (all of $\o{\Nilp}^{\on{reg}}$ is), the
Euler characteristic of the (derived) fibers of $\sG$ at points
of $\o{\Nilp}$ are constant. In particular, if $\iota:\Nilp \to T^*\Bun_G$
is the embedding, we see that the Euler characteristic 
of $\Gamma(\Kos_G^{\on{glob}},\sigma^*\iota_*\sG)$
is the rank of $\sG$ at the generic point of $\Nilp^{\Kos}$.
More generally, we deduce that for $\sH \in \Coh(T^*\Bun_G)^{\heart}$
set-theoretically supported on $\Nilp$, the Euler 
characteristic of $\Gamma(\Kos_G^{\on{glob}},\sigma^*\sH)$ is the multiplicity
of $\sH$ at the generic point of $\o{\Nilp}^{\on{reg}}$.

As we have a \emph{good} filtration on $\sF$, we see that $\gr_{\dot} \sF$
is set-theoretically supported on $\Nilp \supset \SS(\sF)$.
Applying the definition of characteristic cycle, we now obtain:
\[
\chi(\Gamma(\Kos_G^{\on{glob}},\sigma^*\gr_{\dot} \sF)) = c_{\Kos,\sF}.
\]

\noindent Reincorporating the twist by the graded line bundle
discussed above, we see that:
\[
\chi(\gr_{\dot}\coeff(\sF)) = \eps \cdot c_{\Kos,\sF}.
\]

It follows immediately that the same equation holds 
for $\coeff(\sF)$ itself, completing the argument.

\begin{rem}

If we knew we could choose a good filtration on $\sF$
such that $\gr_{\dot}\sF$ was (a successive extension of) 
sheaves flat at $f^{\on{glob}}$, we would obtain a similarly easy proof of
Theorem \ref{t:coeff-exact}. Unfortunately, we do not know a way to do 
this.\footnote{We tried and
failed to use the filtration of Kashiwara-Kawai from \cite{kashiwara-kawai} for this
purpose, but were not successful.}

\end{rem}

\section{Exactness of tempered Hecke functors}\label{s:hecke-exact}

In this section, we establish exactness for Hecke functors acting on the 
\emph{tempered} automorphic category. 
This material is a sort of digression: Theorem \ref{t:hecke-exact} does
not mention Whittaker coefficients (although they are used in the proof).
The results of this section are
independent of the rest of the paper up to this point. 

\subsection{Statement of the main result}

Fix a point $x \in X$. 
Recall from \cite{independence} that we have the category 
$D(\Bun_G)^{\xtemp}$. There is a natural quotient functor:
\[
p:D(\Bun_G) \to D(\Bun_G)^{\xtemp}
\]

\noindent with a fully faithful left adjoint $p^L$. The main theorem
of \cite{independence} asserts that this data is actually independent of the
point $x \in X$, although we will not need this until the discussion in 
\S \ref{ss:vanishing}. 

We consider the action of $\Rep(\check{G})$ on $D(\Bun_G)$ via
Hecke functors at $x \in X$. For $V \in \Rep(\check{G})$, we let
$\sS_{V,x} \star -$ denote the corresponding endofunctor of $D(\Bun_G)$.

The goal for this section is to prove:

\begin{thm}\label{t:hecke-exact}

\begin{enumerate}

\item There is a unique $t$-structure on $D(\Bun_G)^{\xtemp}$ such that 
$p$ is $t$-exact. 

\item The action of $\Rep(\check{G})^{\heart} \subset \Rep(\check{G})$ 
on $D(\Bun_G)^{\xtemp}$ is by $t$-exact functors. 

\end{enumerate}

\end{thm} 

The results in this section adapt to etale sheaves in positive characteristic
(conditional on derived geometric Satake in that context).

As we discuss in \S \ref{ss:vanishing}, the above result
strengthens the main results of \cite{vanishing}. In fact, our proof 
is dramatically simpler and has clear conceptual meaning;\footnote{By contrast,
the construction of 
$\widetilde{D}(\Bun_{GL_n})$ from \cite{vanishing} is ad hoc, and by its
nature cannot generalize to other reductive groups. Our construction 
produces a different category with the same nice features, and which 
does generalize.}
it turns out the assertion is something purely local, and 
a simple application of known results about derived Satake and the
spherical Whittaker category.

\subsubsection{}

The argument is purely local. Therefore, we largely take
$\sC$ to be a $G(K)$-category throughout this section
(for loops being based at $x$); the application will be
when $\sC = D^!(\Bun_G^{\xlevel})$, where 
$\Bun_G^{\xlevel}$ is the moduli scheme of
$G$-bundles with complete level structure at $x$.

We remind that in this setting, one can speak of a $t$-structure on $\sC$
being \emph{strongly compatible} with the $G(K)$-action; 
see \cite{methods} \S 10. 

\subsection{Averaging from the spherical category}

Suppose $\sC$ is a category with a $G(K)$-action. 
We use the notation of \S \ref{ss:baby-at}.

\begin{lem}\label{l:sph-av-exact}

Suppose that $\sC \in G(K)\mod$ is equipped with a 
$t$-structure strongly compatible with the $G(K)$-action. 
Then the $!$-averaging functor 
$\Av_!^{\psi}[(2\check{\rho},\rho)]:\sC^{G(O)} \to \Whit^{\leq 1}(\sC)$ is $t$-exact.

\end{lem}

\begin{proof}

The proof is standard: we review it here. 

The functor $\Av_*^{I_1^- \to I}:\sC^{I_1^-} \to \sC^I$ is an equivalence,
and in particular, admits a left adjoint, which we denote $\Av_!^{I \to I_1^-}$. 
On general grounds (cf. \cite{whit} Appendix B), 
there is a canonical natural transformation:
\[
\Av_!^{I \to I_1^-}[-2\dim (I_1^-\cdot I/I)] \to \Av_*^{I \to I_1^-}.
\]

\noindent We remark that the displayed dimension is $2(\check{\rho},\rho)+|\Delta^+|$,
where $\Delta^+$ is the set of positive roots for $G$. 

As in \cite{fg2} Lemma 15.1.2, for $\sF \in \sC^{G(O)}$, the cokernel
of the natural map:
\[
\Av_!^{I \to I_1^-}(\sF)[-2\dim (I_1^-\cdot I/I)] \to \Av_*^{I \to I_1^-}(\sF)
\]

\noindent is partially integrable (in the sense of \emph{loc. cit}.); 
in particular, the above map induces an isomorphism on applying
$\Av_*^{\o{I}_1,\psi}$ (cf. \cite{fg2} Proposition 14.2.1).

Now recall from \cite{bbm} (see also \cite{b-function} Appendix A) that we have 
$\Av_!^{\o{I}_1,\psi} = \Av_*^{\o{I}_1,\psi}[2\dim N]$.

By Lemma \ref{l:baby}, the functor in question is a composition:
\[
\sC^{G(O)} \xar{\Oblv} \sC^{I} \xar{\Av_*^{I \to I_1^-}} 
\sC^{I_1^-} \xar{\Av_!^{\o{I}_1,\psi}} \Whit^{\leq 1}(\sC).
\]

\noindent Applying \cite{bbm} as above, this may instead be written
as:
\[
\sC^{G(O)} \xar{\Oblv} \sC^{I} \xar{\Av_*^{I \to I_1^-}} 
\sC^{I_1^-} \xar{\Av_*^{\o{I}_1,\psi}[2\dim N]} \Whit^{\leq 1}(\sC).
\]

By \cite{whit} Lemmas B.2.2-3 and the above, this functor
has amplitude $\leq \dim (I_1^-\cdot I/I) + \dim (\o{I}_1 I_1^-/I_1^-) - 2\dim N = 
2(\check{\rho},\rho)$. 

On the other hand, by the above, we can also rewrite this functor as 
the composition:
\[
\sC^{G(O)} \xar{\Oblv} \sC^{I} \xar{\Av_!^{I \to I_1^-}[-2\dim (I_1^-\cdot I/I)]} 
\sC^{I_1^-} \xar{\Av_!^{\o{I}_1,\psi}} \Whit^{\leq 1}(\sC).
\]

\noindent Because $\Av_!^{I \to I_1^-}$ is 
inverse, hence right adjoint, to $\Av_*^{I_1^- \to I}$, which 
has amplitude $\leq \dim (I I_1^-/I_1^-) = 2(\check{\rho},\rho)+\dim N$ 
(by \cite{whit} Lemma B.2.2),
this $\Av_!$ functor has amplitude $\geq -2(\check{\rho},\rho)-\dim N$.\footnote{This
would be obvious from \cite{whit} Lemma B.2.3, but that result requires
the two subgroups in question to mutually lie in a compact open subgroup,
where \cite{whit} Lemma B.2.2 does not.} 
Applying \cite{whit} Lemma B.2.3 to $\Av_!^{\o{I}_1,\psi}$, we see
that the above functor has amplitude $\geq 2(\check{\rho},\rho)+\dim N-\dim N = 
2(\check{\rho},\rho)$. 

Combined with the above, we find that $\Av_!^{\psi}:\sC^{G(O)} \to \Whit^{\leq 1}(\sC)$
has amplitude exactly $2(\check{\rho},\rho)$.

\end{proof}

\begin{rem}

As in \cite{whit} Appendix B, a $G(K)$-category with strongly compatible
$t$-structure has an induced $t$-structure on $\Whit(\sC)$. It is likely
the case that $\iota_{1,\infty,!}[2(\check{\rho},\rho)]:\Whit^{\leq 1}(\sC) \into \Whit(\sC)$ 
is $t$-exact; this was shown in \emph{loc. cit}. for $\sC = \widehat{\fg}_{\kappa}\mod$,
as was observed there also for $\sC$ being $D$-modules on a reasonable
indscheme (equipped with a dimension theory, to obtain a $t$-structure).
In this case, the above result would simply say that 
$\Av_!^{\psi}: \sC^{G(O)} \to \Whit(\sC)$ is $t$-exact. 

In other words, the use of baby Whittaker rather than full Whittaker 
in the above (and what follows) simply
reflects our ignorance regarding this point.

\end{rem}

\subsection{Construction of the $t$-structure}

Suppose again that $\sC$ is a $G(K)$-category with a strongly compatible 
$t$-structure. 

In this case, we may form $\sC^{G(O)}$ and its tempered quotient
$\sC^{G(O),\xtemp}$. We let $p:\sC^{G(O)} \to \sC^{G(O),\xtemp}$ denote
the canonical projection. We remind that $p$ admits a fully faithful
left adjoint $p^L$.

We observe that $\sC^{G(O)}$ admits a canonical $t$-structure,
characterized by the fact that $\Oblv:\sC^{G(O)} \to \sC$ is $t$-exact.

\begin{prop}\label{p:temp-t}

In the above setting, there is a unique $t$-structure on 
$\sC^{G(O),\xtemp}$ such that the projection $p:\sC^{G(O)} \to \sC^{G(O),\xtemp}$
is $t$-exact.

\end{prop}

\begin{proof}

By \cite{localization} Lemma 10.2.1,\footnote{The cited lemma uses an adjunction
in which the quotient admits a right adjoint, not a left adjoint. However,
the proof in \emph{loc. cit}. works for arbitrary DG categories, not
necessarily cocomplete ones (although it is written in that context). Therefore,
we may safely pass to opposite categories to deduce the claim
(or observe that the argument in \emph{loc. cit}. immediately applies
in the present context).} it suffices
to show that $\Ker(p) = \sC^{G(O),\xat}$ is closed under truncations, and that
the resulting abelian category $\sC^{G(O),\xat,\heart}$ is closed under 
subobjects (cf. \cite{localization} Remark 10.2.2).

By Lemma \ref{l:at-whit}, we have:
\[
\sC^{G(O),\xat} = \Ker\big(\Av_!^{\psi}:\sC^{G(O)} \to \Whit^{\leq 1}(\sC)\big).
\]

\noindent By Lemma \ref{l:sph-av-exact}, $\Av_!^{\psi}$ is $t$-exact up to shift; it 
follows immediately that its kernel is closed under truncations, and the heart of the
kernel is closed under taking subobjects.

\end{proof}

\subsection{More on Whittaker functors}

We continue to assume $\sC$ is acted on by $G(K)$.

Fix a dominant coweight $\check{\lambda}$. In this case,
we can perform two constructions.

\begin{itemize}

\item 

Let $\psi^{\check{\lambda}}:\Ad_{-(\check{\rho}+\check{\lambda})(t)}\o{I} \to \bG_a$ be the composition:
\[
\Ad_{-(\check{\rho}+\check{\lambda})(t)}\o{I} \xar{\Ad_{\check{\lambda}(t)}} \o{I}_1 \xar{\psi} \bG_a.
\]

\noindent (I.e., take a character of conductor $\check{\lambda}$ for $N(K)$ and
apply the corresponding baby Whittaker construction.)

\item 

Take a representation $V^{\check{\lambda}} \in \Rep(\check{G})^{\heart}$.

\end{itemize}

We consider $\sC^{G(O)}$ as acted on by $\Rep(\check{G})$ via Satake.

\begin{prop}\label{p:baby-cs}

In the above setting, there is a canonical commutative diagram:
\begin{equation}\label{eq:cs-local}
\begin{tikzcd}
\sC^{G(O)} 
\arrow[d,equals]
\arrow[rr,"V^{\check{\lambda}} \star -"]
&&
\sC^{G(O)} 
\arrow[rr,"\Av_!^{\o{I}_1,\psi}"]
&&
\sC^{\o{I}_1,\psi} = \Whit^{\leq 1}(\sC) 
\arrow[d,equals]
\\
\sC^{G(O)}
\arrow[rr,"\Av_*^{\psi^{\check{\lambda}}}{[(\check{\lambda},2\rho)]}"]
&&
\sC^{\Ad_{-(\check{\rho}+\check{\lambda})(t)}\o{I},\psi^{\check{\lambda}}}
\arrow[rr,"\overset{\check{\lambda}(t) \cdot -}{\simeq}"]
&&
\sC^{\o{I}_1,\psi} =
\Whit^{\leq 1}(\sC)
\end{tikzcd}
\end{equation}

\end{prop}

\begin{proof}

This result is an easy application of the Casselman-Shalika formula.
Specifically, we will see both sides are given by convolving with the
same sheaf. 

For a coweight $\check{\mu}$, 
let $\jmath^{\check{\mu}}: \o{I}_1\cdot (-\check{\mu})(t) G(O)/G(O) \to \Gr_G$
be the locally closed embedding of the $\o{I}_1$-orbit through 
$(-\check{\mu})(t) \in \Gr_G$.
For $\mu$ dominant, let $\psi_{\o{I}_1}^{\check{\mu},!}(\exp)$ denote the 
character sheaf on this orbit, normalized to lie in the same cohomological
degree as the dualizing sheaf. 

The top line in \eqref{eq:cs-local} is then given by convolution with 
\[
\jmath_!^0(\psi_{\o{I}_1}^{0,!}(\exp))  \star \sS_{V^{\check{\lambda}}} \in 
\Whit^{\leq 1}(D(\Gr_G)).
\]
\noindent for $\sS_{V^{\check{\lambda}}}$ the spherical sheaf corresponding
to $V^{\check{\lambda}}$.

The bottom line in \eqref{eq:cs-local} is given by convolution 
with: 
\[
\jmath_*^{\check{\lambda}}(\psi_{\o{I}_1}^{\check{\lambda},!}(\exp))
[-2\dim \big((\Ad_{(-\check{\lambda}-\check{\rho})(t)} \o{I})\cdot G(O)/G(O)\big)
+(\check{\lambda},2\rho)].
\]

\noindent Here the first summand in the shift appears because we should use constant
sheaves (rather than dualizing sheaves) for $*$-averaging, and the
second summand appears simply because it is in \eqref{eq:cs-local}.
We observe that:
\[
\dim \big((\Ad_{(-\check{\lambda}-\check{\rho})(t)}\o{I}) \cdot G(O)/G(O)\big) = 
(\check{\lambda}+\check{\rho},2\rho)
\]

\noindent so the above may be rewritten as:
\[
\jmath_*^{\check{\lambda}}(\psi_{\o{I}_1}^{\check{\lambda},!}(\exp))
[-(\check{\lambda}+2\check{\rho},2\rho)].
\]

Finally, by the form of the geometric Casselman-Shalika formula 
given in \cite{abbgm} Theorem 2.2.2 and Corollary 2.2.3, we have:
\[
\jmath_!^0(\psi_{\o{I}_1}^{0,!}(\exp))  \star \sS_{V^{\check{\lambda}}}
[-(\check{\rho},2\rho)] \simeq 
\jmath_*^{\check{\lambda}}(\psi_{\o{I}_1}^{\check{\lambda},!}(\exp))
[-(\check{\lambda}+\check{\rho},2\rho)].
\]

\noindent This yields the claim. 

\end{proof}

\subsection{A generalization}

We remind that by construction, the quotient 
$\sC^{G(O),\xtemp} = \sC^{G(O)}/\sC^{G(O),\xat}$ inherits
a (unique) $\Rep(\check{G})$-action for which the projection
$p:\sC^{G(O)} \to \sC^{G(O),\xtemp}$ is $\Rep(\check{G})$-linear.

We now prove:

\begin{thm}\label{t:hecke-exact-loc}

Suppose $G(K)$ acts on $\sC \in \DGCat_{cont}$, and that
$\sC$ is equipped with a $t$-structure that is strongly compatible with this action.

Then for every $V \in \Rep(\check{G})^{\heart}$, the functor:
\[
V \star -:\sC^{G(O),\xtemp} \to \sC^{G(O),\xtemp}
\]

\noindent is $t$-exact with respect to the $t$-structure from Proposition \ref{p:temp-t}.

\end{thm}

\begin{proof} 

We treat right and left exactness separately.

\step First, we show that $V \star -:\sC^{G(O),\xtemp} \to \sC^{G(O),\xtemp}$ is 
right $t$-exact.

It suffices to prove this result for irreducible representations.
Therefore, we take $V = V^{\check{\lambda}}$.

Suppose $\sF \in \sC^{G(O),\leq 0}$ is connective. It suffices to show
that $p(V^{\check{\lambda}} \star \sF) \in \sC^{G(O),\xtemp,\leq 0}$.

By Lemma \ref{l:sph-av-exact} and Lemma \ref{l:at-whit}, we have a (necessarily
unique) commutative diagram:
\[
\begin{tikzcd}
\sC^{G(O)} 
\arrow[d,"p"]
\arrow[drr,"{\Av_!^{\psi}[(2\check{\rho},\rho)]}"]
\\
\sC^{G(O),\xtemp} 
\arrow[rr,dotted]
&& 
\Whit^{\leq 1}(\sC)
\end{tikzcd}
\]

\noindent in which (crucially!) the bottom arrow is $t$-exact and conservative.

Therefore, it suffices to show that:
\[
\Av_!^{\psi}(V^{\check{\lambda}} \star \sF)[(2\check{\rho},\rho)] \in \Whit^{\leq 1}(\sC)^{\leq 0}.
\]

By Proposition \ref{p:baby-cs}, we can rewrite this term as:
\[
\Av_*^{\psi^{\check{\lambda}}}(\sF)[(\check{\lambda}+\check{\rho},2\rho)].
\]

\noindent By in this form, the desired estimate 
follows the usual estimates for the amplitude of $\Av_*$ functors: 
see \cite{whit} Lemma B.2.2.

\step 

We now prove left $t$-exactness. It suffices to prove this
for finite dimensional $V$. 
Then the functor $V \star -:\sC^{G(O),\xtemp} \to \sC^{G(O),\xtemp}$ is
right adjoint to $V^{\vee} \star -:\sC^{G(O),\xtemp} \to \sC^{G(O),\xtemp}$.
The latter is right $t$-exact by the above, so the former must be left $t$-exact
as desired.

\end{proof}

Finally, Theorem \ref{t:hecke-exact} follows by taking 
$\sC = D^!(\Bun_G^{\xlevel})$ in Theorem \ref{t:hecke-exact-loc}.

\subsection{Variant for nilpotent sheaves}\label{ss:hecke-exact-nilp}

Note that the spherical Hecke action $\sH_x^{\on{sph}} = D(\Gr_G)^{G(O)} 
\actson D(\Bun_G)$
preserves $\Shv_{\Nilp}(\Bun_G)$. Therefore, we may form:
\[
\Shv_{\Nilp}(\Bun_G)^{\xtemp} \coloneqq 
\Shv_{\Nilp}(\Bun_G) \underset{\IndCoh(\Omega_0\check{\fg}/\check{G})}{\otimes}
\QCoh(\Omega_0\check{\fg}/\check{G}). 
\]

\noindent as for $D(\Bun_G)^{\xtemp}$. The same applies for 
$\Shv_{\Nilp}(\Bun_G)^{\at}$.

By functoriality, we then have a commutative diagram:
\[
\begin{tikzcd}
\Shv_{\Nilp}(\Bun_G)^{\at}
\arrow[rr]
\arrow[d]
&&
D(\Bun_G)^{\at} 
\arrow[d]
\\
\Shv_{\Nilp}(\Bun_G)
\arrow[rr]
\arrow[d,"p"]
&&
D(\Bun_G)
\arrow[d,"p"]
\\
\Shv_{\Nilp}(\Bun_G)^{\temp}
\arrow[rr]
&&
D(\Bun_G)^{\temp}.
\end{tikzcd}
\]

\noindent The horizontal functors are fully faithful e.g. 
because $\Shv_{\Nilp}(\Bun_G) \to D(\Bun_G)$ admits a 
$\sH_x^{\on{sph}}$-linear right adjoint.

We see from Theorem \ref{t:hecke-exact} that $\Shv_{\Nilp}(\Bun_G)^{\at}$
is closed under truncations and subobjects (since this is 
true for $D(\Bun_G)^{\at}$ and $\Shv_{\Nilp}(\Bun_G)$ separately).
Therefore, $\Shv_{\Nilp}(\Bun_G)^{\temp}$ inherits a $t$-structure
for which the projection $p:\Shv_{\Nilp}(\Bun_G) \to \Shv_{\Nilp}(\Bun_G)^{\xtemp}$
is $t$-exact. Then the bottom horizontal arrow above is $t$-exact
(and conservative, being fully faithful) for this $t$-structure.  

From the diagram above and Theorem \ref{t:hecke-exact}, 
we see that Hecke functors are $t$-exact on $\Shv_{\Nilp}(\Bun_G)^{\temp}$
with respect to the above $t$-structure.

\subsection{Relation to Gaitsgory's work for $GL_n$}\label{ss:vanishing}

In this section, we briefly indicate how the above results
can be used to better understand the main results of \cite{vanishing}.

\subsubsection{A variant with moving points}

Recall from \cite{independence} that $D(\Bun_G)^{\xtemp}$ is canonically
independent of the point $x \in X$. We therefore use the
notation $D(\Bun_G)^{\temp}$ instead.

Let $V \in \Rep(\check{G})^{\heart}$ be given. Recall that there
is a Hecke functor:
\[
H_{V,X}:D(\Bun_G) \to D(X \times \Bun_G) = D(X) \otimes D(\Bun_G)
\]

\noindent whose $!$-fibers at points $x \in X$ give the usual
Hecke functors at points.

It is easy to see that the functor $H_{V,X}$ induces a functor:
\[
H_{V,X}^{\temp}:D(\Bun_G)^{\temp} \to D(X) \otimes D(\Bun_G)^{\temp}.
\]

We have the following generalization of Theorem \ref{t:hecke-exact}:

\begin{thm}\label{t:hecke-exact-curve}

The functor
\[
H_{V,X}^{\temp}[-1]:D(\Bun_G)^{\temp} \to D(X) \otimes D(\Bun_G)^{\temp}
\]

\noindent is $t$-exact.

\end{thm}

This follows by performing the proof of Theorem \ref{t:hecke-exact} over
$X$, and applying \cite{independence}. 

\subsubsection{}

We now observe that Theorem \ref{t:hecke-exact-curve} yields
a quotient of $D(\Bun_G)$ with the properties described in 
\cite{vanishing} \S 2.12. Namely, Hecke functors are $t$-exact,
and $D_{\cusp}(\Bun_G) \subset D(\Bun_G)$ is right orthogonal 
to $D(\Bun_G)^{\at}$ by \cite{dario-ramanujan}.

The construction of such a quotient is the main technical input in 
\cite{vanishing}; see the discussion of \emph{loc. cit}. \S 2.13.

Our argument is valid for general reductive groups $G$. Moreover,
the degree restrictions in \emph{loc. cit}. are not necessary here.
Finally, we observe that the quotient we consider here has evident
conceptual meaning in geometric Langlands, which was not the case for
the quotient considered in \cite{vanishing}. 

Finally, we remark that even for $GL_n$, our methods are much 
simpler than those in \cite{vanishing}. 

However, to obtain a result for $\ell$-adic sheaves, one needs
some additional input. First, one needs derived Satake
for $\ell$-adic sheaves (which has been announced by Arinkin-Berzukavnikov).
More seriously, one would need the independence of\footnote{Here,
unlike in the rest of the paper, $\Shv$ denotes $\ell$-adic sheaves,
not regular holonomic $D$-modules.} 
$\Shv(\Bun_G)^{\xtemp}$ of the point $x \in X$; \cite{independence} shows this only 
for $\Shv_{\Nilp}(\Bun_G)$ in the $\ell$-adic context. 

We also refer back to Remark \ref{r:vanishing} for more context on our result.

\section{Whittaker coefficients of nilpotent sheaves}\label{s:coeff-exact}

In this section, we establish favorable properties of Whittaker coefficients
for sheaves with nilpotent singular support.

For our later applications, the main result of this section is:

\begin{thm}\label{t:coeff-exact}

The functor $\coeff[\dim \Bun_G]:\Shv_{\Nilp}(\Bun_G) \to \Vect$ is $t$-exact.

\end{thm}

Using \cite{agkrrv2}, we 
deduce this result from a theorem of Kevin Lin; at the time we are writing
this article, Lin's result is not yet publicly 
available.\footnote{Since we originally wrote this article, \cite{nadler-taylor}
has appeared. Their work yields an alternative proof of Theorem \ref{t:coeff-exact}
that is topological in nature and does not use Theorem \ref{t:coeff-exact}.}

\subsection{Around Lin's theorem}

We begin this section by describing 
Lin's result and deducing some immediate consequences of it.

\subsubsection{Formulation}

Below, we let\footnote{The notation is a bit funny; we follow \cite{agkrrv1} in 
letting $\sfe$ be opaque notation for the field $k$, thought of as 
\emph{the field of coefficients for our sheaf theory}. 

We find this notation $\ul{\sfe}_{-}$ a bit more geometrically communicative for constant sheaves
than $\ul{k}_{-}$... although it gets tricky for a point.} 
$\ul{\sfe}_{\Bun_G} \in D(\Bun_G)$ denote the constant sheaf, 
i.e., $\ul{\sfe}_{\Bun_G} = \omega_{\Bun_G}[-2\dim \Bun_G]$.

We let $\Delta = \Delta_{\Bun_G}$ denote the diagonal map 
$\Bun_G \to \Bun_G \times \Bun_G$, and we let 
$\pi_{\Bun_N^{\Omega}}:\Bun_N^{\Omega} \to \Spec(k)$ denote the projection.

\begin{thm}[K. Lin, to appear]\label{t:lin}

There is a canonical isomorphism:
\[
\begin{gathered}
(\coeff \otimes \id)(\Delta_!\ul{\sfe}_{\Bun_G}) 
\coloneqq 
(\pi_{\Bun_N^{\Omega}} \times \id)_{*,dR}
\Big((\fp \times \id)^!(\Delta_!\ul{\sfe}_{\Bun_G})\overset{!}{\otimes} 
p_1^!\psi^!(\exp)\Big)[-\dim \Bun_N^{\Omega}]
\simeq \\
 \Poinc_![-2\dim \Bun_G] \in D(\Bun_G).
 \end{gathered}
\]

\end{thm}

\begin{rem}

We highlight that this theorem is a particular isomorphism between
two explicit sheaves on $\Bun_G$. It is in the spirit of
many\footnote{Cf. 
\cite{geometric-eisenstein}, \cite{ic-drinfeld}, \cite{lin-dl},
\cite{schieder}, \cite{sakellaridis-wang}, among others.} 
results on quasi-maps spaces in geometric Langlands.
The proof uses geometry of the Vinberg degeneration (via \cite{lin-dl}) 
and Zastava spaces. Specifically, the argument \emph{(i)} constructs a map,
\emph{(ii)} shows that the map is an isomorphism at the cuspidal level
(using that the pseudo-identity and identity coincide there, cf. 
\cite{strange}), and \emph{(iii)} checks that the map is an isomorphism after 
applying constant term functors, using Zastava geometry (and other tools) 
to study the results. 

\end{rem}

\subsubsection{Derivation from geometric Langlands conjectures}

We now (heuristically) show how Theorem \ref{t:lin} is predicted
by standard compatibilities from geometric Langlands.
We first work up to shifts,\footnote{Unfortunately, the compatibility between
both Eisenstein series and strange duality with Langlands duality 
are often only stated up to shifts (and tensoring with line bundles). 
It is not our purpose here to correct 
that issue in the literature here, which unfortunately leaves the ambiguity in 
shifts at the end. 

Our understanding is that the forthcoming
work of Ben-Zvi--Sakellaridis--Venkatesh will systematically clarify such 
issues, including the precise compatibility between both Eisenstein series 
and miraculous duality with geometric Langlands.}

Specifically, we recall from \cite{strange} \S Conjecture 0.2.3 that:
\[
\Delta_!\ul{\sfe}_{\Bun_G} \in D(\Bun_G) \otimes D(\Bun_G)
\]

\noindent \emph{up to shifts} is supposed to correspond to:
\[
(\Psi_{\Nilp} \otimes \Psi_{\Nilp})\Delta_*^{\IndCoh}(\omega_{\LocSys_{\check{G}}})
\]

\noindent under geometric Langlands. Here we abuse notation in letting
$\Nilp \subset T^*[-1](\LocSys_{\check{G}})$ denote the spectral
global nilpotent cone, as in \cite{arinkin-gaitsgory}. 
We also let $\Psi_{\Nilp}:\IndCoh(\LocSys_{\check{G}}) \to 
\IndCoh_{\Nilp}(\LocSys_{\check{G}})$ denote the natural projection. 

On the other hand, recall that the functor: 
\[
\coeff:D(\Bun_G) \to \Vect
\]

\noindent is supposed to correspond to: 
\[
\Gamma^{\IndCoh}(\LocSys_{\check{G}},-):
\IndCoh_{\Nilp}(\LocSys_{\check{G}}) \to \Vect.
\]

Combining these two assertions, we see that:
\begin{equation}\label{eq:coeff-id-gl}
(\coeff \otimes\id)(\Delta_!\ul{\sfe}_{\Bun_G}) \in D(\Bun_G)
\end{equation}

\noindent should correspond (up to shifts) to:
\[
(\Gamma^{\IndCoh}(\LocSys_{\check{G}},-) \otimes \id)
((\Psi_{\Nilp} \otimes \Psi_{\Nilp})\Delta_*^{\IndCoh}(\omega_{\LocSys_{\check{G}}}))
= \omega_{\LocSys_{\check{G}}} \in \IndCoh_{\Nilp}(\LocSys_{\check{G}}).
\]

We now observe that using the symplectic structure on $\LocSys_{\check{G}}$,
we have: 
\[
\omega_{\LocSys_{\check{G}}} \simeq 
\sO_{\LocSys_{\check{G}}}[\on{v.dim}(\LocSys_{\check{G}})]
\]

\noindent where $\on{v.dim}(\LocSys_{\check{G}})$ is the 
Euler characteristic of the cotangent complex of $\LocSys_{\check{G}}$,
which is $2\dim \Bun_G$.

Now observe that $\coeff$ is corepresented by $\Poinc_!$ while 
$\Gamma^{\IndCoh}(\LocSys_{\check{G}},-)$ is corepresented
by $\sO_{\LocSys_{\check{G}}}$, so these two objects must correspond
to each other under geometric Langlands.

Combining these observations, we find that \eqref{eq:coeff-id-gl} and
$\Poinc_!$ both correspond to $\omega_{\LocSys_{\check{G}}}$ (up to shifts)
under geometric Langlands, so we expect the two to be isomorphic (up to shifts): 
this is the assertion of Lin's theorem.

\begin{rem}

There are various ways to recover the precise shift. First, it is built into
the proof of Lin's theorem, specifically, the construction of the comparison map 
in Theorem \ref{t:lin};
we prefer not to describe the comparison map here and leave it to Lin's forthcoming
work. 

However, assuming Theorem \ref{t:lin} \emph{up to shifts}, 
the precise value is also forced by known results. Specifically, let
$\on{Eis}_{!*} \in D(\Bun_G)$ be the $*$-pushforward of the IC sheaf on 
Drinfeld's compactification $\overline{\Bun}_B$.  
According to \cite{potential} Appendix B (11.18), 
$\coeff_!(\on{Eis}_{!*}) = k[\dim \Bun_G]$.\footnote{The reader who glances
at \cite{potential} Appendix B (11.18) 
will find additional shifts. The shift by $-\dim \Bun_N^{\Omega}$ in 
\emph{loc. cit}. does not appear for us simply because we defined
$\coeff_!$ with a shift by $\dim \Bun_N^{\Omega}$ built in.

There is also a shift by $-\dim\Bun_T$ in \emph{loc. cit}. 
We observe that in the notation of \emph{loc. cit}., we should take
$\mathcal{F} = \IC_{\Bun_T} \in \Shv(\Bun_T)^{\heartsuit}$ to recover our specific
example. Moreover, in \emph{loc. cit}. (11.18), 
the $\lambda = 0$ term involves a $*$-fiber of $\mathcal{F}$ 
at $\check{\rho}(\Omega_X) \in \Bun_T$;
as $\IC_{\Bun_T} = \sfe_{\Bun_T}[\dim \Bun_T]$, 
this yields an additional shift by $\dim \Bun_T$ cancelling the one
appearing in the equation, and ultimately leading the precise value
stated here.}
As $\on{Eis}_{!*}$ is Verdier self-dual, we see that 
$\coeff(\on{Eis}_{!*}) = k[-\dim\Bun_G]$. Finally, it is standard to
see that $\on{Eis}_{!*}$ has nilpotent singular support.

On the other hand, a version of Lin's theorem with an additional
shift (beyond the stated one) by $N \in \bZ$ would ultimately
yield a version of our Theorem \ref{t:coeff-!} with the same shift by $N$
appearing. The only one consistent with the above calculation with 
compactified Eisenstein series is $N = 0$.

\end{rem}

\subsubsection{Relation to miraculous duality}

Recall the Drinfeld-Gaitsgory \emph{miraculous duality} functor
from \cite{strange}:\footnote{In \cite{strange}, this
functor is denoted $\on{Ps-id}_{\Bun_G,!}$. Our notation is
taken instead from \cite{agkrrv2}.}
\[
\Mir:D(\Bun_G)^{\vee} \to D(\Bun_G).
\]

\noindent We remind that \cite{strange} Theorem 0.1.6 asserts that
this functor is an equivalence.

\begin{cor}\label{c:lin-mir}

$\Mir(\coeff) = \Poinc_![-2\dim \Bun_G]$.

\end{cor}

This result is a formal consequence of \ref{t:lin}. 
We review the relevant ideas here.

First, we record the following obvious result.

\begin{lem}\label{l:lambda}

Let $\sC \in \DGCat_{cont}$ be given, and let $\lambda:\sC \to \Vect$
be a functor. Then:
\[
(\lambda \otimes \id):\sC \otimes \sC^{\vee} \to \Vect \otimes \sC^{\vee} = \sC^{\vee}
\]

\noindent maps the unit\footnote{I.e., the object
corresponding to $\id_{\sC}$ under $\sC \otimes \sC^{\vee} \simeq 
\TwoEnd_{\DGCat_{cont}}(\sC)$.} $u_{\sC}$  to $\lambda$.

\end{lem}

\begin{proof}[Proof of Corollary \ref{c:lin-mir}]

The defining property of miraculous duality is that the functor:
\[
D(\Bun_G) \otimes D(\Bun_G)^{\vee} \xar{\id \otimes \Mir} 
D(\Bun_G) \otimes D(\Bun_G) \to D(\Bun_G \times \Bun_G)
\]

\noindent sends $u_{D(\Bun_G)}$ to $\Delta_!\ul{\sfe}_{\Bun_G}$.
(We remark that each arrow above is an equivalence.)

We now have the commutative diagram:
\[
\begin{tikzcd}
D(\Bun_G) \otimes D(\Bun_G)^{\vee} 
\arrow[r,"\id \otimes \Mir"]
\arrow[d,"\coeff \otimes \id"]
&
D(\Bun_G) \otimes D(\Bun_G) 
\arrow[r]
&
D(\Bun_G \times \Bun_G)
\arrow[d,"\coeff \otimes \id",swap]
\\
D(\Bun_G)^{\vee}
\arrow[rr,"\Mir"]
&&
D(\Bun_G)
\end{tikzcd}
\]

We calculate the image of $u_{D(\Bun_G)}$ in two ways. By the above,
if we traverse the upper leg of the diagram, we obtain
$\Delta_!\ul{\sfe}_{\Bun_G}$, so Theorem \ref{t:lin} implies 
we obtain $\Poinc_![-2\dim \Bun_G]$ after applying the right arrow.
On the other hand, if we apply the left arrow, Lemma \ref{l:lambda}
implies we obtain $\coeff \in D(\Bun_G)^{\vee}$, which maps to 
$\Mir(\coeff)$ on applying the bottom arrow.

\end{proof}

\subsection{Whittaker coefficients of nilpotent sheaves}

\subsubsection{Comparison of coefficient functors}

We now prove the following assertion:

\begin{thm}\label{t:coeff-!}

There is a canonical isomorphism of functors:
\[
\coeff \simeq \coeff_![-2\dim\Bun_G]:\Shv_{\Nilp}(\Bun_G) \to \Vect.
\]

\end{thm}

Before proving this theorem, we record the following result:

\begin{lem}[\cite{agkrrv2}, Corollary 4.3.7]\label{l:nilp-kernel}

Let $\lambda \in \Shv(\Bun_G)^{\vee}$ be given.\footnote{We can work
with $D$-modules just as well as sheaves here; but in the non-holonomic
case, one would need to remark that $\Mir(\lambda) \overset{*}{\otimes} -$
may take values in the pro-category (although $C_{c,\dR}(\Bun_G,-)$ will
then map the result into $\Vect \subset \Pro\Vect$).}
Then we have a canonical isomorphism:
\[
\lambda|_{\Shv_{\Nilp}(\Bun_G)} \simeq 
C_{c,\dR}(\Bun_G,\Mir(\lambda) \overset{*}{\otimes} -)
\]

\noindent of functors:
\[
\Shv_{\Nilp}(\Bun_G) \to \Vect.
\]

\end{lem}

\begin{proof}[Proof of Theorem \ref{t:coeff-!}]

For $\sF \in \Shv_{\Nilp}(\Bun_G)$, Lemma \ref{l:nilp-kernel} yields:
\[
\coeff(\sF) = C_c(\Bun_G, \sF \overset{*}{\otimes} \Mir(\coeff)).
\]

\noindent Applying Corollary \ref{c:lin-mir}, the right hand side is:
\[
C_c(\Bun_G, \sF \overset{*}{\otimes} \Poinc_![-2\dim \Bun_G]).
\]

\noindent Applying the formula \eqref{eq:poinc-!} and base-change,
the right hand side is $\coeff_!(\sF)[-2\dim\Bun_G]$, yielding the claim.

\end{proof}

\subsubsection{}

We deduce:

\begin{cor}\label{c:coeff-duality}

For $\sF \in \Shv_{\Nilp}(\Bun_G)$ locally compact, we have:
\[
\coeff(\bD^{\on{Verdier}} \sF) = \coeff(\sF)^{\vee}[-2\dim\Bun_G].
\]

\noindent That is, $\coeff[\dim \Bun_G]$ commutes with Verdier duality
on $\Shv_{\Nilp}(\Bun_G)$.

\end{cor}

Indeed, this follows from Lemma \ref{l:coeff-!-verdier} and Theorem \ref{t:coeff-!}.

\subsubsection{Variant with conductor}

Now fix $D$ a $\check{\Lambda}^+$-valued divisor on $X$. 

We have the following generalization of Theorem \ref{t:coeff-!}.

\begin{cor}\label{c:coeff-!-D}

There is a canonical isomorphism of functors:
\[
\coeff_D \simeq \coeff_{D,!}[-2\dim\Bun_G]:\Shv_{\Nilp}(\Bun_G) \to \Vect.
\]

\end{cor}

\begin{proof}

This is immediate from Theorem \ref{t:coeff-!}, Theorem \ref{t:cs}, and 
Corollary \ref{c:cs-!}.

\end{proof}

As with Corollary \ref{c:coeff-duality}, we have:

\begin{cor}\label{c:coeff-d-duality}

For $D$ as above, $\coeff_D[\dim \Bun_G]$ commutes with 
Verdier duality on $\Shv_{\Nilp}(\Bun_G)$.

\end{cor}

\subsection{Exactness}

We now prove Theorem \ref{t:coeff-exact}. In fact, we prove the 
following generalization.

\begin{thm}\label{t:coeff-d-exact}

For every $D$ a $\check{\Lambda}^+$-valued divisor on $X$, the 
functor:
\[
\coeff_D[\dim \Bun_G]:\Shv_{\Nilp}(\Bun_G) \to \Vect
\]

\noindent is $t$-exact.

\end{thm}

\begin{proof}

\step\label{st:coeff-exact-1}

First, we prove this result in the \emph{sufficiently large} case,
cf. \S \ref{sss:large-div}.

We begin by naively estimating the amplitude of $\coeff_D$ on 
$D(\Bun_G)$. Decomposing into steps, we observe:

\begin{itemize}

\item The functor $\fp_D^!$ has amplitude 
$\leq \dim \Bun_G-\dim \Bun_N^{\Omega(-D)}$. 

\item Tensoring
with the character is designed to be $t$-exact. 

\item By Lemma \ref{l:large-div}, for $D$ sufficiently large, 
$C_{\dR}(\Bun_N^{\Omega(-D)},-)$ is right $t$-exact.

\item We recall that there is a shift by 
$\dim \Bun_N^{\Omega(-D)}$ in the definition of $\coeff_D$. 

\end{itemize}

\noindent Combining these observations, we see that 
$\coeff_D$ has cohomological amplitude
$\leq \dim \Bun_G$. 

On the other hand, the same reasoning shows that\footnote{One 
can work just as well with $D_{\hol}(\Bun_G)$ here.}
$\coeff_{D,!}:\Shv(\Bun_G) \to \Vect$ has amplitude $\geq -\dim \Bun_G$.

Now the exactness follows from Corollary \ref{c:coeff-!-D}.

\step 

We now prove the result for $D = 0$.

First, we note that $\Poinc_! \in D(\Bun_G)$ lies in the left orthogonal to 
$D(\Bun_G)^{\at}$; indeed, this follows immediately from derived Satake
and the definition of temperedness. 
It follows that the functor $\coeff = \ul{\Hom}(\Poinc_!,-)$ factors
through the projection to $D(\Bun_G)^{\temp}$, i.e., we have:
\[
\begin{tikzcd}
D(\Bun_G) 
\arrow[drr,"\coeff"]
\arrow[d,"p"] \\
D(\Bun_G)^{\temp} 
\arrow[rr,"\widetilde{\coeff}",dotted]
&&
\Vect.
\end{tikzcd}
\]

\noindent The same applies in the presence of a divisor. 

Let $x \in X$ be a point and let 
$\check{\lambda}$ be a coweight with $\check{\lambda}\cdot x$ sufficiently 
large (e.g., $\check{\lambda} = 2n\check{\rho}$ for $n\gg 0$).
The Hecke action of $\Rep(\check{G})$ on $D(\Bun_G)$ below is 
considered at the point $x \in X$.

Clearly the trivial representation 
is a summand of $V^{\check{\lambda}} \otimes V^{-w_0(\check{\lambda})}$. 
Therefore, $\coeff$ is a summand of 
$\coeff((V^{\check{\lambda}} \otimes V^{-w_0(\check{\lambda})})\star -)$.
Therefore, it suffices to show that the latter functor is $t$-exact.
For $\sF \in \Shv_{\Nilp}(\Bun_G)^{\leq 0}$, we have:
\[
\begin{gathered}
\coeff\big((V^{\check{\lambda}} \otimes V^{-w_0(\check{\lambda})})\star \sF\big) = 
\coeff(V^{\check{\lambda}} \star (V^{-w_0(\check{\lambda})}\star \sF)) 
\overset{\text{Thm. \ref{t:cs}}}{=} \\ 
\coeff_{\check{\lambda} \cdot x}(V^{-w_0(\check{\lambda})} \star \sF) = 
\widetilde{\coeff}_{\check{\lambda} \cdot x}(p(V^{-w_0(\check{\lambda})}\star \sF)) = 
\widetilde{\coeff}_{\check{\lambda} \cdot x}(V^{-w_0(\check{\lambda})}\star p(\sF)). 
\end{gathered}
\]

\noindent Now by Theorem \ref{t:hecke-exact} (and \S \ref{ss:hecke-exact-nilp}),
$p(\sF) \in \Shv(\Bun_G)^{\temp,\leq 0}$,
so $V^{-w_0(\check{\lambda})} \star p(\sF) \in \Shv_{\Nilp}(\Bun_G)^{\temp,\leq 0}$ by 
Theorem \ref{t:hecke-exact}. 
It follows from Step \ref{st:coeff-exact-1} that 
$\widetilde{\coeff}_{\check{\lambda}\cdot x}(V^{-w_0(\check{\lambda})}\star p(\sF))$ 
is in degrees $\leq \dim \Bun_G$,
giving right exactness of $\widetilde{\coeff}_{\check{\lambda}\cdot x}(V^{-w_0(\check{\lambda}\cdot x)} \star -)[\dim \Bun_G]$,
so (by the above), right $t$-exactness of $\coeff[\dim \Bun_G]$ as well.

The same logic applies for left $t$-exactness, giving the claim.

\step We now deduce the claim for general $D$. In fact, this is obvious
from the $D = 0$ case, given Theorem \ref{t:hecke-exact} (in the form
of \S \ref{ss:hecke-exact-nilp}, and using \cite{independence} to allow
divisors with support at multiple points) and Theorem \ref{t:cs}.

\end{proof}

\part{Conservativeness of the Whittaker functor}

\section{Regular nilpotent singular support and Hecke functors}\label{s:hecke-to-kostant}

\subsection{Statement of the result}

This section is dedicated to the proof of the following result.

\begin{thm}\label{t:hecke-to-kostant}

Suppose $\sF \in \Shv_{\Nilp}(\Bun_G)$ has the property 
that $\SS(\sF) \cap \o{\Nilp} \neq \emptyset$, i.e., 
$\SS(\sF) \not\subset \Nilp_{\irreg}$.

Then there exists $D$ a $\check{\Lambda}^+$-valued divisor on $X$
such that $\Nilp^{\Kos} \subset \SS(\sV^D \star \sF)$.

\end{thm}

Combined with Theorem \ref{t:at}, we obtain:

\begin{cor}\label{c:hecke-to-kostant}

Suppose $\sF \in \Shv_{\Nilp}(\Bun_G)$. Then either:

\begin{enumerate}

\item $\sF \in \Shv_{\Nilp}(\Bun_G)^{\at}$, or:

\item There exists $D$ a $\check{\Lambda}^+$-valued divisor on $X$
such that $\Nilp^{\Kos} \subset \SS(\sV^D \star \sF)$.

\end{enumerate}

\end{cor}

\subsection{A local result}

We begin with a purely local result concerning Hecke modifications 
and affine Springer fibers.

\subsubsection{}

We work around an implicit point $x \in X(k)$ with coordinate $t$.
Below, it is convenient for indexing purposes 
to consider $\Gr_G$ as the quotient $G(O)\backslash G(K)$, i.e., we
quotient on the \emph{left}. There is a residual $G(K)$-action on the
right. We let $\Gr_G^{\check{\mu}}$ denote the $G(O)$-orbit through 
$\check{\mu}(t) \in\Gr_G$, and we let $\ol{\Gr}_G^{\check{\mu}}$ denote
its closure.

Let $\xi \in \fg((t))$ be given. We define the \emph{affine Springer fiber}
$\Spr^{\xi}$ as: 
\[
\Spr^{\xi} \coloneqq
G(O)\backslash \{g \in G(K) \mid \Ad_g(\xi) \in \fg[[t]]\} \subset \Gr_G.
\]

\subsubsection{}

Below, we fix a $k$-point $\vph \in \sN(O)$, i.e., 
a nilpotent element $\vph \in \fg[[t]]$. We suppose
that $\vph$ is \emph{generically regular}, i.e., 
the induced element of $\fg((t))$ is regular.

As in \S \ref{sss:discrepancy}, there is a canonical
\emph{discrepancy} $\on{disc}(\vph) \in \check{\Lambda}_{G^{\ad}}^+$ 
attached to this element. Specifically, the perspective of
\emph{loc. cit}. attaches a $\check{\Lambda}_{G^{\ad}}^+$-valued divisor
on the formal disc to $\vph$, but we can think of this simply 
as an element of $\check{\Lambda}_{G^{\ad}}^+$ via its degree.

Specifically, if $\check{\mu} \in \check{\Lambda}_{G^{\ad}}^+$ is 
a coweight, saying $\vph$ has discrepancy $\check{\mu}$ means
that it can be $G(O)$-conjugated into:
\begin{equation}\label{eq:discrepancy-explicit}
\Ad_{T(O)\check{\mu}(t)}(e)+[\fn,\fn][[t]] \subset \fg[[t]]
\end{equation}

\noindent for $e \in \oplus_{i \in \sI_G}\fn_{\alpha_i} \subset \fn$ 
a regular nilpotent element.

\begin{rem}

For $\check{\mu} \in \check{\Lambda}_{G^{\ad}}^+$, 
the construction of \S \ref{sss:springer-desc} yields
a locally closed scheme $\sN(O)_{\check{\mu}} \subset \sN(O)$
parametrizing generically regular $\vph \in \sN(O)$ with
discrepancy $\check{\mu}$. Specifically, 
we take:
\[
\sN(O)_{\check{\mu}} \coloneqq \widetilde{\sN}(O) 
\underset{\prod_{i \in \sI_G} (\bA^1/\bG_m)(O)}{\times} \Spec(k)
\]

\noindent where $\Spec(k) \to \prod_{i \in \sI_G} (\bA^1/\bG_m)(O)$
corresponds to $\check{\mu}$ (i.e., it is the point
$(t^{(\check{\mu},\alpha_i)})_{i \in \sI_G}$).

\end{rem}

\subsubsection{}

For $\xi \in \fg((t))$ regular nilpotent and
$\check{\mu} \in \check{\Lambda}_{G^{\ad}}^+$, let 
$\Spr_{\check{\mu}}^{\xi} \subset \Spr^{\xi}$ denote the locally
closed subscheme:
\[
\Spr_{\check{\mu}}^{\xi} \coloneqq
G(O)\backslash \{g \in G(K) \mid \Ad_g(\xi) \in \fg[[t]], 
\on{disc}(\Ad_g(\xi)) = \check{\mu}
\} \subset \Gr_G.
\]

\noindent In other words, we take:
\[
\sN(O)_{\check{\mu}}/G(O) \underset{\fg((t))/G(K)}{\times} \{\xi\}.
\]

\subsubsection{Main local result}

We now have:

\begin{prop}\label{p:spr}

Let $\vph \in \o{\sN}(O)$ be an (\emph{everywhere}) regular nilpotent element
of $\fg[[t]]$.

Suppose $\check{\lambda} \in \check{\Lambda}^+$ is a dominant
coweight, and let $\ol{\check{\lambda}} \in \check{\Lambda}_{G^{\ad}}^+$ denote
the induced coweight for $G^{\ad}$.

Then:
\[
(\ol{\Gr}_G^{\check{\lambda}} 
\cap \Spr_{\ol{\check{\lambda}}}^{\vph})^{\on{red}} = \Spec(k) \in \Gr_G.
\]

\noindent That is, the displayed intersection is the point $\check{\lambda}(t)$, 
at least at the reduced level. 

\end{prop}

\begin{proof}

Below, we let $\pi:G(K) \to G(O)\backslash G(K) = \Gr_G$ denote the projection.

The assertion clearly depends only on the $G(O)$-orbit of $\vph$.
Therefore, we can assume $\vph = e \in \oplus_{i \in \sI_G} \fn_{\alpha_i} \subset
\fn \subset \fn[[t]]$
(with each projection $e_i \in \fn_{\alpha_i}$ of $e$ necessarily non-zero, of course). 

Suppose $g \in G(K)$ is a $k$-point with $\pi(g) \in \ol{\Gr}_G^{\check{\lambda}}$ 
with $\Ad_g(e) \in \fg[[t]]$ having discrepancy $\ol{\check{\lambda}}$.
We will show
$g \in G(O)\check{\lambda}(t)$. Clearly this would suffice. 

By \eqref{eq:discrepancy-explicit}, we can find $\gamma \in G(O)$
and $h \in N(K)T(O)\check{\lambda}(t)$ such that:
\[
\Ad_{\gamma}\Ad_g(e) = \Ad_h(e).
\]

The assertion about $g$ depends only depends on its left $G(O)$-coset;
therefore, we may replace $g$ by $\gamma g$ to instead write:
\[
\Ad_g(e) = \Ad_h(e).
\]

By assumption, we can write $h = n \tau\check{\lambda}(t)$ for $n \in N(K)$
and $\tau \in T(O)$. 

Observe that $g h^{-1}$ centralizes $e$; by standard facts about 
regular nilpotent elements, this implies $g h^{-1} \in N(K)Z_G(K)$ (we remind
that $Z_G \subset G$ is the center of $G$).
This implies:
\[
g = gh^{-1}\cdot h \in N(K)T(O)Z_G(K) \check{\lambda}(t) = 
 \check{\lambda}(t)N(K)T(O)Z_G(K).
\]

\noindent Therefore:
\[
\pi(g) \in \coprod_{\check{\zeta} \in \check{\Lambda}_{Z(G)^{\circ}}} 
(\check{\lambda}+\check{\zeta})(t)N(K) \subset \Gr_G.
\]

\noindent (By standard convention, we have omitted a $\pi$ before 
$(\check{\lambda}+\check{\zeta})(t)$.)

It is well-known (cf. \cite{mirkovic-vilonen}) 
that for $\check{\eta} \in \check{\Lambda}^+$, we 
have:
\[
\ol{\Gr}_G^{\check{\lambda}} \cap \check{\eta}(t) N(K) \neq \emptyset
\Leftrightarrow 0 \leq \check{\eta} \leq \check{\lambda}.
\]

It follows that:
\[
\pi(g) \in \ol{\Gr}_G^{\check{\lambda}} \cap \check{\lambda}(t) N(K).
\]

\noindent In addition, it is well-known (cf. \cite{mirkovic-vilonen}) that
this intersection is the single point $\check{\lambda}(t)$.

Therefore, we see that $\pi(g) = \check{\lambda}(t)$, as desired.

\end{proof}

\subsection{Proof of Theorem \ref{t:hecke-to-kostant}}

\subsubsection{A review on bounding singular support from below}

Let $f:\sH \to \sY$ be a map of algebraic stacks with $\sY$ smooth
and let $\Lambda \subset T^*\sH$ be a closed conical substack.
An \emph{isolated pair for $\Lambda$} is a point $(x,\xi) \in \sH \times_{\sY} T^*\sY$  
with $x \in \sH$ (a field-valued point) and $\xi \in T_{f(x)}^*\sY$
such that:

\begin{itemize}

\item $df(\xi) \in \Lambda|_x$.

\item The intersection $\sH \times_{\sY} T^*\sY \cap df^{-1}(\Lambda)$
is zero-dimensional at $(x,\xi)$.

\end{itemize}

We have:

\begin{thm}[\cite{agkrrv1} Theorem Theorem 20.1.3]\label{t:isolated-ss}

Suppose $\sF \in \Shv(\sH)$ with $(x,\xi)$ an isolated pair
for $\SS(\sF)$. Then $(f(x),\xi) \in \SS(f_{*,dR}(\sF))$.

\end{thm}

\begin{rem}

In the \'etale setting, one needs additional hypotheses;
cf. \cite{saito-cc} and \cite{agkrrv1} Remark 20.1.5.

\end{rem}

\subsubsection{Proof}

Choose a coweight $\check{\lambda} \in \check{\Lambda}$ such that:\footnote{The
funny indexing is chosen this way for later convenience.}
\[
\o{\Nilp}^{-w_0(\check{\lambda})+(2-2g)\check{\rho}} \subset \SS(\sF).
\]

\noindent We may do this as we assumed $\SS(\sF) \subset \Nilp$ and 
$\o{\Nilp} \cap \SS(\sF) \neq \emptyset$,
so some irreducible component of $\o{\Nilp}$ must lie in $\SS(\sF)$.

The relevance (cf. \S \ref{sss:springer-desc}) of 
$-w_0(\check{\lambda})+(2-2g)\check{\rho}$ exactly means that
$\check{\lambda} \in \check{\Lambda}^+$ (equivalently:
$-w_0(\check{\lambda}) \in \check{\Lambda}^+$).

Let $x \in X$ be a $k$-point and let $\check{\lambda}$ be a dominant coweight.
Let $\o{\sS}^{\check{\lambda}} \in D(\Gr_{G,x})$ denote the $*$-extension 
of the dualizing sheaf on the $\check{\lambda}$-orbit. By geometric Satake,
this object has a finite filtration in the derived category with associated
graded objects being usual spherical sheaves (up to shifts). 
Therefore, by the standard interaction of singular support
in exact triangles, it suffices to show that 
$\SS(\o{\sS}^{\check{\lambda}} \star \sF) \cap \Nilp^{\Kos} \neq \emptyset$.

Indeed, let $\Bun_G \overset{p_1}{\leftarrow} \sH_x^{\check{\lambda}} 
\xar{p_2} \Bun_G$
be the corresponding Hecke correspondence associated to $x$ and the
coweight $\check{\lambda}$, normalized so the $p_2$ has fibers 
that are twisted forms of $\Gr_G^{\check{\lambda}}$.
We obtain a correspondence:
\[
\begin{tikzcd}
&
T^*\Bun_G \underset{\Bun_G}{\times} \sH_x^{\check{\lambda}}
\arrow[dl]
\arrow[dr]
&&
T^*\Bun_G \underset{\Bun_G}{\times} \sH_x^{\check{\lambda}}
\arrow[dl]
\arrow[dr]
\\
T^*\Bun_G 
&&
T^*\sH_x^{\check{\lambda}}
&&
T^*\Bun_G
\end{tikzcd}
\]

\noindent where the first correspondence relates to $p_1$ and the
second relates to $p_2$.
As in \cite{agkrrv1} \S 20, 
the composition of these two correspondences is:
\[
\begin{tikzcd}
&
\{((\sP_{G,1},\vph_1),(\sP_{G,2},\vph_2),\tau)\}
\arrow[dl,"\alpha",swap]
\arrow[dr,"\beta"]
\\
T^*\Bun_G
&&
T^*\Bun_G
\end{tikzcd}
\]

\noindent where the top line indicates
that $(\sP_{G,i},\vph_i) \in \Higgs_G$ ($i=1,2$) 
and $\tau$ is an isomorphism of these Higgs bundles 
away from $x$ so that the underlying isomorphism
$\sP_{G,1}|_{X\setminus x}\simeq \sP_{G,2}|_{X\setminus x}$
has relative position $\check{\lambda}$.

The singular support of $p_1^!(\sF)$ is computed by the usual naive estimate because
$p_1$ is a smooth.
Applying Theorem \ref{t:isolated-ss}, we find that
it suffices to find points $(\sP_{G,1},\vph_1) \in \SS(\sF) \subset T^*\Bun_G$
and a point $(\sP_{G,2},\vph_2) \in T^*\Bun_G$ such that:
\[
(\sP_{G,1},\vph_1),(\sP_{G,2},\vph_2),\tau) \in 
(\alpha\times\beta)^{-1}\big(\SS(\sF) \times (\sP_{G,2},\vph_2)\big)
\]
and the right hand side should be zero-dimensional at this point; 
indeed, in this case, we necessarily 
have $(\sP_{G,2},\vph_2) \in \SS(f_{*,dR}\sF)$ (by Theorem \ref{t:isolated-ss}).

Let $\sP_{G,1}$ be the $G$-bundle induced from the $T$-bundle
$(\check{\rho})(\Omega_X^1)(-\check{\lambda}\cdot x)$.\footnote{For the reader's
convenience in verifying some formulae below, we note that if 
we twist by $w_0$, we see that
this $G$-bundle is also induced from the $T$-bundle 
$(-\check{\rho})(\Omega_X^1)(-w_0(\check{\lambda})\cdot x)$.} 
There is a canonical nilpotent Higgs field $f_D^{\on{glob}}$ on $\sP_{G,1}$, 
generalizing the $D = 0$ case from \S \ref{sss:kost-defin}.

Note that the point $(\sP_{G,1},f_D^{\on{glob}})$ lies in 
$\o{\Nilp}^{-w_0(\check{\lambda})+(2-2g)\check{\rho}}$, so 
lies in $\SS(\sF)$. 

We take $(\sP_{G,2},\vph_2)$ to be the base-point of $\o{\Nilp}^{\Kos}$;
i.e., we apply the above construction with $D$ replaced by $0$.
We have an evident choice of Hecke modification $\tau$.

By Proposition \ref{p:spr}, we have:
\begin{equation}\label{eq:alpha-beta-hecke}
(\alpha\times\beta)^{-1}\big(\o{\Nilp}^{-w_0(\check{\lambda})+(2-2g)\check{\rho}} 
\times (\sP_{G,2},\vph_2)\big)
\end{equation}

\noindent is the single point constructed above.

In more detail: Hecke modifications of $(\sP_{G,2},\vph_2)$ are
determined by their restrictions to the formal neighborhood of $x$.
Moreover, any point $(\widetilde{\sP}_{G,1},\widetilde{\vph}_1,\tau)$ 
of \eqref{eq:alpha-beta-hecke} must have that the discrepancy divisor
of $\widetilde{\vph}_1$ is supported only at $x$, since $(\sP_{G,2},\vph_2)$
has vanishing discrepancy divisor and the two are isomorphism away from
$x$. As we know the degree of $\widetilde{\sP}_{G,1}$ (as it is
a $\check{\lambda}$-modification of $\sP_{G,2}$) and have
assumed $(\widetilde{\sP}_{G,1},\widetilde{\vph}_1) \in 
\o{\Nilp}^{-w_0(\check{\lambda})+(2-2g)\check{\rho}}$, the
discussion of Remark \ref{r:disc-c1} determines the value
of the discrepancy divisor, which here is found to be
$\ol{-w_0(\check{\lambda})}$, i.e., the image of
$-w_0(\check{\lambda})$ in $\check{\Lambda}_{G^{\ad}}$. 
We track signs: $\sP_{G,2}$ is a
modification of $\sP_{G,1}$ of type $\check{\lambda}$,
so $\sP_{G,1}$ is a modification of $\sP_{G,2}$ of type $-w_0(\check{\lambda})$.
Choosing arbitrarily a trivialization of $\sP_{G,2}$ on the formal
disc of $x$, the proposition now applies and yields our claim.

\section{Conservativeness of Whittaker coefficients}\label{s:finale}

At this point, the proof of the main theorem 
is essentially just a matter of combining our previous results.
Specifically, this is true in the nilpotent setting; we deduce
the assertion for general $D$-modules using a straightforward
application of the method of \cite{agkrrv1} \S 21.

\subsection{Conservativeness for nilpotent sheaves}

\subsubsection{}

Let $\Nilp^{\neq\Kos} \subset \Nilp$ be the union of all
components of $\Nilp$ \emph{besides} $\Nilp^{\Kos}$;
this is a closed conical subset of $\Nilp$ because $\Nilp^{\Kos}$
is open in $\Nilp$.

We have:

\begin{lem}\label{l:cons-kos}

$\Shv_{\Nilp^{\neq \Kos}}(\Bun_G)$ is exactly the kernel
of the functor $\coeff:\Shv_{\Nilp}(\Bun_G) \to \Vect$.

\end{lem}

\begin{proof}

\step 

First, we show $\Shv_{\Nilp^{\neq \Kos}}(\Bun_G) \subset \Ker(\coeff)$.

Note that $\Shv_{\Nilp^{\neq \Kos}}(\Bun_G) \subset \Shv_{\Nilp}(\Bun_G)$
is closed under truncations and subobjects by definition of singular
support. Moreover, again by definition, this category is left complete
for its $t$-structure; in particular, an object is zero if and only
if all its cohomology groups are zero.

Therefore, by $t$-exactness (up to shift) of $\coeff$ (Theorem \ref{t:coeff-exact}),
it suffices to show $\coeff(\sF) = 0$ for 
$\sF \in \Shv_{\Nilp^{\neq \Kos}}(\Bun_G)^{\heart}$.
Any such object is the union of its constructible subobjects, so 
we may assume $\sF$ is constructible (by exactness again).

In this case, $\coeff(\sF) \in \Vect^c$ and lies in a single
cohomological degree. Therefore, it suffices to show that its
Euler characteristic is zero. This follows from the assumption 
on $\sF$ and from the index
theorem, Theorem \ref{t:index}. 

\step 

Next, we show that if $\sF \in \Shv_{\Nilp}(\Bun_G)$ with 
$\Nilp^{\Kos} \subset \SS(\sF)$,
then $\coeff(\sF) \neq 0$.

The logic is the essentially same as in the previous step.
There exists some integer $i$ and some construcible subobject
$\sG \subset H^i(\sF)$ such that $\Nilp^{\Kos} \subset \SS(\sG)$.

In this case, $\coeff(\sG)$ is concentrated in cohomological
degree $\dim \Bun_G$. Moreover, it is non-zero 
by Theorem \ref{t:index}; here we remind that for objects 
of $\Shv^{\heart}$, the characteristic
cycle assigns \emph{positive} integers to components of the singular support.

By exactness of $\coeff$, we then have:
\[
H^{\dim \Bun_G}(\coeff(\sG)) \into H^{\dim \Bun_G}(\coeff(H^i(\sF))) = 
H^{\dim \Bun_G+i}\coeff(\sF).
\]

\noindent Clearly this implies $\coeff(\sF) \neq 0$.

\end{proof}

\begin{cor}\label{c:main-nilp}

Suppose $\sF \in \Shv_{\Nilp}(\Bun_G)$ has non-zero projection
to $\Shv_{\Nilp}(\Bun_G)^{\temp}$. 

Then there is a $\check{\Lambda}^+$-valued
divisor $D$ on $X$ such that $\coeff_D(\sF) \neq 0$.

\end{cor}

\begin{proof}

By assumption, $\sF \not\in \Shv_{\Nilp}(\Bun_G)^{\at}$.
Therefore, by Theorem \ref{t:at},
$\sF \cap \o{\Nilp} \neq \emptyset$.

Therefore, by Theorem \ref{t:hecke-to-kostant}, there is 
a $\check{\Lambda}^+$-valued divisor $D$ on $X$ such that
$\Nilp^{\Kos} \subset \SS(\sV^D \star \sF)$.
In this case, by Lemma \ref{l:cons-kos} and 
Theorem \ref{t:cs}, we obtain:
\[
\coeff_D(\sF) = \coeff(\sV^D \star \sF) \neq 0.
\]

\end{proof}

\subsection{Enhanced coefficient functors}\label{ss:coeff-enh}

It is now convenient to introduce the following functor encoding
all Whittaker coefficient functors simultaneously. 

\subsubsection{Nilpotent setting}

Recall the prestack $\LocSys_{\check{G}}^{\on{restr}}$ from \cite{agkrrv1}.
By \emph{loc. cit}. Theorem 14.3.2, there is a canonical \emph{spectral} action 
of $\QCoh(\LocSys_{\check{G}}^{\on{restr}})$ on $\Shv_{\Nilp}(\Bun_G)$
that is suitably compatible with Hecke functors.

Moreover, the category $\QCoh(\LocSys_{\check{G}}^{\on{restr}})$ is
canonically self-dual by \cite{agkrrv1} Corollary 7.8.9.
As in \emph{loc. cit}., we let:
\[
\Gamma_! = \Gamma_!(\LocSys_{\check{G}}^{\on{restr}},-):\QCoh(\LocSys_{\check{G}}^{\on{restr}})
\to \Vect
\]

\noindent denote the functor dual to the structure sheaf
$\sO_{\LocSys_{\check{G}}^{\on{restr}}} \in \QCoh(\LocSys_{\check{G}}^{\on{restr}})$.

\subsubsection{}\label{sss:shv-nilp-coeff-enh}

On formal grounds, we obtain a canonical functor:
\[
\coeff^{\enh}:\Shv_{\Nilp}(\Bun_G) \to \QCoh(\LocSys_{\check{G}}^{\on{restr}})
\]

\noindent fitting into a commutative diagram:
\begin{equation}\label{eq:coeff-enh}
\begin{tikzcd}
\Shv_{\Nilp}(\Bun_G)
\arrow[d,"\coeff^{\enh}"]
\arrow[drr,"\coeff"]
\\
\QCoh(\LocSys_{\check{G}}^{\on{restr}})
\arrow[rr,"\Gamma_!"]
&&
\Vect.
\end{tikzcd}
\end{equation}

Namely, the construction proceeds as follows.
We have an action functor:
\[
\QCoh(\LocSys_{\check{G}}^{\on{restr}}) \otimes 
\Shv_{\Nilp}(\Bun_G) \to \Shv_{\Nilp}(\Bun_G).
\]

\noindent Dualizing the first tensor factor and applying its self-duality,
we obtain a functor:
\[
\Shv_{\Nilp}(\Bun_G) \to 
\QCoh(\LocSys_{\check{G}}^{\on{restr}}) \otimes \Shv_{\Nilp}(\Bun_G).
\]

\noindent Now compose this functor with $\id \otimes \coeff$.
By construction, this functor has the desired property.

By construction, $\coeff^{\enh}$ factors through $\Shv_{\Nilp}(\Bun_G)^{\temp}$.
We abuse notation in also denoting this functor by $\coeff^{\enh}$.

\subsubsection{}\label{sss:coeffd-coeffenh}

More generally, for $\sG \in \QCoh(\LocSys_{\check{G}}^{\on{restr}})$
and $\sF \in \Shv_{\Nilp}(\Bun_G)$, if we let $\star$ denote the
action of the former category on the latter, we have:
\[
\coeff(\sG \star \sF) = 
\Gamma_!(\LocSys_{\check{G}}^{\on{restr}},\sG \otimes \coeff^{\enh}(\sF)).
\]

Recalling that Hecke functors (at points) factor through the
action of $\QCoh(\LocSys_{\check{G}}^{\on{restr}})$ and applying
Theorem \ref{t:cs}, we see that $\coeff_D(\sF)$
can be algorithmically extracted from $\coeff^{\enh}(\sF)$.

In particular, we obtain:

\begin{cor}\label{c:coeff-enh-nilp}

The functor:
\[
\coeff^{\enh}:\Shv_{\Nilp}(\Bun_G)^{\temp} \to \QCoh(\LocSys_{\check{G}}^{\on{restr}})
\]

\noindent is conservative.

\end{cor}

Indeed, this is immediate from Corollary \ref{c:main-nilp} and Lemma \ref{l:cons-kos}.

\subsubsection{Variant for $D$-modules}

Recall from \cite{generalized-vanishing} (see also \cite{dennis-laumonconf} 
\S 4.3-4.5 and \S 11.1) that there
is a canonical action of $\QCoh(\LocSys_{\check{G}})$ on $D(\Bun_G)$,
similar to the above functors.

As $\LocSys_{\check{G}}$ is a QCA stack, $\QCoh(\LocSys_{\check{G}})$
is canonically self-dual; this time, the functor dual to 
$\sO$ is simply usual global sections.

Therefore, we obtain a functor:
\[
\coeff^{\enh}:D(\Bun_G) \to \QCoh(\LocSys_{\check{G}})
\]

\noindent fitting into a commutative diagram:
\[
\begin{tikzcd}
D(\Bun_G)
\arrow[d,"\coeff^{\enh}"]
\arrow[drr,"\coeff"]
\\
\QCoh(\LocSys_{\check{G}})
\arrow[rr,"\Gamma"]
&&
\Vect.
\end{tikzcd}
\]

\subsubsection{}\label{sss:coeff-enh-compats}

We now state the natural compatibility between the above two constructions.

Let $\iota:\LocSys_{\check{G}}^{\on{restr}} \to \LocSys_{\check{G}}$
denote the natural map. The symmetric monoidal functor:
\[
\iota^*:\QCoh(\LocSys_{\check{G}}) \to \QCoh(\LocSys_{\check{G}}^{\on{restr}})
\]

\noindent admits a left adjoint $\iota_?$ (denoted $\iota_!$ in 
\cite{agkrrv1} \S 7.1.3); this functor is naturally the
dual to $\iota^*$ for the self-duality of both sides.
In particular, $\Gamma(\LocSys_{\check{G}},\iota_?(-)) = \Gamma_!$.

We obtain commutative diagrams:
\[
\begin{tikzcd}
\Shv_{\Nilp}(\Bun_G)
\arrow[rr,"\coeff^{\enh}"]
\arrow[d]
&&
\QCoh(\LocSys_{\check{G}}^{\on{restr}})
\arrow[d,"\iota_?"]
\\
D(\Bun_G)
\arrow[rr,"\coeff^{\enh}"]
&&
\QCoh(\LocSys_{\check{G}})
\end{tikzcd}
\]

\noindent and:
\[
\begin{tikzcd}
\Shv_{\Nilp}(\Bun_G)
\arrow[rr,"\coeff^{\enh}"]
&&
\QCoh(\LocSys_{\check{G}}^{\on{restr}})
\\
D(\Bun_G)
\arrow[rr,"\coeff^{\enh}"]
\arrow[u]
&&
\QCoh(\LocSys_{\check{G}})
\arrow[u,"\iota^*",swap].
\end{tikzcd}
\]

\noindent Here the left arrow in the second diagram is
the right adjoint to the embedding, and the commutativity of
this latter diagram follows from \cite{agkrrv1} Proposition 14.5.3.

\subsection{Conservativeness for general $D$-modules} 

We now conclude the proof of our main result, Theorem \ref{t:temp-cons} below.

\subsubsection{Field extensions}\label{sss:temperedness-field-extns}

We briefly digress to discuss field extensions.

Suppose $k^{\prime}/k$ is a (possibly transcendental) field extension.
For $\sY$ over $k$, we let $\sY_{k^{\prime}}$ denote the base-change
of $\sY$ to $k^{\prime}$.

For prestacks over $k^{\prime}$, we write $D_{/k^{\prime}}(-)$ to 
denote $D$-modules considered relative to the field $k^{\prime}$.
We use $\Shv_{/k^{\prime}}$ similarly: this means the ind-category
version of regular holonomic objects of $D_{/k^{\prime}}(-)$.

\subsubsection{}\label{sss:field-temp}

Let $x \in X(k)$ be fixed, and let
$x^{\prime} \in X^{\prime}(k^{\prime})$ be the induced point.

There is a natural equivalence:
\[
\sH_x^{\on{sph}} \otimes \Vect_{k^{\prime}} \simeq 
\sH_{x^{\prime}}^{\on{sph}}.
\]

\noindent Here the right hand side is taken to be defined with $D$-modules
over $k^{\prime}$, as above.
In other words, up to extending scalars, the spherical
Hecke categories are the same. This identification is compatible
in the natural sense with derived Satake.

It follows from the definitions that for $\sC \in \sH_x^{\on{sph}}\mod$ with 
$\sC^{\prime} \coloneqq \sC \otimes \Vect_{k^{\prime}}$, we have
commutative diagrams:
\[
\begin{tikzcd}
\sC^{\xat} 
\arrow[r]
\arrow[d]
&
\sC
\arrow[r]
\arrow[d]
&
\sC^{\xtemp}
\arrow[d]
\\
\sC^{\prime,\xat} 
\arrow[r]
&
\sC^{\prime}
\arrow[r]
&
\sC^{\prime,\xtemp}.
\end{tikzcd}
\]

\noindent There are similar functors if we work with the adjoints to the horizontal
arrows. Moreover, the vertical arrows induce isomorphisms
after tensoring with $\Vect_{k^{\prime}}$.

\subsubsection{General case}

Finally, we show:

\begin{thm}\label{t:temp-cons}

The functor
\[
\coeff^{\enh}:
D(\Bun_G)^{\temp} \to \QCoh(\LocSys_{\check{G}})
\]

\noindent is conservative.

\end{thm} 
 
Suppose $\sF \in D(\Bun_G)$ is given with
$\coeff^{\enh}(\sF) \in \QCoh(\LocSys_{\check{G}})$ vanishing.
We need to show that $\sF$ is anti-tempered, 
i.e., its image in $D(\Bun_G)^{\temp}$ is zero.

As in \cite{agkrrv1} Lemma 21.4.6, it suffices
to show that for any field extension $k^{\prime}/k$ and any 
$\sigma \in \LocSys_{\check{G}}(k^{\prime})$,
the image of $\sF$ in:
\[
D(\Bun_G)^{\temp}
\underset{\QCoh(\LocSys_{\check{G}})}{\otimes} \Vect_{k^{\prime}}
\]

\noindent is zero.

The functor: 
\[
\Shv_{/k^{\prime},\Nilp_{k^{\prime}}}(\Bun_{G,k^{\prime}}) 
\underset{\QCoh(\LocSys_{\check{G},k^{\prime}})}{\otimes} \Vect_{k^{\prime}}
\to 
D(\Bun_G)
\underset{\QCoh(\LocSys_{\check{G}})}{\otimes} \Vect_{k^{\prime}}
\]

\noindent is an equivalence by \cite{agkrrv1} Proposition 13.5.3.
The same applies for tempered variants by 
\S \ref{sss:field-temp}.

From \S \ref{sss:coeff-enh-compats}, we have a commutative diagram:
\begin{equation}\label{eq:finale}
\begin{tikzcd}
\Shv_{/k^{\prime},\Nilp_{k^{\prime}}}(\Bun_{G,k^{\prime}})^{\temp} 
\underset{\QCoh(\LocSys_{\check{G},k^{\prime}})}{\otimes} \Vect_{k^{\prime}}
\arrow[d,"\simeq"]
\arrow[dr]
\\ 
D(\Bun_G)^{\temp}
\underset{\QCoh(\LocSys_{\check{G}})}{\otimes} \Vect_{k^{\prime}}
\arrow[r] 
& 
\Vect_{k^{\prime}}.
\end{tikzcd}
\end{equation}

\noindent Here the rightward arrows are induced by extension of scalars 
from the functors $\coeff^{\enh}$.
The left arrow in this diagram is conservative by Corollary \ref{c:coeff-enh-nilp},
observing that we have a commutative diagram:
\[
\begin{tikzcd}
\Shv_{/k^{\prime},\Nilp_{k^{\prime}}}(\Bun_{G,k^{\prime}})^{\temp}
\underset{\QCoh(\LocSys_{\check{G},k^{\prime}})}{\otimes} \Vect_{k^{\prime}}
\arrow[d] 
\arrow[rr]
&&
\Shv_{/k^{\prime},\Nilp_{k^{\prime}}}(\Bun_{G,k^{\prime}})^{\temp}
\arrow[d,"\coeff^{\enh}"]
\\
\Vect_{k^{\prime}}
\arrow[rr]
&&
\QCoh(\LocSys_{\check{G},k^{\prime}})
\end{tikzcd}
\]

\noindent with horizontal and right arrows conservative.
Therefore, the right arrow in \eqref{eq:finale} is conservative,
implying $\sF$ maps to zero in this category, yielding the claim.

\section{The structure of Hecke eigensheaves}\label{s:eigensheaves}

In this section, we use our earlier results to deduce structural
properties of Hecke eigensheaves. The main result is Theorem \ref{t:eigensheaf}.

Throughout this section, to simplify the discussion a bit,
we assume the ground field $k$ to be 
algebraically closed (in addition to being of characteristic $0$).

\subsection{Setup}

\subsubsection{Notation for local systems}

Throughout this section, fix 
$\sigma \in \LocSys_{\check{G}}(k)$ an \emph{irreducible} $\check{G}$-local system
on $X$; i.e., $\sigma$ does not admit a reduction to any parabolic 
$\check{P} \subsetneq \check{G}$.

\subsubsection{}\label{sss:very-irred}

Our discussion will be nicest when $\sigma$ is \emph{very irreducible}
in the sense described below.  

Let $\Aut(\sigma)$ denote the algebraic\footnote{In
particular, $\Aut(\sigma)$ is something classical -- we ignore
the finer derived structures on $\Aut(\sigma)$.} 
group of automorphisms of $\sigma$ as a local system. 
Note that a choice of point $x \in X(k)$ induces an embedding
$\Aut(\sigma) \into \check{G}$. 

\begin{defin}

We say that $\sigma$ is \emph{very irreducible} if the natural
map $Z_{\check{G}} \to \Aut(\sigma)$ is an isomorphism;
here $Z_{\check{G}}$ is the center of $\check{G}$.

\end{defin}

\begin{rem}

Of course, for $GL_n$, an irreducible local system is very irreducible.

\end{rem}

\begin{rem}\label{r:oper-very-irred}

Very irreducible local systems form an open substack
of $\LocSys_{\check{G}}$ that is non-empty when the genus 
of the curve is greater than $1$. 
This follows from the
existence of opers (without singularity) 
on $X$, which are always very irreducible, 
cf. \cite{hitchin} \S 3.1.

Probably very irreducible local systems are dense in $\LocSys_{\check{G}}$
(for genus $>1$), but we are not sure.

\end{rem}

\begin{rem}\label{r:finite-gps}

It is not hard to see that $\Aut(\sigma)/Z_{\check{G}}$ is zero-dimensional.
Therefore, the gap between irreducible and very irreducible local
systems concerns finite groups.

\end{rem}

\begin{example}

For completeness, we provide an explicit (quite elementary)
example of an irreducible local system that is
not very irreducible. The group is $PGL_2$.

Suppose $k = \bC$ and $X$ has genus $2$; we freely use Riemann-Hilbert. 
It is simple to see
that the topological fundamental
group $\pi_1^{\on{top}}(X)$ surjects onto the symmetric group $S_3$.
Indeed, the former has 
standard generators $a_1,b_1,a_2,b_2$ with defining relation
$[a_1,b_1][a_2,b_2] = 1$, while the latter
can be generated by elements $r$ and $s$ with 
$r$ an element of order $3$ and
$s$ a transposition with defining relation $srs = r^{-1}$.
Then the map:
\[
a_1 \mapsto r, \hspace{5pt}
b_1 \mapsto s, \hspace{5pt}
a_2 \mapsto r^2, \hspace{5pt}
b_2 \mapsto s
\]

\noindent defines our desired surjective homomorphism.

Then $S_3$ has a unique (up to isomorphism) irreducible 
$2$-dimensional representation, which via the homomorphism
above induces an irreducible
$GL_2$-local system on $X$. The induced $PGL_2$-local system
is also irreducible, but is easily seen not to be 
\emph{very} irreducible. 

\end{example}

\subsubsection{}

Let $k_{\sigma} \in \QCoh(\LocSys_{\check{G}})^{\heart}$ denote the structure sheaf
of $\sigma$, i.e., the $*$-pushforward of $k \in \Vect \simeq \QCoh(\Spec(k))$ along the map:
\[
\Spec(k) \xar{\sigma} \LocSys_{\check{G}}
\]

\subsubsection{The main result}

The goal of this section is to outline a proof of the following result.

\begin{thm}\label{t:eigensheaf}

There exists $\sF_{\sigma} \in D(\Bun_G)$ an eigensheaf for
$\sigma$ such that:

\begin{itemize}

\item $\sF_{\sigma}$ is perverse up to shifts 
(i.e., locally compact, concentrated in cohomological degree $0$, and 
with regular singularities).

\item If $\sigma$ is very irreducible, 
the restriction of $\sF_{\sigma}$ to every connected component of $\Bun_G$
is an irreducible perverse sheaf.

\end{itemize}

In addition, one has:

\begin{itemize}

\item $\sF_{\sigma}$ is cuspidal.

\item $\coeff^{\enh}(\sF_{\sigma}) \simeq k_{\sigma}[-\dim \Bun_G] 
\in \QCoh(\LocSys_{\check{G}})$.

\end{itemize}

\end{thm}

We emphasize that the argument combines work of many other authors, and some portion of what follows is simply a matter
of proving a folklore result by stringing together results
of other authors.
In particular, the existence of a non-zero \emph{complex} of sheaves $\sF_{\sigma}$ (Theorem \ref{t:aut-exists} below)
is the culmination of work of many authors, notably Arinkin, Gaitsgory, Frenkel-Gaitsgory,
Beilinson-Drinfeld, and not us. 

We consider our contribution to be to the questions of perversity and 
irreducibility of $\sF_{\sigma}$. This is a classical question: the above
result proves \cite{frenkel-survey} Conjecture 1.\footnote{Technically,
\emph{loc. cit}. ignores the discrepancy between irreducible
and very irreducible local systems. For irreducible $\sigma$ that
are not very irreducible, the conjecture from 
\emph{loc. cit}. is not reasonable; cf. Remark \ref{r:jh-genl}.}
We highlight that it has previously been unknown how to deduce
any result of this type from the categorical properties
of the geometric Langlands conjecture.

\begin{rem}

At some points, complete proofs of some key assertions are missing in 
the literature. Most glaringly: we need 
a generalization of \cite{bg-deformations}; \emph{loc. cit}. is
written for the Borel only, and we need the (folklore) generalization
to a general parabolic subgroup.\footnote{We also remark that the 
literature has at other times appealed
to this same folklore generalization. See e.g. \cite{potential} Appendix A.} 

\end{rem}

\begin{rem}\label{r:jh-genl}

In the case of general irreducible $\sigma$, $\sF_{\sigma}$ ought to be a semi-simple 
perverse sheaf with irreducible factors indexed by isomorphism
classes irreducible representations $W_i \in \Rep(\Aut(\sigma))^{\heart}$,
with each simple factor $\sF_{\sigma,W_i}$ appearing with multiplicity
$\dim(W_i)$. We are unable to unconditonally prove this assertion at the moment,
but our methods combined with the categorical geometric Langlands
conjecture yield this conclusion, which seems not to have been previously
contemplated.

\end{rem}

\begin{rem}\label{r:whit-norm}

The condition $\coeff^{\enh}(\sF_{\sigma}) \simeq k_{\sigma}[-\dim \Bun_G] $
is often referred to as the \emph{Whittaker normalization} of
a Hecke eigensheaf (at least, up to conventions regarding the shift).
It plays a key role for us in what follows, cf. Theorem \ref{t:eigen->perv}.

\end{rem}

\begin{rem}

The assumption that $k$ is algebraically closed is used in the reference
to \cite{dima-opers}, and specifically to the existence of Whittaker
normalized $\sF_{\sigma}$; otherwise, we may a priori need to extend the ground field
to obtain an oper structure (conjecturally, this should not be necessary).

\end{rem}

\subsection{Formulation of intermediate results}

We now formulate a series of results from which we deduce
Theorem \ref{t:eigensheaf}.

\subsubsection{Existence of eigensheaves}

Crucially, we have:

\begin{thm}[Folklore]\label{t:aut-exists}

There exists an object $\sF_{\sigma} \in D(\Bun_G)$
such that:

\begin{itemize}

\item $\sF_{\sigma}$ is a Hecke eigensheaf with eigenvalue $\sigma$
(see \S \ref{sss:eigensheaf-defin} for the definition).

\item $\coeff^{\enh}(\sF_{\sigma}) \simeq k_{\sigma}[-\dim \Bun_G] \in 
\QCoh(\LocSys_{\check{G}})$.

\end{itemize}

\end{thm}

Briefly: the proof is via the Kac-Moody localization method pioneered
by Beilinson-Drinfeld, appealing to later developments
due to Frenkel-Gaitsgory, independent ideas of Gaitsgory, and Arinkin.
We review the relevant results in Appendix \ref{a:localization}.

In \S \ref{ss:perv-pf}, we will show the following result.

\begin{thm}\label{t:eigen->perv}

Any object $\sF_{\sigma}$ satisfying the conclusion of 
Theorem \ref{t:aut-exists} also satisfies the conclusion of
Theorem \ref{t:eigensheaf}.

\end{thm}

\begin{rem}[Application to \cite{hitchin}]\label{r:bd-irred}

Suppose $\sigma$ admits an oper structure (without singularities).
In this case, $\sigma$ is necessarily very irreducible 
(cf. Remark \ref{r:oper-very-irred}).
We can take $\sF_{\sigma}$ to be the $D$-module constructed in 
\cite{hitchin} \S 5.1.1.;\footnote{In \cite{hitchin}, our $\sigma$ is denoted
$\sF$ and our $\sF_{\sigma}$ is denoted $M_{\sF}$.} that
$\sF_{\sigma}$ is Whittaker-normalized is a special case of 
the discussion in Appendix \ref{a:localization}.
Therefore, Theorem \ref{t:eigen->perv} 
implies that the eigensheaves constructed in \cite{hitchin} 
are perverse sheaves with irreducible restrictions to each connected component
of $\Bun_G$, answering in the
affirmative a question of Beilinson-Drinfeld (see \cite{hitchin} \S 5.2.7).

\end{rem}

\subsubsection{Cuspidality}

We have the following result:

\begin{thm}[Braverman-Gaitsgory]\label{t:bg-eis}

Let $P \subset G$ be a parabolic subgroup with Levi quotient $M$.
Let $\check{P} \subset \check{G}$ be the dual parabolic.

Let $\QCoh(\LocSys_{\check{G}})_{\check{P}} \subset 
\QCoh(\LocSys_{\check{G}})$ 
denote the full subcategory of objects supported (set theoretically) 
on the image of the (proper) morphism 
$\LocSys_{\check{P}} \to \LocSys_{\check{G}}$.

Then the composition:
\[
D(\Bun_M) \xar{\on{Eis}_!} D(\Bun_G)
\]

\noindent maps into the full subcategory:
\[
D(\Bun_G) 
\underset{\QCoh(\LocSys_{\check{G}})}{\otimes} 
\QCoh(\LocSys_{\check{G}})_{\check{P}} \subset
D(\Bun_G) 
\underset{\QCoh(\LocSys_{\check{G}})}{\otimes} 
\QCoh(\LocSys_{\check{G}}) = D(\Bun_G).
\]

\end{thm}

\begin{proof}

For $\check{P} = \check{B}$, this follows from the
Hecke property of Eisenstein series shown in 
\cite{bg-deformations} Theorem 8.8 (see also \emph{loc. cit}.
Theorem 1.11). 

In general, it is expected that the results of 
\cite{bg-deformations} generalize without major changes
to parabolics. In particular, the assertion of the present
theorem is asserted (and refined) by Gaitsgory in
\cite{dennis-laumonconf} Proposition 11.1.3.
 
\end{proof}

Let $\LocSys_{\check{G}}^{\on{irred}} \subset \LocSys_{\check{G}}$
denote the open parametrizing irreducible $\check{G}$-local systems;
i.e., $\LocSys_{\check{G}}^{\on{irred}}$ is the complement to the
images of the maps $\LocSys_{\check{P}} \to \LocSys_{\check{G}}$ for
$\check{P} \neq \check{G}$. Let $j$ denote the relevant open embedding.

By adjunction, we have adjoint functors:
\[
j^*:D(\Bun_G) \rightleftarrows 
D(\Bun_G) 
\underset{\QCoh(\LocSys_{\check{G}})}{\otimes}
\QCoh(\LocSys_{\check{G}}^{\on{irred}})
:j_*
\]

\noindent with $j_*$ being fully faithful.

\begin{cor}

For $M$ a Levi besides $G$, the composition:
\[
D(\Bun_M) \xar{\on{Eis}_!} D(\Bun_G) \xar{j^*} 
D(\Bun_G) 
\underset{\QCoh(\LocSys_{\check{G}})}{\otimes}
\QCoh(\LocSys_{\check{G}}^{\on{irred}})
\]

\noindent is zero.

\end{cor}

As cuspidal objects of $\Bun_G$ are exactly those objects
in the right orthogonal to Eisenstein series along proper parabolics
(cf. \cite{geometric-ct} \S 1.4), we obtain:

\begin{cor}\label{c:irred->cusp}

Any object of $D(\Bun_G)$ in the essential image
of the functor:

\[ 
j_*:D(\Bun_G) 
\underset{\QCoh(\LocSys_{\check{G}})}{\otimes}
\QCoh(\LocSys_{\check{G}}^{\on{irred}})
\to D(\Bun_G)
\]
 
\noindent is cuspidal.

\end{cor}

\begin{rem}

The geometric Langlands conjectures predict that conversely,
any cuspidal object of $D(\Bun_G)$ lies in the essential image
of the above functor. This converse appears to be out of reach using 
present methods. 

\end{rem}

\subsection{Proof of Theorem \ref{t:eigen->perv}}\label{ss:perv-pf}

\subsubsection{}\label{sss:eigensheaf-defin}

Let $D_{\sigma}(\Bun_G)$ denote the category:
\[
D_{\sigma}(\Bun_G) \coloneqq 
\Hom_{\QCoh(\LocSys_{\check{G}})\mod}(\Vect,D(\Bun_G)) \simeq  
D(\Bun_G) \underset{\QCoh(\LocSys_{\check{G}})}{\otimes} 
\Vect
\]

\noindent where $\QCoh(\LocSys_{\check{G}})$ acts on $\Vect$ 
via the symmetric monoidal functor of pullback along
$\sigma:\Spec(k) \to \LocSys_{\check{G}}$ and we are using
self-duality of $\Vect$ in the above identification. 

By definition, a \emph{Hecke eigensheaf} with eigenvalue $\sigma$ 
is an object of this category $D_{\sigma}(\Bun_G)$.

Let $D_{\widehat{\sigma}}(\Bun_G) \subset D(\Bun_G)$ denote the
full subcategory generated under colimits by the essential image
of $D_{\sigma}(\Bun_G) \to D(\Bun_G)$.

We now have:

\begin{lem}[\cite{agkrrv1}]\label{l:shv-sigma}

\begin{enumerate}

\item\label{i:shv-sigma-1} 
Every object of $D_{\widehat{\sigma}}(\Bun_G)$ lies in 
$\Shv_{\Nilp}(\Bun_G)$ and has regular singularities.

\item\label{i:shv-sigma-2} 
The embedding $D_{\widehat{\sigma}}(\Bun_G) \into \Shv_{\Nilp}(\Bun_G)$
extends to a decomposition:
\[
\Shv_{\Nilp}(\Bun_G) = 
D_{\widehat{\sigma}}(\Bun_G) \times \Shv_{\Nilp,\neq \sigma}(\Bun_G).
\]

\end{enumerate}

\end{lem}

\begin{proof}

These are all structural results from \cite{agkrrv1}. 

That every object of $D_{\widehat{\sigma}}(\Bun_G)$ lies in 
$\Shv_{\Nilp}(\Bun_G)$ (and in particular, has regular singularities) 
is \cite{agkrrv1} Proposition 14.5.3 combined with 
\emph{loc. cit}. Main Corollary 16.5.6.

The asserted product decomposition follows from \cite{agkrrv1} Corollary 14.3.5.

\end{proof}

\begin{rem}

Only \eqref{i:shv-sigma-2} above uses irreducibility of $\sigma$.

\end{rem}

\subsubsection{}

By Lemma \ref{l:shv-sigma} \eqref{i:shv-sigma-2}, the subcategory: 
\[
D_{\widehat{\sigma}}(\Bun_G) \subset \Shv_{\Nilp}(\Bun_G)
\]

\noindent is closed under truncation functors for the natural $t$-structure
on the right hand side; therefore, this subcategory inherits a canonical $t$-structure
(uniquely characterized by $t$-exactness of the above embedding).

\subsubsection{}

We now consider the functor:
\[
\coeff_{\widehat{\sigma}}^{\enh}:D_{\widehat{\sigma}}(\Bun_G) \to 
\QCoh(\LocSys_{\check{G}}^{\on{restr}})
\]

\noindent given as the restriction of $\coeff^{\enh}$
(in its $\Shv_{\Nilp}$ incarnation, cf. \S \ref{sss:shv-nilp-coeff-enh}).

By the argument of \S \ref{sss:temp-gl-exact-intro}, combining
Theorem \ref{t:coeff-exact}, Theorem \ref{t:hecke-exact}, and 
\cite{agkrrv1} Theorem 1.4.5, we find that
the functor: 
\[
\coeff_{\widehat{\sigma}}^{\enh}[\dim \Bun_G]:\Shv_{\Nilp}(\Bun_G) \to 
\QCoh(\LocSys_{\check{G}}^{\on{restr}})
\]

\noindent is $t$-exact. The same applies for $\coeff_{\widehat{\sigma}}^{\enh}$.

Moreover, by Corollary \ref{c:irred->cusp}, any object of 
$D_{\widehat{\sigma}}(\Bun_G)$ is cuspidal, hence, by \cite{dario-ramanujan},
tempered. Therefore, Theorem \ref{t:temp-cons} implies that
$\coeff_{\widehat{\sigma}}^{\enh}$ is conservative, as well as $t$-exact
(up to shift).

\subsubsection{Irreducible case}

We now clearly obtain the assertion of Theorem \ref{t:eigen->perv} in
the case of (possibly not very) irreducible $\sigma$.

Namely, for $\sF_{\sigma} \in D_{\sigma}(\Bun_G)$ as in the statement of the 
theorem. We abuse notation in also letting $\sF_{\sigma}$ denote
the corresponding object of $D_{\widehat{\sigma}}(\Bun_G)$. 

We have already noted that any object of $D_{\widehat{\sigma}}(\Bun_G)$
has regular singularities and is cuspidal.

For $i \in \bZ$, $t$-exactness of $\coeff^{\enh}$ yields:
\[
\coeff^{\enh}(H^i(\sF_{\sigma})) = 
H^{i+\dim \Bun_G}(\sF_{\sigma})[-\dim \Bun_G].
\]

\noindent By assumption on $\sF_{\sigma}$, the right hand side vanishes for 
$i \neq 0$. As $\coeff^{\enh}$ is conservative, this means
$H^i(\sF_{\sigma}) = 0$ for $i \neq 0$, so $\sF_{\sigma}$ is ind-perverse.

We prove that $\sF_{\sigma}$ is perverse in Remarks \ref{r:not-very-1}
and \ref{r:not-very-2} below.

\subsubsection{Very irreducible case; $G$ is semi-simple}\label{sss:ss}

We now prove Theorem \ref{t:eigen->perv} for $\sigma$ very irreducible. 
The argument is more direct for $G$ semi-simple, so we impose this assumption
at first. 

Recall that the connected components of $\Bun_G$ are labeled by 
elements $c \in \pi_1^{\on{alg}}(G)$. For $c \in \pi_1^{\on{alg}}(G)$,
let $\Bun_G^c$ denote the corresponding connected component and let
$\sF_{\sigma} = \oplus_{c \in \pi_1^{\on{alg}}(G)} \sF_{\sigma}^c$ denote
the decomposition of $\sF_{\sigma}$ by connected components 
(so $\sF_{\sigma}^c$ is the restriction of $\sF_{\sigma}$ to 
$\Bun_G^c$). 

It suffices to show:

\begin{lem}\label{l:length}

\begin{enumerate}

\item\label{i:length-lower} For each $c \in \pi_1^{\on{alg}}(G)$, 
$\sF_{\sigma}^c$ is non-zero.

\item\label{i:length-upper} The length of $\sF_{\sigma} \in D(\Bun_G)^{\heart}$ 
as a perverse\footnote{A priori,
$\sF$ is ind-perverse. The finiteness in this bound amounts to perversity.} sheaf
is $\leq |\pi_1^{\on{alg}}(G)|$.

\end{enumerate}

\end{lem}

Indeed, this suffices as we then have (for $\ell$ denoting length):
\[
|\pi_1^{\on{alg}}(G)| \leq \sum_{c \in \pi_1^{\on{alg}}(G)} \ell(\sF_{\sigma}^c)
= \ell(\sF_{\sigma}) \leq |\pi_1^{\on{alg}}(G)|
\]

\noindent with the first inequality being \eqref{i:length-lower} and the
second being \eqref{i:length-upper}; this forces the inequalities to be
equalities, which then forces each summand to be exactly $1$, as desired.

\begin{proof}[Proof of Lemma \ref{l:length}]

First, \eqref{i:length-lower} is immediate from the eigensheaf property,
since $\sF_{\sigma}$ is non-zero. 

Alternatively, we can see this from the Whittaker normalization. Calculating
$\coeff_D(\sF_{\sigma})$ from $\coeff^{\enh}(\sF_{\sigma}) = 
k_{\sigma}[-\dim \Bun_G]$
(as in \S \ref{sss:coeffd-coeffenh}), we find that $\coeff_D(\sF_{\sigma})$
is non-zero for \emph{every} $D$; however, by definition, we have:
\[
\coeff_D(\sF_{\sigma}) = \coeff_D(\sF_{\sigma}^c)
\]

\noindent where $c \in \pi_1^{\on{alg}}(G)$ 
is the class represented by the cocharacter
$(\check{\rho}\cdot\deg(\Omega_X^1))-\deg(D)$. This clearly implies that $\sF_{\sigma}^c$
must be non-zero for each $c$.

We now prove \eqref{i:length-upper}. Note that any subquotient of
$\sF_{\sigma} \in D(\Bun_G)$ also lies in $D_{\widehat{\sigma}}(\Bun_G)^{\heart}$. 
As $\coeff_{\widehat{\sigma}}^{\enh}[\dim\Bun_G]$ 
is conservative and $t$-exact,
it follows that:
\begin{equation}\label{eq:length-ineq}
\ell(\sF_{\sigma}) \leq 
\ell(\coeff_{\widehat{\sigma}}^{\enh}(\sF_{\sigma})[\dim\Bun_G]) = 
\ell(k_{\sigma}).
\end{equation}

\noindent But $k_{\sigma} \in \QCoh(\LocSys_{\check{G}})^{\heart}$ 
is calculated as a pushforward along the composition:
\[
\Spec(k) \to \bB Z_{\check{G}} \into \LocSys_{\check{G}}
\]

\noindent where the second map is a closed embedding by very irreducibility of
$\sigma$. Therefore, the length of $k_{\sigma}$ is the same as the
length of the regular representation of $Z_{\check{G}}$. As 
$Z_{\check{G}}$ is finite abelian, we have:
\[
\ell(k_{\sigma}) = |Z_{\check{G}}| = |\pi_1^{\on{alg}}(G)|
\]

\noindent as desired.

\end{proof}

\begin{rem}\label{r:not-very-1}

In the case of (not necessarily very) irreducible $\sigma$, 
the statement of Lemma \ref{l:length} \eqref{i:length-upper}
should instead say $\ell(\sF_{\sigma}) \leq \sum \dim(W_i)$ for
notation as in Remark \ref{r:jh-genl}, i.e., the $W_i$ are isomorphism
classes of irreducible representations of $\Aut(\sigma)$, 
the automorphism group of $\sigma$. We note that the 
same argument as in Lemma \ref{l:length} \eqref{i:length-upper} yields
this bound. As in Remark \ref{r:finite-gps}, 
$\Aut(\sigma)$ is a finite group (when $G$ is semi-simple), so 
this upper bound is finite. We again reiterate: the categorical geometric Langlands
conjecture predicts that this upper bound is an equality; however, 
it is not clear how to a priori obtain the lower bound in this
setting without the categorical conjecture.

In particular, this estimate implies that $\sF_{\sigma}$ is perverse (not ind-perverse)
under this relaxed hypothesis, resolving a leftover point from 
\S \ref{sss:ss} (for $G$ semi-simple).

\end{rem}

\subsubsection{Very irreducible case; general reductive $G$}\label{sss:red}

In the previous section, it was important that $\sF_{\sigma}$ ultimately
had finite length, so that the inequality \eqref{eq:length-ineq} was 
forced to be an equality. In turn, this corresponded to the geometric fact that
$\LocSys_{\check{G}}$ is Deligne-Mumford in a neighborhood of $\sigma$.
This will not be the case when  $\pi_1^{\on{alg}}(G)$ is infinite; 
we describe the remedy for below.

Let $Z_G^\circ$ be the connected component of the identity in the 
center of $G$. Let $\check{\Lambda}_{Z_G^\circ} 
\subset \check{\Lambda}$ be the sublattice of coweights of the torus
$Z_G^\circ$. Note that the torus dual to $Z_G^\circ$ is $\check{G}^{\on{ab}}$, 
the abelianization of $\check{G}$. 

Fix $x \in X(k)$ a $k$-point. This choice yields\footnote{Via 
the natural map $\check{\Lambda}_{Z_G^\circ} = 
(\Gr_{Z_G^\circ,x})^{\on{red}} \to \Bun_{Z_G^{\circ}}$ and the evident
action of $\Bun_{Z_G^{\circ}}$ on $\Bun_G$.}
an action of the lattice $\check{\Lambda}_{Z_G^\circ}$ on $\Bun_G$.
It also yields a map $\LocSys_{\check{G}} \to \bB\check{G}$ by
restriction of a local system to $x \in X$, and then by 
composition, a map $\LocSys_{\check{G}} \to \bB\check{G}^{\on{ab}}$.

These are compatible in the following sense. The action of
$\QCoh(\check{\Lambda}_{Z_G^\circ})$ (with its convolution monoidal structure) 
on $D(\Bun_G)$ is the same as the one obtained from:
\[
\QCoh(\check{\Lambda}_{Z_G^\circ}) \simeq 
\Rep(\check{G}^{\on{ab}}) = \QCoh(\bB \check{G}^{\on{ab}})
\to \QCoh(\LocSys_{\check{G}}) \actson D(\Bun_G).
\]

\noindent Indeed, this is an immediate consequence of the construction 
of the spectral action (via Satake) and a basic compatibility of geometric Satake.

Therefore, up to trivializing\footnote{This has the effect of simplifying the
notation below by removing certain twists involving this restriction. In other
words, this trivialization is innocuous and chosen out of laziness.}
the restriction of $\sigma$ at $x$, we obtain a commutative diagram:
\[
\begin{tikzcd}
D_{\sigma}(\Bun_G)
\arrow[rr]
\arrow[d]
&&
\Vect
\arrow[d]
\\
D_{\widehat{\sigma}}(\Bun_G/\check{\Lambda}_{Z_G^\circ}) = 
D_{\widehat{\sigma}}(\Bun_G)^{\check{\Lambda}_{Z_G^\circ}}
\arrow[rr]
\arrow[d]
&&
\QCoh(\LocSys_{\check{G}}^{\on{restr}} \underset{\bB \check{G}^{\on{ab}}}{\times} \Spec(k))
\arrow[d]
\\
D_{\widehat{\sigma}}(\Bun_G) 
\arrow[rr,"\coeff_{\widehat{\sigma}}^{\enh}"]
&&
\QCoh(\LocSys_{\check{G}}^{\on{restr}}).
\end{tikzcd}
\]

\noindent Here the left arrows are forgetful functors, and 
the right arrows are the evident pushforwards.

Observe that 
$\LocSys_{\check{G}}^{\on{restr}} \underset{\bB \check{G}^{\on{ab}}}{\times} \Spec(k)$
is Deligne-Mumford in an open neighborhood of the point $\sigma$. 
Therefore, the analysis of \S \ref{sss:ss} goes through when considering the
middle arrow, up to replacing $\pi_1^{\on{alg}}(G)$ with the finite
group $\pi_1^{\on{alg}}(G/Z_G^{\circ})$. 
Noting that for each $c \in \pi_1^{\on{alg}}(G) = \pi_0(\Bun_G)$, the 
map $\Bun_G^c \to \Bun_G/\check{\Lambda}_{Z_G^\circ}$ is the embedding
of a connected component, this clearly yields the result in the case
of general $G$ and very irreducible $\sigma$.

\begin{rem}\label{r:not-very-2}

Combining the above method with that of Remark \ref{r:not-very-1} yields
the perversity (i.e., local compactness) of $\sF_{\sigma}$ for 
general $G$ and irreducible $\sigma$; specifically, this shows that the restriction of
$\sF_{\sigma}$ to each connected component of $\Bun_G$ has finite length.
This finally resolves the leftover point from \S \ref{sss:ss}.

\end{rem}

\appendix

\section{Existence of Hecke eigensheaves via localization}\label{a:localization}

In this appendix, we prove Theorem \ref{t:aut-exists}, the existence
of Whittaker-normailzed Hecke eigensheaves (at least 
when $k$ is algebraically closed).
The construction of $\sF_{\sigma}$ from \emph{loc. cit}. 
is via the localization construction
of Beilinson-Drinfeld developed in \cite{hitchin}, using refinements of the
critical level Kac-Moody theory due to Frenkel-Gaitsgory and Arinkin's
existence of oper structures on irreducible local systems.

At times, we use mild extensions of existing results that are not 
well recorded in the
literature. In general, we point to \cite{jerusalem2014} and \cite{charles-lin}
for an introduction to the relevant circle of ideas.

We reiterate that we do not claim originality for the material here,
which we generally consider to be folklore consequences of the
work of others.

\subsection{Background on opers and localization}\label{ss:opers}

\subsubsection{Opers with singularities: local aspects}

Let $x \in X(k)$ be a marked point with coordinate $t$. 
We let $\sD_x = \Spec(k[[t]])$ be a disc and let $\lambda \in \Lambda^+$
be a dominant weight. Recall that there is a scheme
$\Op_{\check{G},x}^{\lambda}$ of \emph{opers with singularity $\lambda$} (at $x$).
These are defined in \cite{fg2} \S 2.9, where
they are denoted $\Op_{\check{\fg}}^{\lambda,\on{reg}}$. We remind that
these opers are $\check{G}$-local systems on the disc $\sD_x$ 
(i.e., points of $\bB \check{G}$) equipped with extra structure.

\begin{rem}\label{r:mf,red->reg}

We remind that any $k$-point of $\Op_{\check{G}}(\o{\sD})$, the indscheme of
opers on the punctured disc, underlying a local system on $\sD_x$ (rather
than $\o{\sD}_x$) lies in $\Op_{\check{G},x}^{\lambda}$ for some 
$\lambda$.\footnote{See for example the second equation of \cite{fg-spherical}
\S 2.2, where this is stated explicitly. We remark that, as in \emph{loc. cit}., 
the assertion holds for $k$ replaced by any reduced $k$-algebra, but not
for non-reduced algebras.} 

\end{rem}

\subsubsection{}

We use Kac-Moody notation as in \cite{localization}.
In particular, we let 
$\widehat{\fg}_{\crit}\mod^{G(O)}$ be the Kazhdan-Lusztig category at
critical level, which we also denote by $\KL_{\crit,x}$.
Let $\bV_{\crit}^{\lambda} \in \widehat{\fg}_{\crit}\mod^{G(O),\heartsuit}$ be
the Weyl module, i.e., the module $\ind_{\fg[[t]]}^{\widehat{\fg}_{\crit}}(V^{\lambda})$
where $V^{\lambda}$ is the $G$-representation corresponding
to $\lambda$, acted on by $\fg[[t]]$ via the evaluation homomorphism $\fg[[t]] \to \fg$. 

By \cite{fg-weyl} Theorem 1, there is a natural action of
$\Fun(\Op_{\check{G},x}^{\lambda})$ by endomorphisms on $\bV_{\crit}^{\lambda}$.
Therefore, we obtain a functor:
\begin{equation}\label{eq:weyl-x}
\QCoh(\Op_{\check{G},x}^{\lambda}) \to \widehat{\fg}_{\crit}\mod^{G(O)}
\end{equation}

\noindent sending the structure sheaf in the left hand side to the Weyl module
$\bV_{\crit}^{\lambda}$.

Moreover, by \cite{fg-fusion} Theorem 1.10,
the above is compatible with the $\Rep(\check{G})$ actions on both 
sides. Specifically, the map $\Op_{\check{G},x}^{\lambda} \to \bB \check{G}$
induces a symmetric monoidal functor $\Rep(\check{G}) \to 
\QCoh(\Op_{\check{G},x}^{\lambda})$, while $\Rep(\check{G})$ acts on the
right hand side via geometric Satake. For us, \cite{fg-fusion} Theorem 1.10
amounts to the assertion that the above functor is 
naturally $\Rep(\check{G})$-linear. We refer to \cite{localization} 
Corollary 7.10.1 for homotopical details regarding a similar situation.

\begin{notation}

Later, we will wish to explicitly note the dependence on the point $x$.
We will write $\bV_{\crit}^{\lambda\cdot x} \in \KL_{\crit,x}$
in this case.

\end{notation}

\subsubsection{}

Suppose $x \in X(k)$ is a marked point, which the reader should imagine
was implicitly equipped with the coordinate $t$ in the previous discussion.

There is a localization functor:
\[
\Loc_x:\widehat{\fg}_{\crit}\mod^{G(O)} \to D(\Bun_G)
\]

\noindent where we implicitly choose a square root of the canonical
bundle on $\Bun_G$, as in \cite{hitchin} \S 4, to identify $D_{\crit}(\Bun_G)$
with $D(\Bun_G)$.

This functor is equivariant for the action of the spherical Hecke category
at $x$ acting on both
sides, and in particular, $\Rep(\check{G})$-linear for Hecke functors (again: at $x$).

We recall that $\Loc_x(\bV_{\crit}^{\lambda})$ is the critically twisted 
$D$-module:
\[
\ind_{\crit}(\sV^{\lambda\cdot x}) \in D_{\crit}(\Bun_G) \simeq D(\Bun_G)
\] 

\noindent induced from induced from the vector bundle $\sV^{\lambda\cdot x}
\coloneqq \ev_x^*(V^{\lambda})$ where
where the map  $\ev_x:\Bun_G \to \bB G$ takes the fiber at $x$ of a $G$-bundle on $X$.

\subsubsection{}

By composition, we obtain a $\Rep(\check{G})$-linear functor:
\[
\EuScript{L}\text{oc}_x:\QCoh(\Op_{\check{G},x}^{\lambda}) \to D(\Bun_G)
\]

\noindent sending the structure sheaf to $\ind_{\crit}(\sV^{\lambda\cdot x})$.

\subsubsection{Ran space extension}

We now use a variant of the above with multiple points, working over Ran space.

Fix points $x_1,\ldots,x_n \in X$. We let $S \coloneqq \{x_1,\ldots,x_n\} \subset X$.
We also fix dominant weights $\lambda_1,\ldots,\lambda_n$.

Let $\Ran_{X,S}$ denote the marked Ran space, as in \cite{contractibility}. 
As in \cite{jerusalem2014} and \cite{charles-lin}, 
there is a natural category $\KL_{\crit,\Ran_X}$ over $\Ran_{X,\dR}$ with fiber 
$\widehat{\fg}_{\crit}\mod^{G(O)}$ at a marked point $x \in X(k)$.
We can pull it back to $\Ran_{X,S,\dR}$ to obtain a similar category
$\KL_{\crit,\Ran_{X,S}}$. 

The category $\Rep(\check{G})_{\Ran_X}$ acts by Hecke functors on 
$\KL_{\crit,\Ran_X}$. We can pull back to $\Ran_{X,S,\dR}$ to obtain 
an action of $\Rep(\check{G})_{\Ran_{X,S}}$ on $\KL_{\crit,\Ran_{X,S}}$.

We also have relative affine scheme 
$\Op_{\check{G},\Ran_{X,S}}^{\sum \lambda_i x_i} \to \Ran_{X,S,\dR}$ whose fibers
parametrize a finite set $\Sigma \subset X$ containing our marked points $S$,
plus a $\check{G}$-bundle with connection on the formal completion to $X$ at $\Sigma$,
which is equipped with an oper structure of type $\lambda_i$ at each $x_i \in S$
and type $0$ (i.e., regular) at each $y \in \Sigma \setminus S$.

\subsubsection{}

We let:
\[
\bV_{\crit}^{\sum \lambda_i\cdot x_i} \in \KL_{\crit,S} = \otimes_{i=1}^n \KL_{\crit,x_i}
\]

\noindent denote the object $\boxtimes_{i=1}^n \bV_{\crit}^{\lambda\cdot x_i}$.

By unital chiral algebra techniques, there is an induced object:
\[
\bV_{\crit,\on{unit}}^{\sum \lambda_i\cdot x_i} \in \KL_{\crit,\Ran_{X,S}}
\]

\noindent obtained by inserting the vacuum representation at points away from $S$.

Using standard (non-derived!) chiral algebra techniques, we readily 
obtain an extension of \cite{fg-weyl} that yields a ($D(\Ran_{X,S})$-linear) functor
(generalizing \eqref{eq:weyl-x}):
\begin{equation}\label{eq:weyl-ran}
\QCoh(\Op_{\check{G},\Ran_{X,S}}^{\sum \lambda_i x_i}) \to \KL_{\crit,\Ran_{X,S}}
\end{equation}

\noindent sending the structure sheaf on the left hand side to 
$\bV_{\crit,\on{unit}}^{\sum \lambda_i\cdot x_i}$. Moreover, this
functor is naturally $\Rep(\check{G})_{\Ran_{X,S}}$-linear.

\subsubsection{}

We can then compose with the $\Rep(\check{G})_{\Ran_{X,S}}$-linear localization
functor:
\[
\Loc_{\Ran_{X,S}}:\KL_{\crit,\Ran_{X,S}} \to D(\Bun_G \times \Ran_{X,S})
\]

\noindent to obtain a $\Rep(\check{G})_{\Ran_{X,S}}$-linear functor:
\[
\QCoh(\Op_{\check{G},\Ran_{X,S}}^{\sum \lambda_i x_i}) \to D(\Bun_G \times \Ran_{X,S}).
\]

\noindent Taking cohomology along $\Ran_{X,S}$ then yields 
a\footnote{Here and elsewhere, the subscript $(-)_{\on{indep}}$ is taken
to mean the \emph{independent category}, see \cite{extended-whit} for 
detailed discussion.} 
$\Rep(\check{G})_{\Ran_{X,S},\on{indep}} = \Ran(\check{G})_{\Ran_X,\on{indep}}$-linear functor:
\begin{equation}\label{eq:sloc}
\sL\text{oc}_{\Ran_{X,S}}:\QCoh(\Op_{\check{G},\Ran_{X,S}}^{\sum \lambda_i x_i})_{\on{indep}}
\to D(\Bun_G).
\end{equation}

\begin{rem}

We abuse notation in omitting the divisor $\sum \lambda_i x_i$ from the notation
$\sL\text{oc}$.

\end{rem}

\subsubsection{Globalization}

Now let $\Op_{\check{G}}^{\on{glob},\sum \lambda_i x_i}$ denote the
scheme of \emph{global} opers with singularity, i.e., $\check{G}$-local systems
on $X$ with $\check{B}$-reductions as a $\check{G}$-bundle satisfying the usual oper
condition away from $S = \{x_1,\ldots,x_n\}$ and satisfying the version with 
singularity $\lambda_i$ at $x_i$.

There is a natural symmetric monoidal functor:
\[
\Loc_{\Op_{\check{G}}}:
\QCoh(\Op_{\check{G},\Ran_{X,S}}^{\sum \lambda_i x_i})_{\on{indep}} \to 
\QCoh(\Op_{\check{G}}^{\on{glob},\sum \lambda_i x_i})
\]

\noindent admitting a fully faithful right adjoint. In particular,
$\Loc_{\Op_{\check{G}}}$ is a quotient functor.

It follows from the constructions that \eqref{eq:sloc} factors as:
\[
\begin{tikzcd}
\QCoh(\Op_{\check{G},\Ran_{X,S}}^{\sum \lambda_i x_i})_{\on{indep}}
\arrow[dr,"\sL\text{oc}_{\Ran_{X,S}}"] 
\arrow[d,swap,"\Loc_{\Op_{\check{G}}}"]
& \\
\QCoh(\Op_{\check{G}}^{\on{glob},\sum \lambda_i x_i})
\arrow[r,dotted,"\sL\text{oc}^{\on{glob}}"]
&
D(\Bun_G).
\end{tikzcd}
\]

The functor $\sL\text{oc}^{\on{glob}}$ is a priori 
$\Rep(\check{G})_{\Ran_{X,S},\on{indep}}$-linear; but as 
$\Rep(\check{G})_{\Ran_{X,S},\on{indep}}$ acts through its quotient 
$\QCoh(\LocSys_{\check{G}})$, this functor is actually 
$\QCoh(\LocSys_{\check{G}})$-linear.

\subsubsection{}

Next, we claim that there is a commutative diagram:
\begin{equation}\label{eq:coeff-opers}
\begin{tikzcd}
\QCoh(\Op_{\check{G}}^{\on{glob},\sum \lambda_i x_i})
\arrow[d,swap,"\sL\text{oc}^{\on{glob}}"]
\arrow[drrr,bend left = 10,"\pi_*"]
&
\\
D(\Bun_G)
\arrow[rrr,"{\coeff^{\enh}(-)}"]
&&&
\QCoh(\LocSys_{\check{G}})
\end{tikzcd}
\end{equation}

\noindent where $\pi$ is the natural map 
$\Op_{\check{G}}^{\on{glob},\sum \lambda_i x_i} \to \LocSys_{\check{G}}$.

Indeed, since these two functors 
$\QCoh(\Op_{\check{G}}^{\on{glob},\sum \lambda_i x_i}) \to 
\QCoh(\LocSys_{\check{G}})$ are $\QCoh(\LocSys_{\check{G}})$-linear,
it suffices to produce a commutative diagram:
\[
\begin{tikzcd}
\QCoh(\Op_{\check{G}}^{\on{glob},\sum \lambda_i x_i})
\arrow[d,swap,"\sL\text{oc}^{\on{glob}}"]
\arrow[drrrr, bend left = 14, pos = .6, "{\Gamma(\LocSys_{\check{G}},\pi_*(-))}"]
\\
D(\Bun_G)
\arrow[r,"{\coeff^{\enh}(-)}"]
&
\QCoh(\LocSys_{\check{G}})
\arrow[rrr,pos = .4,"{\Gamma(\LocSys_{\check{G}},\pi_*(-))}"]
&&&
\Vect.
\end{tikzcd}
\]

\noindent Equivalently, it suffices to produce a commutative diagram:
\[
\begin{tikzcd}
\QCoh(\Op_{\check{G}}^{\on{glob},\sum \lambda_i x_i})
\arrow[d,swap,"\sL\text{oc}^{\on{glob}}"]
\arrow[drr, bend left = 10, "{\Gamma(\Op_{\check{G}}^{\on{glob},\sum \lambda_i x_i},-)}"]
\\
D(\Bun_G)
\arrow[rr,"{\coeff}"]
&&
\Vect.
\end{tikzcd}
\]

By construction of $\sL\text{oc}^{\on{glob}}$ and 
\cite{fg-weyl} Theorem 2,
it suffices to construct a commutative diagram:
\[
\begin{tikzcd}
\KL_{\crit,\Ran_{X,S}}
\arrow[d,swap,"\Loc_{\Ran_{X,S}}"]
\arrow[r, "\Psi"]
&
D(\Ran_{X,S})
\arrow[d,"{C_{\dR}(\Ran_{X,S},-)}"]
\\
D(\Bun_G)
\arrow[r,"{\coeff}"]
&
\Vect.
\end{tikzcd}
\]

\noindent Here the functor $\Psi$ displayed above is the quantum Drinfeld-Sokolov
functor (at a point: this functor is BRST for $\fn((t))$ twisted by the standard 
character). 

The necessary commutative diagram now follows from 
\cite{dennis-opers} Corollary 6.4.4 and Proposition 7.3.2.\footnote{In fact,
Theorem 5.1.5 from \cite{dennis-opers} immediately yields a stronger
statement than we are using here. However, it references a certain
functor denoted in 
\cite{dennis-opers} by $\on{D-S}^{\on{KL}}$, and whose construction is not given
there. There is folklore knowledge about how to construct this functor,
so this point can be overcome; still, we prefer to circumvent it using 
the above simplification.}\footnote{Note that \cite{dennis-opers}
does not account for the shift $[-\dim \Bun_N^{\Omega}]$ in the 
definition of our functor $\coeff$. This shift is explained 
in \cite{charles-lin} Theorem 4.0.5 (see also \cite{charles-lin} Example 4.0.4
and the appearance of $\on{CT}_*^{\on{shifted}}$ in \S 2.4). }

\subsection{Proof of Theorem \ref{t:aut-exists}}\label{ss:pf-aut}

Let us return to our fixed, irreducible local system 
$\sigma \in \LocSys_{\check{G}}(k)$. According to \cite{dima-opers} 
Theorem A (and more precisely, its Corollary 1.1), there
exists an open $U \subset X$ such that $\sigma|_U$ admits a
(genuine, without defect) oper structure. By Remark \ref{r:mf,red->reg},
this means that there exists a $\check{\Lambda}^+$-valued 
divisor $\sum \check{\lambda}_i x_i$ on $X$ such that $\sigma$
lifts to a $k$-point 
$\chi_{\sigma} \in \Op_{\check{G}}^{\on{glob},\sum \lambda_i x_i}$.

Finally, take: 
\[
\sF_{\sigma} \coloneqq \sL\text{oc}^{\on{glob}}(k_{\chi_{\sigma}})[-\dim \Bun_G]
\]

\noindent for $k_{\chi_{\sigma}} \in \QCoh(\Op_{\check{G}}^{\on{glob},\sum \lambda_i x_i})$
the skyscraper sheaf at the point $\chi_{\sigma}$.

That $\sF_{\sigma}$ is a Hecke eigensheaf follows from 
$\QCoh(\LocSys_{\check{G}})$-linearity of $\sL\text{oc}^{\on{glob}}$.
The Whittaker normalization: 
\[
\coeff^{\enh}(\sF_{\sigma}) \simeq k_{\sigma}[-\dim \Bun_G]
\]

\noindent follows from \eqref{eq:coeff-opers}.

\bibliography{bibtex.bib}{}
\bibliographystyle{alphanum}

\end{document}